\documentclass{article}

\usepackage{arxiv}
\usepackage{amsmath,amsthm}
\usepackage[utf8]{inputenc} 
\usepackage[T1]{fontenc}    
     
\usepackage{url}            
\usepackage{booktabs}       
\usepackage{amsfonts}       
\usepackage{nicefrac}       
\usepackage{microtype}      
\usepackage{lipsum}
\usepackage{graphicx}
\usepackage{subfigure}
\usepackage{epstopdf}
\usepackage[utf8]{inputenc} 
\usepackage{amsmath}
\usepackage{amssymb}
\usepackage{booktabs}
\usepackage{mathrsfs}
 \usepackage{multirow}
 
 \newtheorem{theorem}{Theorem}[section]
\newtheorem{lemma}[theorem]{Lemma}
\newtheorem{proposition}[theorem]{Proposition}

\title{Pattern formation study of an eco-epidemiological model with cannibalism and disease in predator population}

\author{
Nitu Kumari \\
  School of Basic Sciences\\ 
  Indian Institute of Technology Mandi\\
  175005, India\\
  \texttt{nitu@iitmandi.ac.in}
  \And
 Vikas Kumar \\
  School of Basic Sciences\\ 
  Indian Institute of Technology Mandi\\
  175005, India\\
  \texttt{vikaskharwar51@gmail.com}
    \And
 Ravi P. Agarwal\\
 Department of Mathematics, Texas A \& M University - Kingsville,\\
  \texttt{Ravi.Agarwal@tamuk.edu} 
 }

\begin{document}
\maketitle

\begin{abstract}
Pattern formation analysis of eco-epidemiological models with cannibalism and disease has been less explored in the literature. Therefore, motivated by this, we have proposed a diffusive eco-epidemiological model and performed pattern formation analysis in the model system. Sufficient conditions for local asymptotic stability and global asymptotic stability for the constant positive steady state are obtained by linearization and Lyapunov function technique.  A priori estimate for the positive steady state is obtained for the nonexistence of the nonconstant positive solution using Cauchy and Poincar\'e inequality. The existence of the nonconstant positive steady states is studied using Leray-Schauder degree theory. The importance of the diffusive coefficients which are responsible for the appearance of stationary patterns is observed. Pattern formation is done using numerical simulation. Further, the effect of the cannibalism and disease are observed on the dynamics of the proposed model system. The movements of prey and susceptible predator plays a significant role in pattern formation. These movements cause stationary and non-stationary patterns. It is observed that an increment in the movement of the susceptible predator as well as cannibalistic attack rate converts non-Turing patterns to Turing patterns. Lyapunov spectrum is calculated for quantification of stable and unstable dynamics. Non-Turing patterns obtained with parameter set having unstable limit cycle are more interesting and realistic than stationary patterns. Stationary and non-stationary non-Turing patterns are obtained.
\end{abstract}


\section{Introduction}\label{intro}
 Cannibalism is a phenomenon which takes place under scarce resources and high population densities. The process involves killing and consumption of full or a part of an individual of the same species (conspecifics). The classical work of Polis \cite{polis1981evolution} cites more than 1300 species amongst which bacteria, protozoans, invertebrates, and vertebrates, including humans, practised cannibalism. Cannibalism as a phenomenon has been found mostly in structured population models because individual organisms differ in size and have different life cycle stages including birth, growth, development and death. These structured populations can be modeled using ordinary differential equations, partial differential equations, delay differential equations and integro-differential equations  \cite{metz2014dynamics,cushing1992size,m1925applications,foerster1959some,frauenthal1983some,peng1997state}. However, these population models are difficult to analyze due to their mathematical and dynamical complexity. Cannibalism has also been observed in unstructured population models where all individuals are subject to the same general ecological pressures. That is, the rate of growth, reproduction, and mortality is roughly the same for all individuals in the population. Kohlmeier and Ebenhoh \cite{kohlmeier1995stabilizing} were first to propose an unstructured ODE model after introducing cannibalism in the classical Rosenzweig-McArthur model \cite{rosenzweig1963graphical} and found that cannibalism has a strong stabilizing effect on the population. Chakraborty and  Chattopadhyay  \cite{chakraborty2011effect} proposed a prey-predator model by incorporating the nutritional values in Kohlmeier and Ebenhoh model \cite{kohlmeier1995stabilizing} and stated cannibalism as one of the possible candidates for resolving the paradox of enrichment. 

Predator cannibalism is abundant in the natural ecosystem. Research work of Basheer et al. \cite{basheer2016prey,al2018exploring} showed significant effects of prey cannibalism and predator cannibalism and their joint effect on the dynamics of the model system. They observed that prey cannibalism could alter the spatial dynamics of the model and also can lead to the pattern forming Turing dynamics. The joint effect of prey cannibalism and predator cannibalism stabilize the unstable equilibrium and can also form spatial patterns. Cannibalism has also been studied in epidemic and eco-epidemiological models. Research work of Biswas et al. \cite{biswas2015cannibalism,biswas2015model,biswas2015cannibalistic,biswas2018cannibalistic} showed the significant effect of cannibalism in prey-predator systems with the disease in predator population and also observed the impact of cannibalism with the disease in both predator and prey population. They observed that cannibalism could control the disease in the model system. The above studies show that cannibalism has a strong stabilizing effect and also led to the pattern forming dynamics. Therefore motivated by this, for a more realistic scenario, we extend the cannibalistic eco-epidemiological model proposed by Biswas et al. \cite{biswas2018cannibalistic}. The effect of cannibalism in diffusive systems has been less explored in the literature. Therefore we have performed pattern formation analysis to understand the role of cannibalism and disease in the population fluctuation.

Diffusion plays an important role in population biology and gives rise to pattern formation. The mathematical analysis and simulation of diffusive models provide an idea of how the species are distributed in two-dimensional space (landscape or aquatic environment). In reality, the population distributed into localized patches can be fixed, or it can change with time. The fixed or time-invariant patches correspond to stationary patterns, while the patches changing with time corresponds to non-stationary patterns. Non-stationary patterns can be oscillatory, quasi-periodic or chaotic \cite{banerjee2015turing}. Starting from the Turing seminal work \cite{turing1990chemical}, Brusselator model \cite{brown1995global}, Gierer-Meinhardt model \cite{iron2001stability,wei2002spikes}, Sel'kov model \cite{davidson2000priori,wang2003non}, Lotka-Volterra predator-prey model \cite{ni1998diffusion,du2001qualitative} and the references therein, diffusion (self-diffusion) has been observed as one of the causes for stationary patterns. The per capita diffusion rate of each species, influenced only by its population density is called self-diffusion. These patterns arise due to diffusion-driven instability also known as Turing instability studied by many authors \cite{medvinsky2002spatiotemporal,upadhyay2008wave,sun2012pattern,banerjee2011self,kumari2013pattern,sun2016mathematical,ghorai2020dispersal}. 

Sun et al. \cite{sun2009predator} first explored the relation between pattern formation and cannibalism in the Rosenzweig-McArthur diffusive model. They showed that spatial patterns are greatly influenced by cannibalism and can not emerge without it. Fasani and Rinaldi \cite{fasani2012remarks} also studied the more general form of the model of Sun et al. and concluded that spatial patterns arise for the highly cannibalistic and dispersive predator. Spatial patterns not only emerge with predator cannibalism but can also emerge with prey cannibalism.  Basheer et al. \cite{basheer2016prey} showed that prey cannibalism could also lead to pattern-forming instabilities. In our proposed diffusive eco-epidemiology system (\ref{eq1}), we studied pattern formation to understand how cannibalism and disease influences the distribution of populations and the structure of communities. 

 Qualitative analysis of the diffusive system helps us to understand how dispersal and spatial effects influence populations and communities. Pang and Wang \cite{pang2003qualitative} studied the qualitative properties of a diffusive ratio-dependent prey-predator model system. They established the existence of nonconstant positive steady states and showed stationary patterns as a result of diffusion. The assumption of homogeneous distribution of species in space yields constant positive solutions for the temporal model system (ODE). In spatially inhomogeneous case, the existence of nonconstant time-independent positive solutions, also called as stationary patterns, indicate the rich dynamics of the corresponding spatiotemporal model system (PDE) \cite{peng2007stationary}. The study of nonconstant steady state solutions is based on the analysis of steady state system (see (\ref{nonhomo})). There are two methods to establish the existence of such non-trivial positive solutions, one is singular perturbation \cite{kan1998singular}, and the other is bifurcation technique \cite{peng2007stationary}. Leray-Schauder degree theory \cite{nirenberg1974topics} is a powerful tool for establishing the existence of nonconstant positive steady states. Many authors have established the existence of nonconstant positive steady states with diffusion \cite{peng2007stationary,gakkhar2011non,zha2015non,gao2018positive,ko2018pattern,tiwari2019qualitative}.
  
  Existence and nonexistence results of nonconstant positive steady state ensure the appearance of stationary patterns. Peng and Wang \cite{peng2005positive} studied the existence and nonexistence of nonconstant positive steady states in a diffusive Holling-Tanner model. Further, in \cite{peng2007global} they have studied the local and global stability of constant positive steady state for the same model. The local and global stability helps us to understand that when species will be homogeneously distributed spatially. From the local stability of constant positive steady state we mean a diffusive system has no nonconstant positive steady states in the neighbourhood of the constant positive solution \cite{wang2004stationary}. This means no stationary patterns arise, and this result is equivalent to diffusion-driven instability or Turing instability. The global stability of constant positive steady state ensures that the species will be homogeneously distributed as time goes to infinity no matter how they diffuse \cite{pang2004strategy}.
  
Some authors have established the existence and nonexistence of nonconstant steady states in diffusive cannibalistic population models. Wang \cite{biao2017positive} studied the model of Sun et al. \cite{sun2009predator} and derived the conditions for existence and nonexistence of nonconstant positive steady states using topological degree theory. Zhang et al. \cite{zhang2019diffusive} proposed and analyzed a diffusive prey-predator system by incorporating prey refuge in Chakraborty and Chattopadhyay model \cite{chakraborty2011effect}. They have investigated global existence, dissipation and persistence of the model system and concluded that the predator cannibalism could stabilize the system and prevent the paradox of enrichment. In this present work, we have observed the conditions for the local and global stability of the constant positive steady state of the model system (\ref{eq1}) as done in \cite{peng2007global}. Next, we have analyzed the corresponding steady state system (\ref{nonhomo}), to obtain the condition for the existence and nonexistence of nonconstant positive steady states as reported in \cite{peng2005positive}. Through numerical simulations, we observe the effect of cannibalism and disease on pattern formation and system dynamics. To this end, we have made the following contributions:

\begin{itemize}
\item To understand the effect of cannibalism and disease on pattern formation, a diffusive eco-epidemiological model system is formulated and analyzed. 
\item To know that how the species are homogeneously distributed, local and global stability of the constant positive steady state of the diffusive system has been investigated.  
\item Existence and nonexistence of nonconstant steady states are discussed using Leray-Schauder degree theory, Cauchy inequality and Poincar\'e inequality.  
\item The condition for Turing instability is discussed, and Turing patterns are obtained. Further stationary and non-stationary non-Turing patterns are also obtained for the proposed system. 
\item The effect of cannibalism and the disease on the dynamics of the proposed model system has been studied. 
\end{itemize}

\noindent The organization of the paper is as follows. In Section \ref{formulation}, spatiotemporal model system has been proposed. In Section \ref{locnglob}, we have proved the local and global stability of the constant positive steady state of model proposed. Linearization technique has been used for local stability and Lyapunov function method for global stability. In Section \ref{nonexist}, we have established a priori upper and lower bounds for nonconstant positive solution of the model. Further, we studied the nonexistence and global existence of the nonconstant positive steady state of spatiotemporal model system. Turing instability is studied in Section \ref{tstability}. In Section \ref{simulation}, we have perfomed the numerical simulation to illustrate the effect of disease transmission rate and cannibalistic attack rate on system dynamics. Further, we have obtained Turing, stationary and non-stationary non-Turing patterns by solving the spatiotemporal model system. In the last section, a detailed discussion of results obtained is presented.
\section{Model Formulation} \label{formulation}
 Here we first consider the temporal model system proposed by Biswas et al. \cite{biswas2018cannibalistic}. It assumes that disease spreads into predator species only. Therefore, the predator species is divided into susceptible and infected populations. Since the predators have cannibalistic behavior, the disease spreads among them via cannibalism. Biswas et al. \cite{biswas2018cannibalistic} proposed the following model system:
 \begin{equation}\label{main}
\begin{cases}\vspace{0.3cm}
\dfrac{d u}{d t}   =ru\left(1-\dfrac{u}{k} \right) - \dfrac{(\alpha_1 v + \alpha_2 w)u}{\gamma + u} \triangleq G_1 (u,v,w), \\
\dfrac{d v}{d t}    = \dfrac{\alpha(\alpha_1 v + \alpha_2 w)u}{\gamma + u} + c_1 \sigma (\beta v +w)v  + c_2 \sigma (\beta v + w)w\\
 ~~~~~~~~~ - \sigma (\beta v + w )v - \sigma lfvw -\lambda v w -dv\triangleq G_2 (u,v,w), \\
\dfrac{d w}{d t}   =\lambda vw + \sigma lfvw-\sigma(\beta v + w)w -(d+e)w  \triangleq G_3 (u,v,w)
\end{cases}
\end{equation}
with positive initial conditions
\begin{equation}
u(0)>0,v(0)>0,w(0)>0.~~~~~~~~~~~~~~~~~~~~~~~~~~~~~~~~~~~~~~~~~
\end{equation}
$u(t)$ is prey population density, $v(t)$ is susceptible predator population density and $w(t)$ is infected predator population density. Predator population follows Holling type II functional response with a half-saturation constant $\gamma$. The biological meaning of the other parameters is given in Table \ref{parametervalues}.
 
 $r$ is intrinsic growth rate, and $k$ is carrying capacity of prey. In the absence of predation, prey population grows according to the logistic law. The susceptible predator becomes infected due to the effect of cannibalism and spreads disease in the predator population, by following a simple law of mass action as the spread is not inherent. For the additional properties of equilibria of the temporal model (\ref{main}), we refer to \cite{biswas2018cannibalistic}. 

In \cite{biswas2018cannibalistic}, it is observed that the temporal system (\ref{main}) has a constant positive solution or steady state ${\textbf{u}}^{*} = (u^{*},v^{*},w^{*})^T$, where $w^{*} =  -\frac{1}{\sigma}((d+e) - (\lambda + \sigma l f - \sigma \beta)v^*)$, $u^{*}$ and $v^{*}$ are obtained by solving this system of equations
\begin{equation}\label{eq2}
\begin{array}{ccl}
&&m_1 u^2 + m_2 u + m_3 v + m_4 =0,\\
&&n_1 v^2 + n_2 \dfrac{uv}{\gamma +u} + n_3 v+n_4 \dfrac{u}{\gamma + u }+ n_5 = 0,\\
\end{array}
\end{equation}
where, 
$m_1 = \frac{r}{k}, m_2 = r\left(\frac{\gamma}{k}-1 \right), m_3 = \alpha_1 +\alpha_2 \left( lf-\beta + \frac{\lambda}{\sigma}\right), m_4 = -\{r \gamma + \frac{\alpha_2}{\sigma}(d+e)\},n_1 = c_2(\lambda + \sigma l  f - \sigma \beta )^2 + (c_1\sigma - \sigma -\sigma l f - \lambda)(\lambda + \sigma l f - \sigma \beta ) + (c_1 \sigma \beta - \sigma \beta ) \sigma , n_2 = \alpha\{\alpha_1 \sigma + \alpha_2 (\lambda + \sigma l f - \sigma \beta )\}, n_3 = -[(c_1 \sigma - \sigma -\sigma l f - \lambda)(d+e)+2c_2 (d+e) (\lambda + \sigma l f  - \sigma \beta)+ \sigma d],n_4 = -\alpha \alpha_2(d+e),~~\text{and}~~ n_5 = c_2(d+e)^2.$\\
  The constant positive steady state ${\textbf{u}}^{*}$ exists if and only if
 \begin{equation}\label{existenceequi}
{(\lambda + \sigma l f - \sigma \beta)} v^{*} > {(d+e)}.
 \end{equation}
 
 \begin{table*}[h!]
\caption{List of parameters}
\centering
\begin{tabular}{@{}cl@{}}
\toprule
\multicolumn{1}{l}{Parameter} & Biological meaning                                                  \\ \midrule
$r$                                    & Intrinsic growth rate of prey                            \\
$k$                                    & Carrying capacity of prey                                           \\
$\lambda$                              & Disease transmission rate                                           \\
 
$\alpha_1$                             & Predation rate of susceptible predator                              \\

$\gamma$                               & Half-saturation constant                                            \\
$\alpha$                               & Conversion efficiency of predator                                   \\
$\alpha_2$                             & Predation rate of infected predator                                 \\
$d$                                    & Natural death rate of predator                                      \\
$e$                                    & Additional disease-related mortality rate                           \\
$\sigma$                               & Attack rate due to cannibalism                                      \\
$c_1$                                  & Conversion rate of susceptible predator for cannibalism\\
$\beta$                                & Dimensionless quantity                             \\
$l$                                    & Probability of transmission from a cannibalistic interaction           \\
$c_2$                                  & Conversion rate of infected predator for cannibalism\\
$f$                                    & Number of predators sharing one conspecific predator                 \\ \bottomrule
\end{tabular}
\label{parametervalues}
\end{table*}


 In this present work, we introduce spatial variations to the model system (\ref{main}). For this purpose, we assume that the species (prey and predators) perform active movements in both $x$ and $y$ directions (in two-dimensional space). Movement is possible due to various reasons such as: searching food, social interactions and finding mates. Therefore we introduce diffusion terms in our model system (\ref{main}) and assume that the prey, mid-predator and top predator species diffuse with the diffusive rates $d_1,d_2$ and $d_3$ respectively. If we assume that the domain $X=(x,y) \in \Omega \subset {\mathbb{R}}^{2}$ is fixed and bounded, then reaction-diffusion model system can be obtained from the model (\ref{main}). The spatially extended model takes the form as:

\begin{equation}\label{eq1}
\begin{cases}\vspace{0.3cm}
\dfrac{\partial u}{\partial t} =ru\left(1-\dfrac{u}{k} \right) - \dfrac{(\alpha_1 v + \alpha_2 w)u}{\gamma + u} + d_1 \Delta u,  \\
\dfrac{\partial v}{\partial t}    = \dfrac{\alpha(\alpha_1 v + \alpha_2 w)u}{\gamma + u} + c_1 \sigma (\beta v +w)v + c_2 \sigma (\beta v + w)w \\
 ~~~~~~~~~- \sigma (\beta v + w )v - \sigma lfvw -\lambda v w -dv + d_2 \Delta v,  ~~~~~~~~ \\
\dfrac{\partial w}{\partial t}   =\lambda vw + \sigma lfvw-\sigma(\beta v + w)w -(d+e)w + d_3 \Delta w.\vspace{0.2cm}
\end{cases}
\end{equation}
The boundary and initial conditions are given by
\begin{equation}
\begin{cases}\vspace{0.2cm}
\dfrac{\partial u }{\partial \nu} =\dfrac{\partial v }{\partial \nu}=\dfrac{\partial w }{\partial \nu} = 0~~~~~~~~~~~~~~~~~~~~~~~~~~~~~~~~~~~~~~~~~~~~~~~~~~~~x \in \partial \Omega,\\
u(x,0)>0,v(x,0)>0,w(x,0)>0~~~~~~~~~~~~~~~~~~~~~~~~~~~~~~~~~x \in\Omega.
\end{cases}
\end{equation}
 Here, $\nu$ is unit outward normal vector of the boundary $\partial \Omega$, $\Delta$ denotes the Laplacian operator as $(\partial^2 /\partial x^2 + \partial^2 / \partial y^2)$ in two-dimensional spatial domain. The initial conditions being continuous functions and the boundary conditions assure that there is no external factor from outside.

 For simplification, taking $\textbf{u}=(u,v,w)^T$ and $\textbf{G}(\textbf{u}) = (G_1(\textbf{u}),G_2(\textbf{u}),G_3(\textbf{u}))^T$, then problem (\ref{eq1}) can be written in this form 
\begin{equation}
\dfrac{\partial \textbf{u} }{\partial t} = \mathcal{D} \Delta \textbf{u} + \textbf{G}(\textbf{u}),
\end{equation}
where $\mathcal{D}=\text{Diag}(d_1,d_2,d_3)$ is a diffusion coefficient matrix.
 
\section{Stability of constant positive steady state}\label{locnglob}
 In this section, we shall discuss local and global stability of constant positive steady state ${\textbf{u}}^{*} = (u^{*},v^{*},w^{*})^T$ of the spatiotemporal system (\ref{eq1}). But before that, we will set up these useful notations.
 
Let $0 = \mu_0< \mu_1 < \mu_2 <\mu_3 < \cdots$ be the eigenvalues of the operator $-\Delta$ on $\Omega$ with the homogeneous Neumann boundary condition, and $\mathbf{E}(\mu_i)$ be the eigenspace corresponding to $\mu_i$ in $C^{1}(\bar{\Omega})$. Let
\[\mathbf{X} = \{\textbf{u}\in [C^{1}(\bar{\Omega})]^3)~|~\partial_{\nu} \textbf{u} = 0~~\text{on}~~\partial \Omega\},\]
$\{\phi_{ij},j=1,...,\dim \mathbf{E}(\mu_i)\}$ be an orthonormal basis of $\mathbf{E}(\mu_i)$, and $\mathbf{X}_{ij} = \{\textbf{c} \phi_{ij} | \textbf{c} \in {\mathbb{R}}^3\}$. Then,
\[{\mathbf{X}} = \bigoplus_{i=1}^{\infty} \mathbf{X}_i ~~~\text{and}~~~\mathbf{X}_i = \bigoplus_{j=1}^{\dim \mathbf{E}(\mu_i)} \mathbf{X}_{ij}.\]
\subsection{Local stability of ${\textbf{u}}^{*}$}
If the constant positive steady state ${\textbf{u}}^{*}$ is locally asymptotically stable then the spatiotemporal system (\ref{eq1}) has no nonconstant positive steady states in the neighbourhood of ${\textbf{u}}^{*}$, that is, no stationary patterns arise if (\ref{existenceequi}) and (\ref{locondition}) hold. We obtain the results in the next theorem.
\begin{theorem}\label{localtheorem}
If the conditions 
\begin{equation}\label{locondition}
\begin{array}{l}
M_1 > D_1, M_2 > E_1, \frac{r}{k} > \frac{A}{B^2}  \\
u^{*} > \max \bigg\{ \frac{1}{w^{*}},\frac{\sigma}{(M_1 - D_1) + \sigma w^{*}}, \frac{-(M_1 -D_1) + \sqrt{(M_1 - D_1)^2 + 4 \left( \frac{r}{k} - \frac{A}{B^2} \right) \sigma}}{2\left( \frac{r}{k} - \frac{A}{B^2} \right)}\bigg\}
\end{array}
\end{equation}
where, $A,B,C,M_1,M_2,E_1,D_1$ and $P_1$ are given in (\ref{notations}),

 hold, then the constant positive solution $\textbf{u}^{*}$ of the spatiotemporal system (\ref{eq1}) is locally asymptotically stable.
\end{theorem}
\begin{proof}
Linearization of (\ref{eq1}) at constant positive solution $\textbf{u}^{*}$ is 
\begin{equation}
\textbf{U}_t =   \mathcal{L}\textbf{ U},
\end{equation}
where 
$\textbf{U}=(u_1,v_1,w_1)^{T},$ and $\mathcal{L} = \mathcal{D} \Delta + G_{\textbf{u}}(\textbf{u}^{*})$. The jacobian matrix $G_{\textbf{u}}(\textbf{u}^{*})$ is given by
\begin{equation}\label{jacobian}
G_{\textbf{u}}(\textbf{u}^{*})= \begin{pmatrix}
a_{11} & a_{12}&a_{13}\\
a_{21}& a_{22} &a_{23} \\
a_{31} & a_{32} & a_{33}
\end{pmatrix}
\end{equation}
 where,
\begin{equation*}
\begin{array}{l}\vspace{0.2cm}
 a_{11}=-\dfrac{ru^*}{k} + \dfrac{(\alpha_1 v^* + \alpha_2 w^*)u^*}{(\gamma +u^*)^2},~a_{12}=-\dfrac{\alpha_1 u^*}{\gamma + u^*},~a_{13}=-\dfrac{\alpha_2 u^*}{\gamma + u^*}, \\\vspace{0.2cm}
 a_{21}=\dfrac{\alpha(\alpha_1 v^* + \alpha_2 w^*)}{\gamma +u^*} - \dfrac{\alpha (\alpha_1 v^* + \alpha_2 w^*)u^*}{(\gamma +u^*)^2},\\\vspace{0.2cm}
 a_{22}=\dfrac{\alpha \alpha_1 u^*}{\gamma + u^*} + c_1\sigma (2 \beta v^* + w^*) + c_2 \sigma \beta w^* -\sigma (2\beta v^* + w^*)-\sigma l f w^* - \lambda w^* -d ,\\\vspace{0.2cm}
 a_{23}=\dfrac{\alpha \alpha_2 u^*}{\gamma + u^*} + c_1 \sigma v^* + c_2 \sigma (\beta v^* + 2w^*)-\sigma v^* - \sigma l f v^* - \lambda v^*, \\\vspace{0.2cm}
 a_{31}=0,~ a_{32}= \lambda w^* + \sigma l f w^* - \sigma \beta w^*,~ a_{33}= -\sigma w^*.
 \end{array}
\end{equation*}

  For each $i \geq 0$, $\mathbf{X}_i$ is invariant under the operator $\mathcal{L}$, and $\eta$  is an eigenvalue of $\mathcal{L}$ on $\mathbf{X}_i$, if and only if it is an eigenvalue of the matrix $-\mu_i \mathcal{D} + G_\textbf{u}(\textbf{u}^{*})$.

The characteristic polynomial of the matrix $G_{\textbf{u}}({\textbf{u}}^{*})$ can be written as
\begin{equation}
\psi(\xi) = \xi^3 + \mathcal{A}_1~ \xi^2 + \mathcal{A}_2~ \xi + \mathcal{A}_3, 
\end{equation}
 
where,
\begin{eqnarray}\label{3.5}
\mathcal{A}_1 &=& (M_1-D_1) + \left( \dfrac{r}{k} - \dfrac{A}{B^2} \right) u^{*},\\
 \mathcal{A}_2 &=& u^{*}\left( \dfrac{r}{k} - \dfrac{A}{B^2} \right) (M_1 - D_1 + \sigma w^{*})  + \sigma w^{*}(M_1 - D_1) \\\nonumber
              &&+ P_1w^{*}(M_2-E_1) + \dfrac{\alpha_1 u^{*}C}{B}      , \\ \label{3.7}\mathcal{A}_3 & =& \left( \dfrac{r}{k} - \dfrac{A}{B^2} \right) \left\{ \sigma (M_1-D_1) + P_1(M_2-E_1)\right\} + \dfrac{\sigma \alpha_1 C}{B},
\end{eqnarray}
 
\begin{eqnarray}\nonumber
\mathcal{A}_1 \mathcal{A}_2 - \mathcal{A}_3 &=& \left\{(M_1-D_1) + \left( \dfrac{r}{k} - \dfrac{A}{B^2} \right) u^{*} \right\} \times \bigg\{u^{*}\left( \dfrac{r}{k} - \dfrac{A}{B^2} \right) (M_1 - D_1 + \sigma w^{*}) \\\nonumber
&& + \sigma w^{*}(M_1 - D_1) + P_1w^{*}(M_2-E_1) + \dfrac{\alpha_1 u^{*}C}{B} \bigg\} - \bigg\{ \left( \dfrac{r}{k} - \dfrac{A}{B^2} \right)\\ 
&& \times \left\{ \sigma (M_1-D_1) + P_1(M_2-E_1)\right\} + \dfrac{\sigma \alpha_1 C}{B} \bigg\}.
\end{eqnarray}
 Where,
\begin{equation}\label{notations}
\begin{array}{lcl}
 A=\alpha_1 v^{*} + \alpha_2 w^{*}, B= \gamma + u^{*}, C=\frac{\alpha(\alpha_1 v^* + \alpha_2 w^*)}{\gamma +u^*} - \frac{\alpha (\alpha_1 v^* + \alpha_2 w^*)u^*}{(\gamma +u^*)^2},\\D_1 =\frac{\alpha \alpha_1 u^*}{\gamma + u^*} + c_1\sigma (2 \beta v^* + w^*)+ c_2 \sigma \beta w^*, M_1 = \sigma (2\beta v^* + w^*)\\+\sigma l f w^* + \lambda w^* +d, E_1 = \frac{\alpha \alpha_2 u^*}{\gamma + u^*} + c_1 \sigma v^* + c_2 \sigma (\beta v^* + 2w^*),\\ M_2 = \sigma v^* + \sigma l f v^* + \lambda v^* ~~\text{and}~~P_1=\lambda + \sigma l f - \sigma \beta.
 \end{array}
\end{equation}
By using the conditions stated in the Theorem \ref{localtheorem}, it is easy to verify that 
\begin{equation}\label{local1}
\mathcal{A}_1 > 0,\mathcal{A}_2 > 0,\mathcal{A}_3 > 0~~\text{and}~~ \mathcal{A}_1 \mathcal{A}_2 - \mathcal{A}_3 > 0.
\end{equation}
The characteristic polynomial of  $-\mu_i \mathcal{D} + G_\textbf{u}(\textbf{u}^{*})$ is  
 \begin{equation}
 \phi_i(\eta) = \eta^3 + \mathcal{B}_{1i}~ \eta^2 + \mathcal{B}_{2i} ~\eta + \mathcal{B}_{3i},
 \end{equation}
 with
 \begin{eqnarray}\nonumber
  \mathcal{B}_{1i} &=& \mu_i (d_1 + d_2 + d_3 ) + \mathcal{A}_1 ,\\\nonumber
  \mathcal{B}_{2i} &=& \mu^2_i (d_1 d_2 + d_2 d_3 + d_1 d_3) - \mu_i \{ d_1 (a_{22} + a_{33}) \\\label{local2}
 && + d_2 (a_{11} + a_{33})  + d_3 (a_{11} + a_{22}) \} + \mathcal{A}_2,\\\nonumber
  \mathcal{B}_{3i} &=& \mu^3_i d_1 d_2 d_3 - \mu^2_i\{d_1 d_2 a_{33} + d_2 d_3 a_{11} + d_1 d_3 a_{22} \} \\\nonumber
  &&+ \mu_i\{d_1(a_{22}a_{33} - a_{23}a_{32}) + d_2 a_{11}a_{33} + d_3 (a_{11}a_{22} - a_{12}a_{21})\}+\mathcal{A}_3,\nonumber
 \end{eqnarray}
 where $a_{ij}$ and $\mathcal{A}_i$ are given in (\ref{jacobian}) and (\ref{3.5})-(\ref{3.7}) respectively. From (\ref{local1}) and (\ref{local2}), it follows that $\mathcal{B}_{1i},\mathcal{B}_{2i},\mathcal{B}_{3i}>0$. A series of calculations yield
 \[\mathcal{B}_{1i}\mathcal{B}_{2i}-\mathcal{B}_{3i} = \mathcal{M}_1 \mu^3_i+ \mathcal{M}_{2} \mu^2_i + \mathcal{M}_3\mu_i + \mathcal{A}_1 \mathcal{A}_2 -\mathcal{A}_3, \]
 in which:
 \begin{eqnarray*}
 \mathcal{M}_1 &=& (d_1 d_2 + d_2 d_3 + d_1 d_3)(d_1 +  d_2 +  d_3) - d_1 d_2 d_3,\\
 \mathcal{M}_2 &=& -\{(a_{11} + a_{22})[  d_3 (d_1 + d_2 +d_3)+d_1 d_2 ] + (a_{22} + a_{33})[ d_2 (d_1 + d_2 +d_3) + d_2 d_3 ] \\
 && + (a_{11} + a_{33})[ d_2 (d_1 + d_2 +d_3) + d_1 d_3 ]\},\\
 \mathcal{M}_3 &=& d_1[\mathcal{A}_2 - \mathcal{A}_1(a_{22} + a_{33}) - (a_{22}a_{33} - a_{23}a_{32})] + d_2 [\mathcal{A}_2 - \mathcal{A}_1(a_{11} + a_{33}) - a_{11}a_{33}]\\
 &&+ d_3[\mathcal{A}_2 - \mathcal{A}_1(a_{11} + a_{22}) - (a_{11}a_{22} - a_{12}a_{21}).
 \end{eqnarray*}

 From the above we can conclude that $\mathcal{B}_{1i}\mathcal{B}_{2i}-\mathcal{B}_{3i} >0,$ for all $i \geq 0$. Therefore by using Routh-Hurwitz criterion, for each $i \geq 0$, the three roots $\eta_{i,1},\eta_{i,2},\eta_{i,3}$ of $\phi_i (\eta)=0$ and all have negative real parts. 
 
Finally, Theorem 5.1.1 of Dan Henry \cite{henry2006geometric} concludes the results.
\end{proof} 
\subsection{Global stability of $\textbf{u} ^{*}$}
If the constant positive steady state $\textbf{u} ^{*}$ is globally asymptotically stable for the spatiotemporal system (\ref{eq1}), then all species will be spatially homogeneously distributed as time tends to infinity. This means, the spatiotemporal system (\ref{eq1}) has no nonconstant positive steady states i.e. no stationary patterns arise if (\ref{existenceequi}) and (\ref{global}) hold, the diffusive rates have no role to play in the nonexistence of nonconstant steady states. We obtain the results in the next theorem.

\begin{theorem}\label{globaltheorem}
If the conditions
\begin{eqnarray}
\begin{cases}
 \dfrac{\alpha_1 v^{*} + \alpha_2 w^{*}}{\gamma + u^{*} }+ \dfrac{\alpha \gamma(\alpha_1 + \alpha_2 w^{*})}{2(\gamma + u^{*})} \leq \dfrac{r}{k},\\
 \\\label{global}
 \dfrac{\alpha \gamma(\alpha_1 + \alpha_2 w^{*})}{2(\gamma + u^{*})} +  c_2 \sigma(w' +w^{*} + \beta)  + \alpha \alpha_2 k \leq \sigma \beta + \sigma(1-c_1).
 \end{cases}
\end{eqnarray}
   satisfy then the constant positive solution $\textbf{u} ^{*}$ of the system (\ref{eq1}) is globally asymptotically stable.
\end{theorem}
\begin{proof}
Define the Lyapunov function 
\[E(t)= \int_{\Omega} \left( u - u^{*} - u^{*} \log \dfrac{u}{u^{*}}\right) + \int_{\Omega} \left( v - v^{*} - v^{*} \log \dfrac{v}{v^{*}}\right) +\int_{\Omega} \left( w - w^{*} - w^{*} \log \dfrac{w}{w^{*}}\right).\]
We note that $ E(t)\geq 0$ for all $t>0$. Furthermore, referring to (\ref{eq1}), we compute
\begin{eqnarray*}
E'(t) &=& \int_{\Omega} \left\{\dfrac{ u- u^{*}}{u} G_1 (\textbf{u}) + \dfrac{ v- v^{*}}{v} G_2 (\textbf{u}) + \dfrac{ w- w^{*}}{w} G_3 (\textbf{u})\right\}dx\\
&&+ \int_{\Omega} \left\{\dfrac{ u- u^{*}}{u} d_1 \Delta u + \dfrac{ v- v^{*}}{v} d_2 \Delta v + \dfrac{ w- w^{*}}{w} d_3 \Delta w \right\}dx 
\end{eqnarray*}
\begin{eqnarray*}
&=& \int_{\Omega}\bigg\{(u - u^{*})\left( r\left(1-\dfrac{u}{k} \right) - \dfrac{\alpha_1 v + \alpha_2 w}{\gamma + u}\right) + (v-v^{*})\bigg(\dfrac{\alpha(\alpha_1 v + \alpha_2 w)u}{(\gamma + u)v}  \\
&&+ c_1 \sigma (\beta v +w)+ c_2 \sigma \left( \beta + \dfrac{w}{v} \right)w - \sigma(\beta v +w)-\sigma lf w - \lambda w  - d \bigg)\\
&& + (w-w^{*})(\lambda v + \sigma lfv-\sigma(\beta v + w) -(d+e) ) \bigg\}dx\\
&&- \int_{\Omega} \left\{ \dfrac{d_1 u^{*}}{u^2}|\bigtriangledown u|^2  +  \dfrac{d_2 v^{*}}{v^2}|\bigtriangledown v|^2  + \dfrac{d_3 w^{*}}{w^2}|\bigtriangledown w|^2 \right\}dx \triangleq I_1(t) + I_2 (t).
\end{eqnarray*}
Because of Neumann boundary condition, it is obvious that:
\[I_2(t) = - \int_{\Omega} \left\{ \dfrac{d_1 u^{*}}{u^2}|\bigtriangledown u|^2  +  \dfrac{d_2 v^{*}}{v^2}|\bigtriangledown v|^2  + \dfrac{d_3 w^{*}}{w^2}|\bigtriangledown w|^2 \right\}dx \leq 0.\]
Now
\begin{eqnarray*}
I_1(t) &=& \int_{\Omega}\bigg[ -\dfrac{r}{k} (u - u^{*})^2  - \dfrac{\alpha_1}{(\gamma + u)}(u - u^{*})(v-v^{*})  - \dfrac{\alpha_2}{(\gamma + u)}(u - u^{*})(w-w^{*})\\
&&+  \dfrac{(\alpha_1 v^{*} + \alpha_2 w^{*})}{(\gamma + u)(\gamma + u^{*})}(u - u^{*})^2  + \bigg( \dfrac{\alpha \alpha_1 \gamma}{(\gamma + u)(\gamma + u^{*})}  + \dfrac{\alpha \alpha_2 \gamma w^{*}}{(\gamma + u)(\gamma + u^{*})v}\bigg)(u-u^{*})(v-v^{*})\\
&&+ \bigg[ c_2 \sigma \dfrac{w}{v} +c_2 \sigma \dfrac{w^{*}}{v}  - \{\sigma(1-c_1)-c_2 \sigma \beta + \sigma l f + \lambda\} + \dfrac{\alpha \alpha_2 u}{(\gamma + u)v}  \bigg](v-v^{*})(w-w^{*})\\
&&- \bigg\{ \sigma(1-c_1) \beta + \dfrac{c_2 \sigma {w^{*}}^2}{v v^{*}} + \dfrac{ \alpha \alpha_2 u^{*}w^{*}} {(\gamma + u^{*})v v^{*}} \bigg\} (v-v^{*})^2 + (\lambda + \sigma l f - \sigma \beta)(v-v^{*})(w-w^{*})\\
&&- \sigma (w-w^{*})^2 \bigg]dx\\
 &\leq &\int_{\Omega}\bigg[ \bigg\{-\dfrac{r}{k}+  \dfrac{(\alpha_1 v^{*} + \alpha_2 w^{*})}{(\gamma + u)(\gamma + u^{*})}\bigg\}(u - u^{*})^2 + (\lambda + \sigma l f - \sigma \beta)(v-v^{*})(w-w^{*})\\&&- \sigma (w-w^{*})^2  + \bigg( \dfrac{\alpha \alpha_1 \gamma}{(\gamma + u)(\gamma + u^{*})}  + \dfrac{\alpha \alpha_2 \gamma w^{*}}{(\gamma + u)(\gamma + u^{*})v}\bigg)(u-u^{*})(v-v^{*})\\
&&+ \bigg\{ c_2 \sigma \dfrac{w}{v} +c_2 \sigma \dfrac{w^{*}}{v}- \{\sigma(1-c_1)-c_2 \sigma \beta + \sigma l f + \lambda\}+ \dfrac{\alpha \alpha_2 u}{(\gamma + u)v}  \bigg\}(v-v^{*})(w-w^{*})
\bigg]dx
\end{eqnarray*}

 Using the inequality $a^2 +b^2 \geq 2ab$ and by Theorem \ref{theoremupper}, we derive for $t>T$ that: 
\begin{eqnarray*}
I_1 (t) & \leq & \int_{\Omega} \bigg[ \bigg\{ \dfrac{\alpha_1 v^{*} + \alpha_2 w^{*}}{\gamma + u^{*} }+ \dfrac{\alpha \gamma(\alpha_1 + \alpha_2 w^{*})}{2(\gamma + u^{*})} - \dfrac{r}{k} \bigg\}(u-u^{*})^2\\
&&+ \dfrac{1}{2}\bigg\{\dfrac{\alpha \gamma(\alpha_1 + \alpha_2 w^{*})}{2(\gamma + u^{*})} +  c_2 \sigma(w' +w^{*} + \beta)  + \alpha \alpha_2 k - \sigma \beta - \sigma(1-c_1)\bigg\}(v-v^{*})^2\\
&&+\dfrac{1}{2}\bigg\{c_2 \sigma(w' +w^{*} + \beta)  + \alpha \alpha_2 k - \sigma \beta - \sigma(1-c_1)-\sigma\bigg\}(w-w^{*})^2 \bigg] dx
\end{eqnarray*}
~~~~~~~~~~~~~~~~~~~~~~~~~~~~~~~~~~~~~~~~~~~~~~~~~~~~~~~~~~~~~~~~~~~~~~~~~~~~~~~~here, $w' = \max\limits_{\substack{\bar{\Omega}}} w,$\\
under the assumptions (\ref{global}), we obtain that $I_1(t) \leq 0$, and thus 
\[E'(t) = I_{1}(t) + I_2 (t) \leq 0,\]
 and the equality holds if $(u,v,w)=(u^{*},v^{*},w^{*})$. The proof is completed. 
\end{proof}

 The conditions imposed in Theorems \ref{localtheorem} and \ref{globaltheorem} are independent of diffusive rates as $d_1,d_2$ and $d_3$. It shows that if the conditions stated in theorems satisfy then, no nonconstant positive steady states exist and hence no stationary patterns arise. Here we are not able to understand the importance of diffusive rates in the distribution of populations and the structure of communities. Therefore, in the next section, we show that the large diffusivity may be reasonable for the existence and nonexistence of nonconstant positive steady states, i.e. stationary patterns.

 \section{Nonconstant positive steady states}\label{nonexist}
 In this section, we have obtained the conditions for the nonexistence and existence of the nonconstant positive steady states of the following steady state system:
\begin{eqnarray} 
 \begin{cases} 
  - d_1 \Delta u = ru\left(1-\dfrac{u}{k} \right) - \dfrac{(\alpha_1 v + \alpha_2 w)u}{\gamma + u},\\
 - d_2 \Delta v = \dfrac{\alpha(\alpha_1 v + \alpha_2 w)u}{\gamma + u} + c_1 \sigma (\beta v +w)v\\
 ~~~~~~~~~~~~~~+ c_2 \sigma (\beta v + w)w - \sigma (\beta v + w )v - \sigma lfvw -\lambda v w -dv,\\
 - d_3 \Delta w  = \lambda vw + \sigma lfvw-\sigma(\beta v + w)w -(d+e)w,\\\label{nonhomo}
 \dfrac{\partial u }{\partial \nu} =\dfrac{\partial v }{\partial \nu}=\dfrac{\partial w }{\partial \nu} = 0.~~~~~x \in \partial \Omega.
  \end{cases}
   \end{eqnarray}
 
For the purpose, we are going to establish a priori upper and lower bounds for positive solutions of the problem (\ref{nonhomo}).

\subsection{A priori estimations}
The main purpose of this subsection is to give a priori upper and lower bounds for the positive solutions of (\ref{nonhomo}). For that, we have used two important results: Harnack Inequality and Maximum Principle. Which are due to Lin et al. \cite{lin1988large}, and Lou and Ni \cite{lou1996diffusion}. For our convenience, let us denote the constants collectively by $\Lambda.$
 
\begin{lemma}\label{lemma1}
Assume that $c(x) \in C(\bar{\Omega})$ and let $\omega(x) \in C^{2}(\Omega) \cap C^{1}(\bar{\Omega})$ be a positive solution to:
 $\Delta \omega + c(x) \omega =0,~x \in \Omega$, where $c(x)$, satisfying the homogeneous Neumann boundary condition, then there exists a positive constant $C_{*} = C_{*}(N,\Omega,||c||_{\infty})$ such that
\[\max_{\substack{\bar{\Omega}}} \omega \leq C_{*} \min_{\substack{\bar{\Omega}}} \omega. \] 
\end{lemma}
\begin{lemma}\label{lemma2}
Suppose that $g \in C(\Omega \times {\mathbb{R}}^{1})$ and $b_j \in C(\bar{\Omega}),$~$j=1,2,...,N$.
\begin{enumerate}
\item[(i)] If $\omega(x) \in C^{2}(\Omega) \cap C^{1}(\bar{\Omega})$ satisfies $\Delta \omega + \sum\limits_{j=1}^{N} b_j(x) \omega_{x_j} + g(x,\omega(x)) \geq 0,~x \in \Omega,~\frac{\partial \omega}{\partial \nu} \leq 0,~x \in \partial \Omega$, and $\omega(x_0) =  \max\limits_{\substack{\bar{\Omega}}} \omega (x)$ then $g(x_0, \omega (x_0)) \geq 0$.
\item[(ii)] If $\omega(x) \in C^{2}(\Omega) \cap C^{1}(\bar{\Omega})$ satisfies $\Delta \omega + \sum\limits_{j=1}^{N} b_j(x) \omega_{x_j} + g(x,\omega(x)) \leq 0,~x \in \Omega,~\frac{\partial \omega}{\partial \nu} \geq 0,~x \in \partial \Omega$, and $\omega(x_0) =  \min\limits_{\substack{\bar{\Omega}}} \omega (x)$ then $g(x_0, \omega (x_0)) \leq 0$.
\end{enumerate}
\end{lemma}
By using above lemmas the result of upper bounds can be stated as follows:
\begin{theorem} \label{theoremupper}
For any positive solution $\textbf{u} = (u,v,w)^{T}$ of (\ref{nonhomo}), it can be drawn that
\begin{eqnarray} 
\max_{\substack{\bar{\Omega}}} u  &\leqslant & k,\\\label{upper}
\max_{\substack{\bar{\Omega}}} v &\leqslant & \dfrac{-B + \sqrt{B^2 - 4AC}}{2A},\\
\max_{\substack{\bar{\Omega}}} w &\leqslant &\dfrac{1}{\sigma}\bigg(\dfrac{(\lambda + \sigma l f - \sigma \beta )(-B + \sqrt{B^2 - 4AC})}{2A} - (d+e)\bigg),
\end{eqnarray}

 where $A>0,B<0$ and $4AC<B^2$, given in (\ref{noto}).
\end{theorem}
\begin{proof}Let $x_u, x_v, x_w \in \Omega,$ and 
\[u(x_u)= \max_{\substack{x \in \bar{\Omega}}} u(x),~~~~v(x_v)= \max_{\substack{x \in \bar{\Omega}}} v(x),~~~~w(x_w)= \max_{\substack{x \in \bar{\Omega}}} w(x),\]

 then \[-\Delta u(x_u) \geqslant 0,~~-\Delta v(x_v) \geqslant 0,~~-\Delta w(x_w) \geqslant 0.\]
From the first equation of (\ref{eq1}),
\[ru(x_u)\left(1-\dfrac{u(x_u)}{k} \right) - \dfrac{(\alpha_1 v(x_u) + \alpha_2 w(x_u))u(x_u)}{\gamma + u(x_u)} \geqslant 0 \]
\[ru(x_u)\left(1-\dfrac{u(x_u)}{k} \right)  \geqslant 0\]
\begin{equation}\label{boundofu}
\implies \max_{\substack{x \in \bar{\Omega}}} u(x)  \leqslant  k.
\end{equation}
 
 From the third equation of (\ref{eq1}),
 
 \[\lambda v(x_w) w (x_w)+ \sigma l f v(x_w) w(x_w) - \sigma (\beta v (x_w)+w(x_w))w(x_w)-(d+e)w(x_w) \geqslant 0, \]
\begin{equation}\label{boundofw}
w(x_w) \leqslant  \dfrac{1}{\sigma}\bigg((\lambda + \sigma l f - \sigma \beta )v(x_w) - (d+e)\bigg)\leqslant  \dfrac{1}{\sigma}\bigg((\lambda + \sigma l f - \sigma \beta )v(x_v) - (d+e)\bigg). 
\end{equation}

 From the second equation of (\ref{eq1}),
 
 \[\dfrac{\alpha(\alpha_1 v(x_v)  + \alpha_2 w(x_v))u(x_v)}{\gamma + u(x_v)} + c_1 \sigma (\beta v(x_v) +w(x_v))v(x_v) + c_2 \sigma (\beta v(x_v) + w(x_v))w(x_v)\]
 \[ - \sigma (\beta v(x_v) + w(x_v) )v(x_v) - \sigma lfv(x_v)w(x_v) -\lambda v(x_v) w(x_v) -dv(x_v) \geqslant 0,\]
  \[\dfrac{\alpha(\alpha_1 v(x_v)  + \alpha_2 w(x_w))u(x_u)}{\gamma + u(x_u)} + c_1 \sigma (\beta v(x_v) +w(x_w))v(x_v) + c_2 \sigma (\beta v(x_v) + w(x_w))w(x_w)\]
 \[ - \sigma (\beta v(x_v) + w(x_w) )v(x_v) - \sigma lfv(x_v)w(x_w) -\lambda v(x_v) w(x_w) -dv(x_v) \geqslant 0,\]
 after putting the upper bound of $u$ and $w$ from (\ref{boundofu}) and (\ref{boundofw}), respectively. We obtain 
 \[A v^2(x_v) + B v(x_v) + C \leqslant 0, \]
 where
\begin{equation}\label{noto}
\begin{array}{l}
A=(1-c_1)\sigma \beta - \dfrac{c_2}{\sigma}(\lambda  + \sigma lf - \sigma \beta)^2 + \dfrac{1}{\sigma}(\lambda  + \sigma lf - \sigma \beta)(\sigma + \sigma l f  + \lambda - c_1 \sigma - c_2 \sigma \beta),\\
 B=- \dfrac{\alpha \alpha_1 k}{\gamma + k} + \dfrac{1}{\sigma}(\lambda  + \sigma lf - \sigma \beta)\left(2 c_2 (d+e) - \dfrac{\alpha \alpha_2 k}{\gamma + k}  \right) + (d+e)(c_1 \sigma + c_2 \sigma \beta - \sigma -\sigma l f  - \lambda),\\
  ~~\text{and}~~C= \dfrac{(d+e)}{\sigma} \left( \dfrac{\alpha \alpha_2 k}{\gamma + k} - c_2 (d+e)  \right).

\end{array}
\end{equation}

 Then 
 \[v(x_v) = \max_{\substack{x \in \bar{\Omega}}} v(x)  \leqslant \dfrac{-B + \sqrt{B^2 - 4AC}}{2A}, \]
 where $A>0,B<0$ and $4AC<B^2$. 
 
  Therefore from (\ref{boundofw})
 \[w(x_w) =  \max_{\substack{x \in \bar{\Omega}}} w(x) \leqslant \dfrac{1}{\sigma}\bigg(\dfrac{(\lambda + \sigma l f - \sigma \beta )(-B + \sqrt{B^2 - 4AC})}{2A} - (d+e)\bigg). \]
\end{proof}
\begin{theorem}\label{theoremlower}
Let $\Lambda$ and  $\underline{d}_1,\underline{d}_2,\underline{d}_3$ be fixed positive constants. Assume that
\[(d_1,d_2,d_3) \in [ \underline{d}_1,\infty)\times[ \underline{d}_2,\infty)\times [ \underline{d}_3,\infty),\]
and 
\begin{equation}\label{condition}
(\lambda + \sigma l f - \sigma \beta)v'<{\sigma {w'}+(d+e)},
\end{equation}
where, $v'=\max \limits_{\substack{\bar{\Omega}}} v(x)~~\text{and}~~w'=\max \limits_{\substack{\bar{\Omega}}} w(x)$, as is given in Theorem \ref{theoremupper}. Then there exist a positive constant $\underline{C} = \underline{C}(\Lambda,\underline{d}_1,\underline{d}_2,\underline{d}_3)$, such that any positive solution $(u,v,w)$ of (\ref{nonhomo}) satisfies
\begin{equation}\label{minimumcond}
\min_{\substack{\bar{\Omega}}} u > \underline{C},~~\min_{\substack{\bar{\Omega}}} v > \underline{C},~~\min_{\substack{\bar{\Omega}}} w > \underline{C}.
\end{equation}
\end{theorem}
\begin{proof} Let 
\begin{eqnarray*}
c_1(x) &\triangleq & \dfrac{1}{d_1} \left( ru\left(1-\dfrac{u}{k} \right) - \dfrac{(\alpha_1 v + \alpha_2 w)}{\gamma + u} \right),\\
c_2(x) &\triangleq & \dfrac{1}{d_2}\bigg(\dfrac{\alpha(\alpha_1 v + \alpha_2 w)}{\gamma + u} + c_1 \sigma (\beta v +w)v +  c_2 \sigma (\beta v + w)w\\
&&~~~~~~~- \sigma (\beta v + w )v - \sigma lfvw -\lambda v w -dv \bigg),\\
c_3(x) &\triangleq & \dfrac{1}{d_3} \left(\lambda vw + \sigma lfvw-\sigma(\beta v + w)w -(d+e)w \right).
\end{eqnarray*}
Then, in view of (\ref{upper}), there exists a positive constant $\bar{C}=\bar{C}(\Lambda , \bar{d})$ such that 
\[||c_1(x)||_{\infty},~~||c_2(x)||_{\infty},~~||c_3(x)||_{\infty} \leq {\bar{C}},~~\text{if}~~d_1,d_2,d_3 \geq \bar{d}.\]
 Lemma \ref{lemma1} shows that there exists a positive constant $C^{*} =C^{*}(\Lambda , \bar{d})$ such that
\[\max_{\substack{\bar{\Omega}}} u \leq C^{*}\min_{\substack{\bar{\Omega}}} u, ~~~\max_{\substack{\bar{\Omega}}} v \leq C^{*} \min_{\substack{\bar{\Omega}}} v,~~~\max_{\substack{\bar{\Omega}}} w \leq C^{*}\min_{\substack{\bar{\Omega}}} w.\]
Presently, assuming, on the contrary, that (\ref{minimumcond}) does not hold, there will be a sequence $\{(d_{1n},d_{2n},d_{3n})\}_{n=1}^{\infty}$ with $(d_{1n},d_{2n},d_{3n}) \in [ \underline{d}_1,\infty)\times[ \underline{d}_2,\infty)\times [ \underline{d}_3,\infty)$  such that the corresponding positive solutions $(u_n, v_n, w_n)$ of (\ref{nonhomo}) satisfy
\[\max_{\substack{\bar{\Omega}}} u_n \rightarrow 0 ~~~~\text{or}~~~ \max_{\substack{\bar{\Omega}}} v_n \rightarrow 0 ~~~~\text{or}~~~~\max_{\substack{\bar{\Omega}}} w_n \rightarrow 0. \]
The standard regularity theorem for the elliptic equations and $d_{1n} \geq \underline{d}_1,d_{2n} \geq \underline{d}_2,d_{3n} \geq \underline{d}_3$, yields that there exists a subsequence of $\{(u_n,v_n,w_n)\}_{n=1}^{\infty}$ which is still denoted by $\{(u_n,v_n,w_n)\}_{n=1}^{\infty}$, and non-negative functions $u,v,w \in C^{2}(\bar{\Omega})$, such that $(u_n,v_n,w_n) \rightarrow (u,v,w)$ as $n \rightarrow \infty$. By (\ref{minimumcond}), it's noted that $u \equiv 0$ or $v \equiv 0$ or  $w \equiv 0$.\\
For all $n \geq 1$, integrating by parts, it is obtained that 
\begin{eqnarray}\nonumber
&&\int_\Omega \left( r{u_n}\left(1-\dfrac{{u_n}}{k} \right) - \dfrac{(\alpha_1 {v_n} + \alpha_2 {w_n})u_n}{\gamma + {u_n}} \right)dx=0,\\\label{sequence}
&&\int_\Omega \bigg(\dfrac{\alpha(\alpha_1 {v_n} + \alpha_2 {w_n})u_n}{\gamma + {u_n}} + c_1 \sigma (\beta {v_n} +{w_n}){v_n} +  c_2 \sigma (\beta {v_n} + {w_n}){w_n}\\\nonumber
&&~~~~~~~- \sigma (\beta {v_n} + {w_n} ){v_n} - \sigma lf{v_n}{w_n} -\lambda {v_n} {w_n} -d{v_n} \bigg)=0, \\\nonumber
&&\int_\Omega \left(\lambda {v_n}{w_n} + \sigma lf{v_n}{w_n}-\sigma(\beta {v_n} + {w_n}){w_n} -(d+e){w_n} \right)dx = 0.
\end{eqnarray}
Letting $n \rightarrow \infty$ in (\ref{sequence}), it can be obtained that
\begin{eqnarray}\nonumber
&&\int_\Omega \left( ru\left(1-\dfrac{u}{k} \right) - \dfrac{(\alpha_1 v + \alpha_2 w)u}{\gamma + u} \right)dx=0,\\\label{sequence1}
&&\int_\Omega \bigg(\dfrac{\alpha(\alpha_1 v + \alpha_2 w)u}{\gamma + u} + c_1 \sigma (\beta v +w)v +  c_2 \sigma (\beta v + w)w\\\nonumber
&&~~~~~~~- \sigma (\beta v + w )v - \sigma lfvw -\lambda v w -dv \bigg)=0, \\\nonumber
&&\int_\Omega \left(\lambda vw + \sigma lfvw-\sigma(\beta v + w)w -(d+e)w \right)dx = 0.
\end{eqnarray}
Here, we have mentioned all the three cases as follows:
\begin{enumerate}
\item[\textbf{Case 1.}]$u \equiv 0, ~v \not\equiv 0,~ w \not\equiv 0$ on $\bar{\Omega}$. 

From the Hopf boundary lemma $v>0,w>0$ on $\bar{\Omega}$, then we have
\begin{equation}\
 \begin{cases}
 -\bar{d}_2\Delta v = c_1 \sigma (\beta v +w)v +  c_2 \sigma (\beta v + w)w - \sigma (\beta v + w )v\\
 ~~~~~~~~~ - \sigma lfvw -\lambda v w -dv,~~~~~~~~~~~~~~~~~~~~~~~~~~~~~~~~~~~~~~~~~~~x \in \Omega,\\\nonumber
 \dfrac{\partial v}{\partial \nu} = 0,~~~~~~~~~~~~~~~~~~~~~~~~~~~~~~~~~~~~~~~~~~~~~~~~~~~~~~~~~~~~~~~~~~~~~x \in \partial \Omega.
\end{cases}
\end{equation}
Let $v(y_0) = \max\limits_{\substack{\bar{\Omega}}} v(x) >0$. Applying Maximum principle and third inequality of (\ref{upper})
\[c_1 \sigma (\beta v(y_0) +w(y_0))v(y_0) +  c_2 \sigma (\beta v(y_0) + w(y_0))w(y_0) - \sigma (\beta v(y_0) + w(y_0) )v(y_0) \]
\[- \sigma lfv(y_0)w (y_0)-\lambda v(y_0) w(y_0) -dv(y_0) \geq 0,\]

 which implies,
 \[v(y_0)  \leq  \dfrac{c_2 \sigma w' }{(\lambda + \sigma lf )} \leq \dfrac{ \sigma w' }{(\lambda + \sigma lf )},\]

~~~~~~~~~~~~~~~~~~~~~~~~~~~~~~~~~~~~~~~~~~~~~~~~~~~~~~~~provided that $c_1<1,c_2<1~~\text{and}~~c_1 + c_2 \beta <1$. 

Based on assumption (\ref{condition}) stated above, it is easy to see that 
\[\lambda v_n w_n + \sigma lfv_n w_n-\sigma(\beta v_n + w_n)w_n -(d+e)w_n <0.\]
Integrating the differential equation for $ w_n$ over $\Omega$ by parts, the result is that 
\[0 = d_{3n} \int_{\Omega} \partial_{\nu} w_n dS = -d_{3n} \int_{\Omega} \Delta w_n dx \]
\[=\int_{\Omega} \left(\lambda v_n w_n + \sigma lfv_n w_n-\sigma(\beta v_n + w_n)w_n -(d+e)w_n \right) dx <0, \]
which is a contradiction.

\item[\textbf{Case 2.}] $v\equiv0,~u\not\equiv0$ on $\bar{\Omega}$. 

Since $v_n \rightarrow v \equiv 0$ as $n \rightarrow \infty$, then 
\[\dfrac{\alpha(\alpha_1 {v_n} + \alpha_2 {w_n})u_n}{\gamma + {u_n}} + c_1 \sigma (\beta {v_n} +{w_n}){v_n} +  c_2 \sigma (\beta {v_n} + {w_n}){w_n} \] 
\[- \sigma (\beta v + w )v - \sigma lfvw -\lambda v w -dv> 0,\]
as $n \gg 1$. Next, integrating, the differential equation by parts for $ v_n$ over $\Omega$, we have 
\[0 = d_{2n} \int_{\Omega} \partial_{\nu} v_n dS = -d_{2n} \int_{\Omega} \Delta v_n dx  \]
\[=\int_{\Omega}\bigg( \dfrac{\alpha(\alpha_1 {v_n} + \alpha_2 {w_n})u_n}{\gamma + {u_n}} + c_1 \sigma (\beta {v_n}+{w_n}){v_n}\]
\[ +  c_2 \sigma (\beta {v_n} + {w_n}){w_n} - \sigma (\beta v_n + w_n )v_n - \sigma lfv_n w_n -\lambda v_n w_n -dv_n \bigg) dx >0,\]
which is a contradiction.
\item[\textbf{Case 3.}] $w \equiv 0, ~u \not\equiv 0,~ v \not\equiv 0$ on $\bar{\Omega}$. 

Then by Hopf boundary lemma, $u>0,v>0$ on $\bar{\Omega}$. And $u$ and $v$ satisfy
\begin{equation}\
 \begin{cases}
 -\bar{d}_2\Delta v = \dfrac{\alpha \alpha_1 u v}{\gamma + u} + c_1 \sigma \beta v^2 - \sigma \beta v^2 -dv,~~~~~x \in \Omega,\\\nonumber
 \dfrac{\partial v}{\partial \nu} = 0,~~~~~~~~~~~~~~~~~~~~~~~~~~~~~~~~~~~~~~~~~~~~x \in \partial \Omega.
\end{cases}
\end{equation}
Let $v(y_1) = \max\limits_{\substack{\bar{\Omega}}} v(x) >0$. Applying Lemma \ref{lemma2} and the first inequality of (\ref{upper}),
\[\dfrac{\alpha \alpha_1 u (y_1) v(y_1)}{\gamma + u(y_1)} + c_1 \sigma \beta {v^2(y_1)} - \sigma \beta {v^2(y_1)} -dv(y_1)\geq 0,\] 
which implies,
\[  v(y_1) \leq \dfrac{\alpha \alpha_1 k}{ \sigma \beta \gamma(1-c_1)} \leq \dfrac{\alpha k}{d_2}\left(d_1 + \dfrac{d_2 r }{d} + \dfrac{d_3 r }{d+e} \right).\]
Based on assumption (\ref{condition}) stated above, it is easy to see that 
\[\lambda v_n w_n + \sigma lfv_n w_n-\sigma(\beta v_n + w_n)w_n -(d+e)w_n <0.\]
Now, integrating the differential equation by parts, for $ w_n$ over $\Omega$, then we get 
\[0 = d_{3n} \int_{\Omega} \partial_{\nu} w_n dS = -d_{3n} \int_{\Omega} \Delta w_n dx \]
\[=\int_{\Omega} \left(\lambda v_n w_n + \sigma lfv_n w_n-\sigma(\beta v_n + w_n)w_n -(d+e)w_n \right) dx <0, \]
which is a contradiction. Hence, proof of the theorem is completed.
\end{enumerate}
\end{proof}
\subsection{Nonexistence of nonconstant positive steady states}
This subsection deals with the nonexistence of nonconstant positive steady states of  (\ref{nonhomo}). Here, we vary the diffusion coefficient $d_3$ and fix the other parameters $d_1, d_2,$ and $\Lambda$. 
 
\begin{theorem}\label{theoremnon}
Let $d^{*}_1$ and $d^{*}_2$ be fixed positive constants such that
\begin{equation}\label{noncons}
d^{*}_1 \geq \dfrac{r}{\mu_1},~~d^{*}_2 \geq \dfrac{\sigma w'(c_1 + c_2 \beta)}{\mu_1}.
\end{equation}
Where $w'=\max_{\substack{\bar{\Omega}}} w(x)$ and $\mu_1$ is the smallest positive eigenvalue of the operator $-\Delta$ on $\Omega$ with homogeneous Neumann boundary condition. Then there exists a positive constant ${D}_3 = D_3 (\Lambda,d^{*}_1,d^{*}_2)$, such that, when $d_3 >D_3,d_1 \geq d^{*}_1$ and $d_2 \geq d^{*}_2$, problem (\ref{nonhomo}) has no nonconstant positive steady state.  
\end{theorem}
\begin{proof}
Let $(\bar{u},\bar{v},\bar{w})$ be a positive solution of (\ref{nonhomo}), where
\[ \bar{u} \cong \dfrac{1}{|\Omega|}\int_{\Omega} u~ dx,~~\bar{v} \cong \dfrac{1}{|\Omega|}\int_{\Omega} v~ dx,~~\bar{w} \cong \dfrac{1}{|\Omega|}\int_{\Omega} w~ dx.\]
Now we multiply first equation of (\ref{nonhomo}) by $(u-\bar{u})$ second equation by $(v-\bar{v})$ and third equation by $(w-\bar{w})$ then,
\begin{eqnarray} \nonumber
  - d_1 (u-\bar{u})\triangledown^2 u &=& (u - \bar{u})~G_1 (u,v,w),\\\label{non1}
 - d_2 (v-\bar{v})\triangledown^2 v &=& (v-\bar{v})~ G_2 (u,v,w),\\
 - d_3 (w-\bar{w})\triangledown^2 w &=& (w-\bar{w})~G_3 (u,v,w).\nonumber
  \end{eqnarray}
Now we integrate (\ref{non1}) over $\Omega$ and apply Green's first identity, then
we have
\begin{eqnarray*}
d_1 \int_{\Omega}| \bigtriangledown u |^2 dx + d_2 \int_{\Omega}| \bigtriangledown v |^2dx + d_3 \int_{\Omega}| \bigtriangledown w |^2 dx
\end{eqnarray*}
\begin{eqnarray*}
&=& \int_{\Omega}(u - \bar{u})(G_1 (u,v,w) - G_1 (\bar{u},\bar{v},\bar{w}))dx + \int_{\Omega}(v - \bar{v})(G_2 (u,v,w) - G_2 (\bar{u},\bar{v},\bar{w}))dx \\
&&+ \int_{\Omega}(w - \bar{w})(G_3 (u,v,w) - G_3 (\bar{u},\bar{v},\bar{w}))dx,\\
&=&\int_{\Omega}(u-\bar{u})^2 \left(r \left(1 - \dfrac{u + \bar{u}}{k}\right) - \dfrac{\gamma (1-\alpha)(\alpha_1 v + \alpha_2 w) }{(\gamma + \bar{u})(\gamma + u)}\right)dx\\
&&+\int_{\Omega}(v - \bar{v})^2  (\sigma(c_1 + c_2 \beta)w -(1-c_1)\sigma \beta (v + \bar{v})-(\sigma + \lambda + \sigma l f)w -d)dx \\
&& + \int_{\Omega} (w - \bar{w})^2 ( (\lambda + \sigma l f -  \sigma \beta ) \bar{v} - \sigma(w + \bar{w}) - (d+e))dx\\
&&-\int_{\Omega}(u-\bar{u})(v-\bar{v})\left( \dfrac{(1 - \alpha) \alpha_1 \bar{u} }{\gamma + \bar{u}}\right) dx- \int_{\Omega}(u-\bar{u})(w-\bar{w})\left( \dfrac{(1-\alpha) \alpha_2 \bar{u} }{\gamma + \bar{u}}\right)dx\\
&&+ \int_{\Omega} (v-\bar{v})(w-\bar{w})((\lambda + \sigma l f - \sigma \beta)w + c_2 \sigma (w + \bar{w}) -((1 - c_1 - c_2 \beta)\sigma + \lambda + \sigma l f) \bar{v})dx. \\
 \end{eqnarray*}
By Cauchy inequality, we have
\begin{eqnarray*}
&& \int_{\Omega}\{ d_1 |\bigtriangledown u |^2  + d_2 |\bigtriangledown  v |^2 + d_3| \bigtriangledown w |^2\}dx \\
&&\\ 
&&\leq \int_{\Omega} (u-\bar{u})^2 \left(r \left(1 - \dfrac{u + \bar{u}}{k}\right) - \dfrac{\gamma (1-\alpha)(\alpha_1 v + \alpha_2 w) }{(\gamma + \bar{u})(\gamma + u)} \right )dx\\
&&~ + \int_{\Omega} (v - \bar{v})^2  (\sigma w'(c_1 + c_2 \beta)-(1-c_1)\sigma \beta (v + \bar{v})-(\sigma + \lambda + \sigma l f)w -d + \epsilon )dx \\
&&~+ \int_{\Omega}(w - \bar{w})^2  ( (\lambda + \sigma l f -  \sigma \beta ) \bar{v} - \sigma(w + \bar{w}) - (d+e))dx	 \\
&&~+ \dfrac{1}{4 \epsilon}\int_{\Omega}(w-\bar{w})^2 \bigg((v-\bar{v})^2 ( (\lambda + \sigma l f -  \sigma \beta ) \bar{v} - \sigma(w + \bar{w}) - (d+e))^2 \bigg)dx,
\end{eqnarray*} 
 where, $w'=\max \limits_{\substack{\bar{\Omega}}} w(x)$ and $\epsilon$ is an arbitrary positive constant. 
 
 \indent Using Poincar\'e Inequality, we have
 \begin{eqnarray*}
&& \int_{\Omega}\{ d_1 |\bigtriangledown u |^2  + d_2 |\bigtriangledown  v |^2 + d_3| \bigtriangledown w |^2\}dx \\
&&~\geq \int_{\Omega} d_1 \mu_1 (u - \bar{u})^2 dx+ \int_{\Omega} d_2 \mu_1(v - \bar{v})^2dx + \int_{\Omega}d_3 \mu_1 (w - \bar{w})^2 dx,
\end{eqnarray*}
a sufficiently small $\epsilon_0$ can be chosen such that
\begin{eqnarray*}
d_1 \mu_1 &> & r \left(1 - \dfrac{u + \bar{u}}{k}\right) - \dfrac{\gamma (1-\alpha)(\alpha_1 v + \alpha_2 w) }{(\gamma + \bar{u})(\gamma + u)},\\
d_2 \mu_1 &>&\sigma w'(c_1 + c_2 \beta)-(1-c_1)\sigma \beta (v + \bar{v})-(\sigma + \lambda + \sigma l f)w -d + \epsilon_0.\\
\end{eqnarray*}
Lastly, by taking ${D}_3 > \dfrac{1}{\mu_1}\bigg[ ((\lambda + \sigma l f -  \sigma \beta ) \bar{v} - \sigma(w + \bar{w}) - (d+e)) + \dfrac{1}{4 \epsilon_0}(v-\bar{v})^2 ( (\lambda + \sigma l f -  \sigma \beta ) \bar{v} - \sigma(w + \bar{w}) - (d+e))^2\bigg] $, 

then we can conclude that  $u=\bar{u},v=\bar{v},w=\bar{w}$, which establishes proof of the theorem.
\end{proof}

From the above theorem, we conclude that if the top predator diffuses with a sufficiently large rate, then the model system (\ref{nonhomo}) has no nonconstant steady state, and hence no stationary patterns emerge.
\subsection{Existence of Nonconstant Positive Steady States}\label{existenceofnon}
 This subsection derives the conditions for the existence of nonconstant positive steady states of (\ref{nonhomo}). That means selecting parameters that exhibit stationary patterns. For this purpose, the diffusion coefficient $d_2$ is varied while the other parameters $\Lambda,d_1$ and $d_3$ are kept fixed.

To obtain the results, first we study the linearization of (\ref{nonhomo}) at constant positive solution ${\textbf{u}}^{*}$. Let's denote
\[\mathbf{X} = \{ \textbf{u} \in [C^2(\bar{\Omega})]^3~ |~ \partial_{\nu} \textbf{u} = 0~~\text{on}~~\partial \Omega\},~~~~{\mathbf{X}}^{+} = \{ \textbf{u} \in \mathbf{X}~|~u>0,v>0,w>0,~x \in \bar{\Omega}\},\]
and
\[\mathbb{B}(\Theta)  = \{\textbf{u} \in \textbf{X}~|~\Theta^{-1} <u,v,w<\Theta,x \in \bar{\Omega}\}.\]

Where $\Theta$ is a positive constant guaranteed
to obtain from Theorems \ref{theoremupper} and \ref{theoremlower}. The elliptic problem (\ref{nonhomo}) can be written in this form: 
\begin{equation}\label{converted}
\begin{cases}
-\mathcal{D}\Delta \textbf{u} = \textbf{G}(\textbf{u}),~~~~x \in \Omega\\
\partial_{\nu} \textbf{u}=0, ~~~~~~~~~~~~~x \in \partial \Omega. 
\end{cases}
\end{equation}
Here, $\textbf{u}$ be a positive solution of (\ref{converted}) if and only if
\[\textbf{F} (\textbf{u}) \triangleq \textbf{u} - (\textbf{I} - \Delta )^{-1}\{\mathcal{D}^{-1}\textbf{G}(\textbf{u}) + \textbf{u}\} =0 ~~\text{in}~~\mathbf{X}^{+}, \]

where $(\textbf{I} - \Delta )^{-1}$ is the inverse of $\textbf{I} - \Delta$ in $\mathbf{X}$. As $\textbf{F}(\cdot)$ is a compact perturbation of the identity operator, for any $\mathbb{B} = \mathbb{B}(\Theta)$. The Leray-Schauder degree deg$(\textbf{F}(\cdot), 0,\mathbb{B})$ is well-defined if $F(\textbf{u}) \neq 0$ on $\partial \mathbb{B}$. Further, we notice that
\[\mathbf{D}_{\textbf{u}} \textbf{F}({\textbf{u}}^{*}) = \textbf{I} - (\textbf{I} - \Delta)^{-1}\{\mathcal{D}^{-1} {\textbf{G}}_{\textbf{u}}({\textbf{u}}^{*}) + \textbf{I}\} ,\]
and if $\mathbf{D}_{\textbf{u}} \textbf{F}({\textbf{u}}^{*})$ is invertible, then the index of $\textbf{F}$ at ${\textbf{u}}^{*}$ is denoted as \\index$(\textbf{F}(\cdot),{\textbf{u}}^{*}) = (-1)^{\rho}$, where $\rho$ is counting multiplicities of eigenvalues with negative real parts of $\mathbf{D}_{\textbf{u}} \textbf{F}({\textbf{u}}^{*})$.

 We notice that, $\mathbf{X}_{ij}$ is invariant under $\mathbf{D}_{\textbf{u}} \textbf{F}({\textbf{u}}^{*})$ for each integer $i \geq 1$ and each integer $1  \leq j \leq  \dim \mathbf{E}(\mu_i)$. Thus, $\lambda$ is an eigenvalue of the matrix
\[\textbf{I} - \dfrac{1}{1 + \mu_i}[\mathcal{D}^{-1}{\textbf{G}}_{\textbf{u}}({\textbf{u}}^{*}) + \textbf{I}] = \dfrac{1}{1 + \mu_i}[\mu_i \textbf{I} -\mathcal{D}^{-1}{\textbf{G}}_{\textbf{u}}({\textbf{u}}^{*}). ]\]

Thus the matrix $\mathbf{D}_\textbf{u}{\textbf{F}}({\textbf{u}}^{*})$ is invertible if and only if, for all $i \geq 1,$ the matrix $\textbf{I} - \dfrac{1}{1 + \mu_i}[\mathcal{D}^{-1}{\textbf{G}}_{\textbf{u}}({\textbf{u}}^{*}) + \textbf{I}]$ is non-singular. Write
\begin{equation}\label{hfunction}
\mathbb{H}(\mu) = \mathbb{H}({\textbf{u}}^{*};\mu) \triangleq \det\{\mu \textbf{I} - \mathcal{D}^{-1}{\textbf{G}}_{\textbf{u}}({\textbf{u}}^{*})\} = \dfrac{1}{d_1 d_2 d_3}\det\{\mu \mathcal{D} - {\textbf{G}}_{\textbf{u}}({\textbf{u}}^{*})\}
\end{equation}
Furthermore, we note that if $\mathbb{H}(\mu_i) \neq 0$ then for each $1 \leq j \leq \dim \mathbf{E}(\mu_i),$ the number of negative eigenvalues of $\mathbf{D}_\textbf{u}\textbf{F}({\textbf{u}}^{*})$ on ${\mathbf{X}}_{ij}$ is odd iff $\mathbb{H}(\mu_i)<0$. From this, we can conclude the following result.
\begin{proposition}\label{propos1}
Suppose that for all $i \geq 0,$ the matrix $\mu_i \textbf{I} - \mathcal{D}^{-1}{\textbf{G}}_\textbf{u}({\textbf{u}}^{*})$ is non-singular. Then
\[\text{index}(\textbf{F}(\cdot), {\textbf{u}}^{*}) = (-1)^{\rho}~~~~~\text{where}~~\rho = \sum_{\substack{
i \geq 0,\mathbb{H}(\mu_i)<0
}}
\dim \mathbf{E}(\mu_i)\] 
\end{proposition}
According to this proposition, in order to calculate $\text{index}(\textbf{F}(\cdot), {\textbf{u}}^{*})$, we ought to consider cautiously the sign of $\mathbb{H}(\mu_i)$. Therefore we calculate
\begin{equation}\label{amu}
\mathcal{A}(\mu) \triangleq \det\{\mu \mathcal{D} - {\textbf{G}}_u({\textbf{u}}^{*})\} = \mathcal{A}_3(d_2) \mu^3 + \mathcal{A}_2(d_2) \mu^2 + \mathcal{A}_1 (d_2)\mu -\det\{{\textbf{G}}_\textbf{u}({\textbf{u}}^{*})\}
\end{equation}
with
\begin{eqnarray*}
 \mathcal{A}_3(d_2) &=& d_1 d_2 d_3,~ \mathcal{A}_2(d_2) = -(d_1 d_2 a_{33} + d_2 d_3 a_{11} + d_1 d_3 a_{22}),\\
\mathcal{A}_1(d_2) &=& d_1 (a_{22}a_{33} - a_{23}a_{32}) + d_2 a_{11}a_{33} + d_3(a_{11}a_{22} - a_{12}a_{21}), 
\end{eqnarray*}
where $a_{ij}$ are given in (\ref{jacobian}).\

Here, we have considered the dependence of $\mathcal{A}$ on $d_2$. If $\tilde{\mu}_1(d_2),\tilde{\mu}_2(d_2),\tilde{\mu}_3(d_2)$ be the roots of $\mathcal{A}(d_2;\mu)=0$ with $\text{Re}\{\tilde{\mu}_1\} \leq \text{Re}\{\tilde{\mu}_2\}\leq \text{Re}\{\tilde{\mu}_3\}$, then  
\[\tilde{\mu}_1 (d_2)\tilde{\mu}_2(d_2) \tilde{\mu}_3 (d_2)= \det\{\textbf{G}_{\textbf{u}}({\textbf{u}}^{*})\}.\]
Note that $\det\{ {\textbf{G}}_u({\textbf{u}}^{*}) \}< 0 $ and $\mathcal{A}_3(d_2)>0.$ Thus, one of $\tilde{\mu}_1(d_2),\tilde{\mu}_2(d_2),\tilde{\mu}_3(d_2)$ is real and negative, and the product of two is positive.\\
Now, performing the following limits:
\begin{eqnarray*}
\lim_{d_2 \to \infty} \dfrac{\mathcal{A}(\mu)}{d_2} &=& \lim_{d_2 \to \infty}  \dfrac{\mathcal{A}_3 \mu^3 + \mathcal{A}_2 \mu^2 + \mathcal{A}_1 \mu}{d_2}\\
&=& \mu[d_1 d_3 \mu^2  - (d_1 a_{33} + d_3 a_{11}) \mu  + a_{11}a_{33}].
\end{eqnarray*}
 If the parameters $\Lambda,d_1, d_3$ satisfy $a_{11} d_3 +a_{33} d_1 > 0$, then we can establish the following result:
 
 \begin{proposition}\label{propo2}
 Assume the condition (\ref{existenceequi}) holds, and $a_{11} >0$. Then there exists a positive constant $D_2$, such that when $d_2 \geq D_2$, then the three roots \\$\tilde{\mu}_1(d_2),\tilde{\mu}_2(d_2),\tilde{\mu}_3(d_2)$ of $A(d_2;\mu) = 0$, all are real and satisfy
 
\begin{equation}
\begin{cases}
\lim \limits_{d_2 \to \infty}\tilde{\mu}_1(d_2)  = \dfrac{d_3 a_{11} + d_1 a_{33} - \sqrt{(d_3 a_{11} + d_1 a_{33})^2 - 4 d_1 d_3 a_{11}a_{33}}}{2d_1 d_3} = \dfrac{a_{33}}{d_3}<0,\\
\lim \limits_{d_2 \to \infty} \tilde{\mu}_2(d_2) = 0,\\
\lim \limits_{d_2 \to \infty}\tilde{\mu}_3(d_2)  = \dfrac{d_3 a_{11} + d_1 a_{33} + \sqrt{(d_3 a_{11} + d_1 a_{33})^2 - 4 d_1 d_3 a_{11}a_{33}}}{2d_1 d_3} = \dfrac{a_{11}}{d_1}>0.
\end{cases}
\end{equation}
Moreover, 
\begin{equation}\label{moreover}
\begin{cases}
-\infty < \tilde{\mu}_1 (d_2) <0 < \tilde{\mu}_2 (d_2)<\tilde{\mu}_3 (d_2),\\
A (d_2;\mu)<0 ~~if~~\mu \in (-\infty ,\tilde{\mu}_1 (d_2)) \cup (\tilde{\mu}_2 (d_2),\tilde{\mu}_3 (d_2)), \\
A (d_2;\mu)>0 ~~if~~\mu \in (\tilde{\mu}_1 (d_2),\tilde{\mu}_2 (d_2)) \cup (\tilde{\mu}_3(d_2),\infty).
 \end{cases}
\end{equation}
\end{proposition}
 
Next, we will prove the existence of nonconstant positive steady states of (\ref{nonhomo}), when $d_2$ is sufficiently large and the other parameters are fixed.
\begin{theorem}\label{theoremconst}
Assume the parameters $d_1$ and $d_3$ are fixed, $a_{11} >0$ and $a_{11}a_{33}<0$ holds. If $\bar{\mu} \in (\mu_n , \mu_{n+1})$ for some $n \geq 1$, and the sum $\sigma_n =  \sum \limits_{i=1}^{ n} \dim \mathbf{E}(\mu_i)$ is odd, then there exists a positive constant $D_2$ such that, if $d_2 \geq D_2$,
then the problem (\ref{nonhomo}) has at least one nonconstant positive steady state to generate stationary patterns.
\end{theorem}
\begin{proof}
If $a_{11} a_{33}<0,$ by proposition \ref{propo2}, here, we have a positive constant $D_2$, such that when $d_2 \geq D_2 $, (\ref{moreover}) holds and
 \begin{equation}\label{proof}
 0 = \mu_0 < \tilde{\mu}_2(d_2)<\mu_1~~~~\tilde{\mu}_3(d_2) \in (\mu_n , \mu_{n+1}).
 \end{equation}
The motivation of this theorem is to prove the existence of nonconstant positive steady state of (\ref{nonhomo}) for any $d_2 \geq D_2$. This proof is based on the homotopy invariance of the topological degree.
Assume that the attestation on the contrary is not valid for some $d_2 = \tilde{d}_2 \geq D_2$. To the continuation of proof, we fixed $d_2 = \tilde{d}_2$,
\[d^{*}_1 = \dfrac{r}{\mu_1},~~d^{*}_2 = \dfrac{\sigma w'(c_1 + c_2 \beta)}{\mu_1} .\]
By Theorem \ref{theoremnon}, we obtain $D_3 = D_3(\Lambda,d^{*}_1,d^{*}_2)$. Fix $\hat{d}_2 \geq d^{*}_2,\hat{d}_1 \geq \max \{d^{*}_1,d_1\},\hat{d}_3 \geq \max\{D_3,d_3\}$. For $ t \in [0,1] $, define $\mathcal{D}(t) = \text{diag}(d_1(t),d_2(t),d_3(t))$ with $d_{i}(t) = td_i + (1-t)\hat{d}_i$, $i=1,2,3$. Now, consider the problem

\begin{equation}\label{diagonal}
\begin{cases}
-\mathcal{D}(t)\Delta \textbf{u} = \textbf{G}(\textbf{u}),~~~~x \in \Omega,\\
\partial_{\nu} \textbf{u}=0, ~~~~~~~~~~~~~~~~x \in \partial \Omega. 
\end{cases}
\end{equation}
Then \textbf{u} is a nonconstant positive steady state of (\ref{nonhomo}) if it is a positive solution of (\ref{diagonal}) for $t = 1$. It is obvious
that ${\textbf{u}}^{*}$ is the unique constant positive solution of (\ref{diagonal}) for any $0 \leq t \leq 1$. For any $0 \leq t \leq 1$, \textbf{u} is a positive solution of (\ref{diagonal}) if
\[\textbf{F}(t;\textbf{u})   \triangleq \textbf{u} - (\textbf{I} - \Delta)^{-1}\{\mathcal{D}^{-1}(t)\textbf{G}(\textbf{u}) + \textbf{u}\} = 0~~~~\text{in}~~{\mathbf{X}}^{+}.\]
Clearly, $\textbf{F}(1; \textbf{u}) = \textbf{F}(\textbf{u})$, Theorem \ref{theoremnon} shows that $\textbf{F}(0; \textbf{u}) = 0$ has only the positive solution ${\textbf{u}}^{*}$ in ${\mathbf{X}}^{+}$. The direct calculation gives,

\[\textbf{D}_\textbf{u}\textbf{F}_{\textbf{u}}(t;{\textbf{u}}^{*})   \triangleq \textbf{I} - (\textbf{I} - \Delta)^{-1}\{\mathcal{D}^{-1}(t){\textbf{G}}_{\textbf{u}}({\textbf{u}}^{*}) + \textbf{I}\},\]
and in particular,
\[\textbf{D}_\textbf{u}\textbf{F}_{\textbf{u}}(0;{\textbf{u}}^{*})   \triangleq \textbf{I} - (\textbf{I} - \Delta)^{-1}\{\widehat{\mathcal{D}}^{-1} {\textbf{G}}_{\textbf{u}}({\textbf{u}}^{*}) + \textbf{I}\},\]

\[\textbf{D}_\textbf{u}\textbf{F}_{\textbf{u}}(1;{\textbf{u}}^{*})   \triangleq \textbf{I} - (\textbf{I} - \Delta)^{-1}\{\mathcal{D}^{-1} {\textbf{G}}_{\textbf{u}}({\textbf{u}}^{*}) + \textbf{I}\}=\textbf{D}_{\textbf{u}}\textbf{F}({\textbf{u}}^{*}),\]
\\
where $\widehat{\mathcal{D}} = \text{diag}(\hat{d}_1,\hat{d}_2,\hat{d}_3)$. From (\ref{hfunction}) and (\ref{amu})
\begin{equation}\label{hmu_1}
\mathbb{H}(\mu)  = \dfrac{1}{d_1 d_2 d_3 }\mathcal{A}(d_2;\mu)
\end{equation}
In view of (\ref{moreover}) and (\ref{proof}), it follows from (\ref{hmu_1}) that
\begin{equation}
\begin{cases}
\mathbb{H}(\mu_0) = \mathbb{H}(0) >0,\\
\mathbb{H}(\mu_i)<0,~~1 \leq i \leq n,\\
\mathbb{H}(\mu_{i+1}) >0,~~~i \geq n+1
\end{cases}
\end{equation}
Therefore, 0 is not an eigenvalue of the matrix $\mu_i I - \mathcal{D}^{-1}{\textbf{G}}_{\textbf{u}}({\textbf{u}}^{*})$ for all $i \geq 0$, and
\[\sum_{\substack{
i \geq 0,\mathbb{H}(\mu_i)<0
}}\dim \mathbf{E}(\mu_i) = \sum \limits_{i=1}^{n} \dim \mathbf{E}(\mu_i) = \sigma_n,~~~~~~\text{which is odd.}\]
Thanks to Proposition \ref{propos1}, we have
\begin{equation}\label{last3}
\text{index}(\textbf{F}(1;\cdot),{\textbf{u}}^{*}) = (-1)^{\rho} = (-1)^{\sigma_n} = -1
\end{equation}
Similarly, it is possible to prove
\begin{equation}\label{last2}
\text{index}(\textbf{F}(0;\cdot),{\textbf{u}}^{*}) = (-1)^{0} = 1
\end{equation}
 
By Theorems \ref{theoremupper} and \ref{theoremlower}, there exists a positive constant $\Theta$ such that, for all $0 \leq t <1$, the positive solutions of (\ref{diagonal}) satisfy $\Theta^{-1}<u,v,w<\Theta$. Therefore $\textbf{F}(t;\textbf{u})$ on $\partial \mathbb{B}(\Theta)$ for all $\Theta^{-1}<u,v,w<\Theta$. By homotopy invariance of the topological degree,
\begin{equation}\label{final}
\text{deg}(\textbf{F}(1;\cdot),0,\mathbb{B}(\Theta)) = \text{deg}(\textbf{F}(0;\cdot),0,\mathbb{B}(\Theta))
\end{equation}
On the other hand, by our assumption, both equations $\textbf{F}(1;\textbf{u}) = 0$ and $\textbf{F}(0;\textbf{u}) =0$ have only the positive solution ${\textbf{u}}^{*}$ in $\mathbb{B}(\Theta)$, and hence, by (\ref{last3}) and (\ref{last2}),

\begin{eqnarray*}
&&\text{deg}(\textbf{F}(0;\cdot),0,\mathbb{B}(\Theta)) = \text{index}(\textbf{F}(0;\cdot),{\textbf{u}}^{*}) = 1,\\
&& \text{deg}(\textbf{F}(1;\cdot),0,\mathbb{B}(\Theta)) = \text{index}(\textbf{F}(1;\cdot),{\textbf{u}}^{*}) = -1. 
\end{eqnarray*}
This contradicts (\ref{final}) and the proof is completed.

\end{proof}

 From the above theorem, we conclude that if the susceptible predator diffuses with a sufficiently large rate, then the model system (\ref{nonhomo}) has at least one nonconstant steady state, and hence stationary patterns started to appear.

 Analytical results obtained in Theorems \ref{theoremnon} and \ref{theoremconst} are sufficient conditions. These results show that stationary patterns obtained for certain values of diffusion coefficients. For the numerical validation of results, calculation of smallest eigenvalue $\mu_1$ of $-\Delta$ is important but not easy to determine. Hence the conditions for nonexistence and existence of nonconstant positive steady states cannot be verified for the suitable choice of parameter values. Therefore in the next section, we have discussed the Turing instability, which gives parametric conditions for Turing patterns to arise.

\section{Turing Instability}\label{tstability}
 In this section, our primary goal is to understand the diffusion induced instabilities in the system generated by random movement of the species. This instability leads to Turing patterns in the spatiotemporal system. The existence of Turing patterns can be proved by Turing instability. Segel and Jackson \cite{segel1972dissipative} were the first who paid attention to Turing instability in the ecological context. 
 
 We perform a detailed study of the local dynamics of the spatiotemporal model system (\ref{eq1}) in a two-dimensional spatial domain. The conditions for the Turing instability is obtained by perturbing the homogeneous steady-state solution. For this, we linearized the spatiotemporal model about the interior equilibrium point $E^*(u^{*},v^{*},w^{*})$ and perturbed the system with following two-dimensional perturbation \cite{kumari2013pattern} of the form:
\begin{eqnarray}
u&=&u^* + \epsilon_1 \exp(\lambda_k t + i(k_x x + k_y y)) = u^* + u_1\\
v&=&v^* + \epsilon_2 \exp(\lambda_k t + i(k_x x + k_y y))=v^* + v_1\\
w&=&w^* + \epsilon_3 \exp(\lambda_k t + i(k_x x + k_y y))=w^* + w_1
\end{eqnarray}

 where $\epsilon_i,i=1,2,3$, are positive and sufficiently small constants, $k_x$ and $k_y$ are the components of wave number $k=\sqrt{k^2_x + k^2_y}$ along $x$ and $y$ directions respectively and $\lambda_k$ is the wavelength. The system is linearized about the non-trivial interior equilibrium point $E^*(u^*,v^*,w^*)$. The characteristic equation of the linearized version of the spatiotemporal model system (\ref{eq1}) is given by 
 
\begin{equation}\label{eq8}
\begin{pmatrix} J_{s} - \lambda_k I_3  \end{pmatrix} 
\begin{pmatrix} u_1 \\ v_1 \\ w_1  \end{pmatrix}
 = 0,
\end{equation}
 with 
 \begin{equation}\label{eq9}
 J_{xyz}
 =
\begin{pmatrix} a_{11}-d_1 k^2 & a_{12} & a_{13}\\
                a_{21}  &  a_{22}-d_2 k^2     & a_{23}\\
                a_{31}  &  a_{32}      & a_{33} - d_3 k^2
                
                   \end{pmatrix}, 
\end{equation}\\
 where $k$ is the wave number given by $k^2 = k_X^2 + k_Y^2$ and $I_3$ is a $3 \times 3$ identity matrix and $a_{ij}$ are given in (\ref{jacobian}). Our interest is to check the stability properties of the attracting interior equilibrium point $E^*$, leading to the conditions causing Turing instability. From (\ref{eq8}) to (\ref{eq9}), we get the characteristic equation of the form
\begin{equation}\label{char}
\det(J_{xyz}-\lambda_k I_3)=\lambda_k^3 + \rho_1 \lambda_k^2 + \rho_2 \lambda_k + \rho_3 = 0,
\end{equation}

  where,
 \begin{eqnarray*}
 \rho_1 &=& -\text{tr}(J_{xyz}) = - (a_{11}+a_{22}+a_{33}) + (d_1 + d_2 +d_3)k^2,\\
 \rho_2 &=& k^4(d_1 d_2 + d_2 d_3 + d_3 d_1)-\{a_{11}(d_2 + d_3) + a_{22} (d_3 + d_1)\\
 &~& + a_{33} (d_1 + d_2)\} k^2 +(a_{11}a_{33}+a_{22}a_{33}+a_{11}a_{22})-a_{12}a_{21}-a_{23}a_{32},\\
 \rho_3 &=& -\det(J_{xyz})=d_1 d_2 d_3 k^6 - k^4 (a_{11} d_2 d_3 + a_{22} d_1 d_3 + a_{33} d_1 d_2) \\
 &~& + k^2 (a_{11}a_{33} d_2 + a_{22} a_{33} d_1 + a_{11}a_{22} d_3 - d_3 a_{12}a_{21} - d _1 a_{23}a_{32} ) \\
 &~& + (a_{32}a_{11}a_{23}+a_{33}a_{12}a_{21}-a_{21}a_{32}a_{13}-a_{11}a_{22}a_{33}).
 \end{eqnarray*}
 
  \noindent Hence from Routh-Hurwitz criterion, the spatiotemporal model system (\ref{eq1}) is stable if 
\begin{equation}\label{locond}
\rho_1(0) >0 ,\rho_3(0) >0~~\text{and}~~\rho_1(0)\rho_2(0)-\rho_3(0)>0
\end{equation}
\begin{theorem}
If the positive equilibrium point $E^{*}({u}^{*},{v}^{*},{w}^{*})$ is locally asymptomatically stable for temporal model system (\ref{main}), then $E^{*}({u}^{*},{v}^{*},{w}^{*})$ is locally asymptomatically stable for spatiotemporal model system (\ref{eq1}) if condition (\ref{locond}) holds.
\end{theorem}
\begin{proof}
The proof of theorem directly follows from Routh-Hurwitz criterion, hence omitted.
\end{proof}
 
\begin{theorem}\label{turingtheorem}
Diffusion-driven instability occurs if one of the following conditions is satisfied:
\begin{enumerate}
\item[(i)]If $q_2<0$ and $q_2^2 - 3q_1q_3>0$ then $k_{d}^2$ is positive and real, where
\[k_{d}^2 =\dfrac{-q_2 + \sqrt{q_2^2 - 3q_1 q_3}}{3q_1} ~\text{and}~\]
\[\rho_3(k_{d}^2) = \dfrac{2q_2^3 - 9q_1 q_2 q_3 + 27 q_1^2 q_4 - 2(q_2^2 -3q_1 q_3 )^{\frac{3}{2}}}{27q_1^2} <0.\]
\item[(ii)]If $r_2<0$ and $r_2^2 - 3r_1r_3>0$ then $k_{f}^2$ is positive and real, where
\[k_{f}^2 =\dfrac{-r_2 + \sqrt{r_2^2 - 3r_1 r_3}}{3r_1} ~\text{and}\]
\[~\Phi(k_{f}^2) = \dfrac{2r_2^3 - 9r_1 r_2 r_3 + 27 r_1^2 r_4 - 2(r_2^2 -3r_1 r_3 )^{\frac{3}{2}}}{27r_1^2} <0.\]
\end{enumerate}
\end{theorem}
\begin{proof}
The spatially homogeneous state will be unstable provided that at least one eigenvalue of the characteristic equation (\ref{char}) is positive. It is clear that the homogeneous steady state $E^{*}$ is asymptotically stable when $\rho_1(0) >0 , \rho_3(0) >0$ and $\rho_1(0)\rho_2(0)-\rho_3(0)>0$. But it will be driven to an unstable state by diffusion if any of the conditions $\rho_1(k^2) >0 , \rho_3(k^2) >0$ and $\rho_1(k^2)\rho_2(k^2)-\rho_3(k^2)>0$ fail to hold. However, it can be easily seen that diffusion-driven instability cannot occur by contradicting $\rho_1(k^2)>0$ because $\delta_1,\delta_2,\delta_3$ and $k^2$ are positive, the inequality $\rho_1(k^2)>0$ always holds as $a_{11}+ a_{22}<0$ from the stability conditions of the interior equilibrium point in homogeneous state. So, for the diffusion driven instability, it is sufficient if any of the following condition holds:
\[\rho_3(k^2)<0 ~\text{and}~ \rho_1(k^2)\rho_2(k^2)-\rho_3(k^2)<0\]
For $\rho_3(k^2)<0$, we can write it from the equation (\ref{char}) and let us suppose $z = k^2$, so we get 
 
\begin{equation}\label{cond2}
\rho_3(z) = q_1 z^3 + q_2 z^2 + q_3 z + q_4, 
\end{equation}
 where,
\begin{eqnarray*}
 q_1&=&d_1 d_2 d_3,\\
 q_2&=&-(a_{11} d_2 d_3 + a_{22} d_1 d_3 + a_{33} d_1 d_2),\\
 q_3&=&a_{11}a_{33} d_2 + a_{22} a_{33} d_1 + a_{11}a_{22} d_3 - d_3 a_{12}a_{21} - d _1 a_{23}a_{32}, \\ 
q_4&=&a_{32}a_{11}a_{23}+a_{33}a_{12}a_{21}-a_{21}a_{32}a_{13}-a_{11}a_{22}a_{33}. 
 \end{eqnarray*}
 For the minimum of $\rho_3(z)$ we have $\dfrac{d \rho_3}{dz}=0$ which gives $3q_1 z^2 + 2q_2 z + q_3 = 0$ has two roots $z_{1,2}=\dfrac{-q_2 \pm \sqrt{q_2^2 - 3q_1 q_3}}{3q_1}$. We find $\dfrac{d^2 \rho}{d z^2} =\pm 2\sqrt{q_2^2 - 3q_1 q_3}$ so minimum at $k_{d}^2 =\dfrac{-q_2 + \sqrt{q_2^2 - 3q_1 q_3}}{3q_1}$ and the minimum value is 
\[\rho_3(k_{d}^2) = \dfrac{2q_2^3 - 9q_1 q_2 q_3 + 27 q_1^2 q_4 - 2(q_2^2 -3q_1 q_3 )^{\frac{3}{2}}}{27q_1^2} <0.\]
So for the instability we need to show that  
$q_2<0$ and $q_2^2 - 3q_1q_3>0$. Now for the third condition $\rho_1(k^2)\rho_2(k^2)-\rho_3(k^2)<0$ we take $ \Phi(k^2)=\rho_1(k^2)\rho_2(k^2)-\rho_3(k^2)$, for some wave number $k\geq 0$. From the equation (\ref{char}) we get,
\begin{equation}\label{cond3}
\Phi(z) = r_1 z^3 + r_2 z^2 + r_3 z + r_4,
\end{equation}
where,
\begin{eqnarray*}
r_1&=&bp_1-q_1,\\
r_2&=&bp_2-ap_1-q_2,\\
r_3&=&bp_3-ap_2-q_3,\\
r_4&=&-(ap_3+q_4).
\end{eqnarray*}
For the suitability in calculation we choose $a$ and $b$ such that $a=a_{11} + a_{22} + a_{33}$ and $b=d_1 + d_2 + d_3$ for the minimum of $\Phi(z)$ we have $\dfrac{d \Phi}{dz}=0$ which gives $3r_1 z^2 + 2r_2 z + r_3 = 0$ has two roots $z_{1,2}=\dfrac{-r_2 \pm \sqrt{r_2^2 - 3r_1 r_3}}{3r_1}$. We find $\dfrac{d^2 \Phi}{d z^2} =\pm 2\sqrt{r_2^2 - 3r_1 r_3}$ so minimum at $k_{f}^2 =\dfrac{-r_2 + \sqrt{r_2^2 - 3r_1 r_3}}{3r_1}$ and the minimum value is $\Phi(k_{f}^2) = \dfrac{2r_2^3 - 9r_1 r_2 r_3 + 27 r_1^2 r_4 - 2(r_2^2 -3r_1 r_3 )^{\frac{3}{2}}}{27r_1^2} <0$. So for the instability we need to show that $r_2<0$ and $r_2^2 - 3r_1r_3>0$.\\
\end{proof}
 \begin{table}[htbp]\caption{Validation of Turing patterns}
 \centering
\begin{tabular}{cccccc}
\hline
Figs. & $k$ & $\rho_1$ & $\rho_2$ & $\rho_3$ & $\rho_1\rho_2-\rho_3$ \\ \hline
\ref{table8}(A),\ref{table8}(B) & 0	& 0.0942	& 0.0297 &	0.0024	& 0.0004 \\
\ref{table8}(A) & 15 &	0.3215 &	0.0455 &	-0.0019 &	0.0165 \\
\ref{table8}(B) & 15 & 0.3194 & 0.0446 & -0.0020 & 0.0162\\\hline
\end{tabular}
\label{turingtable}
\end{table}
The existence of Turing pattern can be checked by planar stability, at wave number $k=0$, $\rho_n(k=0)>0, n=1,2,3,$ and $\rho_1(0) \rho_2(0)>\rho_3(0)$; however there exist atleast one non-zero mode such that one or more of the conditions $\rho_n(k)>0,\rho_1(k) \rho_2(k)>\rho_3(k)$, are violated. Table \ref{turingtable} clearly shows the patterns in Fig. \ref{table8} are Turing patterns because at $k=0$ all $\rho_n >0$ and at non-zero mode $k=15,~\rho_3$ is negative for both \ref{table8}(A) and \ref{table8}(B). Here in two-dimensional case, $k^2 = k^2_x + k^2_y$, we assume $k_y=0$, and obtained the results in Table \ref{turingtable}.
 
 \section{Numerical Simulation}\label{simulation}
Here we present the numerical simulation results for both temporal and spatiotemporal model systems to support the analytical findings. The temporal model system (\ref{main}) has been solved using fourth-order Runga-Kutta method, and spatiotemporal model system (\ref{eq1}) is solved using Finite difference technique. The simulation is done using MATLAB, and for this, we have considered the following set of parameter values:
\begin{equation}\label{paraeq}
\begin{array}{c}\vspace{0.2cm}
r = 0.4, K = 68,\alpha_1 = 0.3,\alpha_2=0.1,\gamma = 10,\alpha = 1,c_2 = 0.2, l = 0.08,\\
f = 10,\lambda = 0.003,\sigma=0.005, d = 0.02, e = 0.01,\beta = 0.5,c_1 = 1.
\end{array}
\end{equation} 
The parameter values given in (\ref{paraeq}) are used to obtain non-Turing patterns, whereas the parameter values for the calculation of Turing patterns are same as given in (\ref{paraeq}) with $\sigma=0.026$. Our goal is to obtain the effect of cannibalism and disease on the dynamics of proposed system. For this, in the subsection \ref{effectofdisease}, we have observed the effect of disease on system dynamics and the effect of cannibalism is shown in subsection
\ref{effectofcannib}.

\subsection{Effects of Disease Transmission Rate, $\lambda$}\label{effectofdisease}
To observe the effect of disease Transmission Rate, $\lambda$, we perform numerical simulations. For this purpose, we choose a parameter value set given in (\ref{paraeq}). It is interesting to note that, if we take disease transmission rate ($\lambda$) as zero, then no positive equilibrium point exists and with any positive initial value we did not find any patterns (see Fig. \ref{nopatterns1}). However, for the biologically feasible case, $\lambda$ should be non-zero. Therefore we plot a bifurcation diagram with bifurcation parameter as $\lambda$. Here we consider non-zero values of $\lambda$. The range of $\lambda$ considered for bifurcation diagram is $0 \leqslant \lambda \leqslant 0.02.$ The bifurcation diagram is presented in Fig. \ref{bifurcation_diagram1}. Here we observe that once $\lambda \geqslant 0.007$ then middle predator settles down to a constant population and shows a stable dynamics for the model system. In particular, we are interested in observing the dynamics of the model system at the parameter set where non-Turing patterns are seen, i.e., $\lambda =0.003$. The  Lyapunov exponents are calculated corresponding to these dynamics. Fig. \ref{spatiodynamics} depicts the phase space, time-series and Lyapunov spectrum. The system (\ref{main}) has two positive and one negative Lyapunov exponent ($L_1 = 0.0132701,L_2 = 0.00165054$ and $L_3 = -0.0454853$). According to Weifeng Shi \cite{shi2007lyapunov}, system with two positive and one negative Lyapunov exponent (+,+,-), has unstable limit cycle. Hence the model system (\ref{main}) has an unstable limit cycle and it can also be verified with the direction field presented in Fig. \ref{spatiodynamics} (b). 

 \subsection{Effect of Cannibalistic Attack Rate, $\sigma$}\label{effectofcannib}
To observe the effect of the cannibalistic attack rate, $\sigma$, we choose a parameter value set given in (\ref{paraeq}). In the absence of cannibalism, i.e., when $\sigma=0$ and other parameters are fixed as given in (\ref{paraeq}), no patterns are observed as can be seen from Fig. \ref{nopatterns}.  The bifurcation diagram is presented in Fig. \ref{bifurcation_diagram}, from which we observe that when $\sigma \geqslant 0.026$, a stable dynamics is seen. When we increase the value of $\sigma$ from 0.005  to 0.026, then model system shows stable focus dynamics at the initial value (15.1342,20.5234,6.3140) as shown in Fig. \ref{stablefocusplot}. We have calculated Lyapunov spectrum corresponding to this stable dynamic and see that all three Lyapunov exponents are negative ($L_1 = -0.00728561,L_2 = -0.00842272$ and $L_3$ = -0.0802341). It shows that the system (\ref{main}) has a stable fixed point.

 \subsection{Turing patterns}
Turing instability leads to stationary patterns which are also called Turing patterns \cite{peng2007stationary}. However, patterns can also be observed due to the effect of Hopf bifurcation and its referred as non-Turing patterns. In Figs. \ref{regiont}(a) - \ref{regiont}(c), the regions are specified where these patterns have emerged. The yellow colored region presents Turing instability, green colored region shows planar stability and brown colored regions are the regions where non-Turing patterns have emerged. Moreover, the red line represents Hopf bifurcation which separates the regions for Turing and non-Turing patterns (Fig. \ref{regiont}(a) and \ref{regiont}(b) ). There is no mechanism developed to classify the non-Turing patterns. Still, these patterns occur for those specific choices of parameter values where the fixed point becomes unstable \cite{huang2019exploring}.

This subsection focuses on understanding the spatial interaction between the species through pattern formation analysis for specific parameter values. To perform pattern formation analysis, we have considered a square domain with a zero-flux boundary condition and applied the FTCS numerical scheme to solve the spatiotemporal model system (\ref{eq1}). For the reaction part, we have used the forward difference scheme while standard five-point explicit finite difference scheme is used for the diffusion part. The initial distribution is considered with small perturbation of the form $ 0.1.\cos^2(10x)$ $\cos^2(10y)$ about the steady state. Here, we have obtained Turing patterns for different parameter values as given in Table \ref{paratable}. 
\begin{table}[!ht]
\centering
\caption{Parameter values used in the simulation}
\begin{tabular}{@{}ccccccc@{}}
\toprule
Figs.                    &  Type of Patterns        & $d_1$     & $d_2$     & $d_3$     & {$\sigma$} & Rest of the parameters                                                                                                                                                                                                                                                                                                                                                       \\ \midrule
\ref{table8}(A) & Turing patterns  & $10^{-5}$ & $10^{-3}$ & $10^{-10}$ & 0.026        & \multirow{4}{*}{\begin{tabular}[c]{@{}c@{}}$r = 0.4, K = 68,\alpha_1 = 0.3,$ \\  $\alpha_2=0.1,\gamma = 10,\alpha = 1,$ \\ $c_2 = 0.2, l = 0.08,f = 10,$ \\  $ \lambda = 0.003,d = 0.02, e = 0.01.$ \\ $\beta = 0.5,c_1 = 1.$\end{tabular}}        \\
\ref{table8}(B) & Turing patterns & $10^{-6}$ & $10^{-3}$ & $10^{-10}$ & 0.026        &                                                                                                                                                                                                                                                                                                                                                                              \\
\ref{table9}-\ref{table91}  & Non-stationary non-Turing patterns & $10^{-6}$ & $10^{-6}$ & $10^{-10}$ & 0.005        &                                                                                                                                                                                                                                                                                                                                                                        \\
\ref{table10}  & Stationary non-Turing patterns & $~10^{-10}$ & $10^{-4}$ & $10^{-10}$ & 0.005        &                                                                                                                                                                                                                                                                                                                                                                        \\
\ref{table11}  & Stationary non-Turing patterns & $~10^{-10}$ & $10^{-6}$ & $10^{-10}$ & 0.005       &                                                                                                                                                                                                                                                                                                                                                                        \\ \bottomrule
\end{tabular}
\label{paratable}
\end{table}

 Turing patterns are presented in Figs. \ref{table8}(A) - \ref{table8}(B). In each snapshot, the red colour represents the high population density, while the blue colour represents the low population density.
The Turing patterns are obtained when we increase the value of cannibalistic attack rate ($\sigma$) from 0.005 to 0.026. Earlier, we have shown the existence of Turing patterns through Turing instability analysis in section \ref{tstability}. The interior equilibrium point $(u^{*},v^{*},w^{*})$ is obtained using the values of the parameters given in (\ref{paraeq}) with $\sigma=0.026$ is (8.1844,19.0716,6.7682).  Here, we perform simulation over $[0,\pi] \times [0,\pi]$ with spatial resolution $\Delta x = \Delta y  = 0.01$ and time step $\Delta t = 0.01$. The snapshots of Fig. \ref{table8}(A), shows distribution of the species in three columns namely prey, susceptible predator and infected predator respectively, with diffusive rates $10^{-5},10^{-3}$ and $10^{-10}$. Fig. \ref{table8}(B) represents Turing patterns at diffusive rates $10^{-6},10^{-3}$ and $10^{-10}$ for $u, v$ and $w$ respectively, in two-dimensional space. The simulation has also been carried out for higher time steps i.e. $t=15000$, however there were no change seen in patterns after $t=1000.$ Hence the Turing patterns obtained at $t=1000$, is presented in Figs. \ref{table8}(A) and \ref{table8}(B).

   \subsection{Stationary and non-stationary non-Turing patterns}
 In this subsection, we obtain the non-Turing patterns, the emergence of which does not satisfy the Turing instability conditions \cite{huang2019exploring}. Moreover, the non-Turing patterns merely appear when the fixed point $E^{*}({u}^{*},{v}^{*},{w}^{*})$ is unstable, i.e., the conditions given in (\ref{locond}) are not satisfied. The set of computations for which two dimensional non-Turing patterns are calculated are given in Table \ref{paratable}. For the set of parameter values, $\rho_1(0)=-0.01156<0 ,\rho_3(0)=0.0005>0$ and $\rho_1(0)\rho_2(0)-\rho_3(0)=-0.0007<0$. Hence the positive equilibrium (1.9756,13.4643,6.1178) is unstable. Here, we have observed two types of non-Turing patterns, stationary and non-stationary patterns. In the spatial distribution, fixed or time-invariant patches correspond to stationary patterns, whereas the patches changing with time corresponds to non-stationary patterns. From the above observations, we can characterize the non-Turing patterns. Patterns obtained in Fig. \ref{table9} to \ref{table91} are non-stationary non-Turing patterns as they are changing continuously with time. In Figs. \ref{table10} to \ref{table11}, the patterns are stationary non-Turing patterns as they being fixed after some time. Patterns are also classified in Table \ref{paratable}. The snapshots of Figs. \ref{table9}-\ref{table91}, shows the distribution of the species in three columns namely prey, susceptible predator and infected predator respectively with their diffusive rates $d_1=10^{-6},d_2=10^{-6}$ and $d_3=10^{-10}$ respectively in two-dimensional space. Each population exhibits different spatial patterns at initial stage, i.e. at $t=200,300,500,1000$ (see Fig. \ref{table9}). However, from $t=2000$ onwards spatial distribution of population at three levels are very similar to each other (see Fig. \ref{table91}). Still the patterns keep changing at different time levels   each after $t=2000$ (as seen from Figs. \ref{table9}-\ref{table91}). At $t=200,300,500,1000$, non-Turing patterns are rectangular and mixed spot patterns while at $t=1500$ prey and susceptible predator species exhibit star and spot patterns as seen from Fig. \ref{table9}. In Figs. \ref{table9}(A) and \ref{table9} (B) the population displays hot spot patterns. As time is increased the patterns change into cold spots. Further in Fig. \ref{table91}, all the parameter values remain same as in Fig. \ref{table9}, however when we increase the simulation time, the non-Turing patterns keep changing. At each stage, pattern changes, which shows instability in the system behavior. Now we increase the diffusive rate of susceptible predator and decrease the diffusive rate of prey as $d_1=10^{-10},d_2=10^{-4}$ and $d_3=10^{-10}$, the spatial pattern hence obtained is shown in Fig. \ref{table10}. At $t=500$ patterns start to appear and changes until $t=5000$. We further decrease the diffusive rate of susceptible predator population as $d_1 = 10^{-10}, d_2 =10^{-6}$ and $d_3 = 10^{-10}$. The corresponding patterns are presented in Fig.\ref{table11}. It is observed that each individual species exhibit different spatial patterns at initial stages, i.e. $t=500,1000,1500$ but after $t=1500$, the population distribution was almost the same.

 The pattern formation analysis has been done rigorously in this section. From this analysis, we observe that movement of the prey and susceptible predator plays a significant role in the spatial distribution of the species. Initially, Turing patterns exist in the proposed eco-epidemiological model, as seen in Figs. \ref{table8}(A) and \ref{table8} (B). However, as the cannibalistic rate increased and the movement of the prey and susceptible predator is decreased then the Turing patterns disappear, and a new type of pattern appear called non-Turing patterns. Also, we observe that as the movement of the prey is substantially decreased from $10^{-6}$ to $10^{-10}$ the patterns converge to a nearly stationary non-Turing pattern (see the transition from Figs. \ref{table91} to \ref{table11}). However, the time taken for patterns to become stationary is relatively less, i.e. it becomes stationary at $t=2000.$ However, when the movement of prey is decreased, and movement of susceptible predator is increased (see the transition from Figs. \ref{table91} to \ref{table10}), then it takes longer time for non-Turing patterns to become stationary, i.e. in this case patterns become stationary at $t=5000.$ This shows that the existence of irregular distribution of population in space largely depends upon the movement of prey and susceptible predator. Self-organized patterns are seen in all the population when the movement of prey and the susceptible predator is controlled.

\section{Conclusion}\label{discuss}
In this work, we have proposed and analyzed a diffusive eco-epidemiological model where the temporal model is already studied in literature \cite{biswas2018cannibalistic}. In this model, it is assumed that disease spreads into predator populations only. Due to this, the predator population has been divided into two classes: susceptible predator $v$ and infected predator $w$. Since the predator species are cannibalistic, the disease also spreads via cannibalism among them. Here we extend the temporal model system (\ref{main}) into spatiotemporal form (\ref{eq1}) to understand the dynamics when the species perform active movements into both $x$ and $y$ directions. Movement is possible due to social interaction, searching for food and finding mates. We have studied the model system (\ref{eq1}), both analytically and numerically. The work done in this paper can be summarized as under:
\begin{enumerate}
\item From Theorems \ref{localtheorem} and \ref{globaltheorem}, sufficient conditions for local asymptotic stability and for global stability for the constant positive solution has been obtained by using linearization and Lyapunov method technique.
\item Conditions on diffusive rates that are responsible for the existence and non-existence of stationary patterns are obtained.
\item From Theorem \ref{theoremnon}, it can be concluded that under certain conditions, there is no nonconstant positive steady state, provided diffusion coefficient $d_1$ and $d_2$ are fixed and $d_3$ is sufficiently large this implies that variation in the diffusive rate of infected predator prevents the emergence of patterns, i.e. self-organization of population.
\item From Theorem \ref{theoremconst}, we can say that at least one constant positive solution exists provided diffusion coefficient $d_1$ and $d_3$, are fixed and $d_2$ is kept sufficiently large. This tells us that the diffusive rates of susceptible predator are responsible for Turing pattern formation. 
\item Turing instability in Theorem \ref{turingtheorem} ensures the existence of Turing patterns. It is obtained by perturbing the homogeneous steady-state solution.
\item There are two sensitive parameters of the model system (\ref{main}), which are the disease transmission rate ($\lambda$) and cannibalistic attack rate ($\sigma$). The model exhibits various nonlinear dynamics, such as stable focus and unstable limit cycle when these sensitive parameters are changed.
\item Lyapunov spectrum has been calculated to understand the dynamical behavior of the temporal system at different parameter values set. We observed unstable and stable dynamics.
\item Non-stationary patterns obtained due to the unstable limit cycle observed in the temporal system. 
\item In the absence of cannibalism and disease, the distribution of the species was not self-organized. That means no spatial patterns were seen.
\item Increment in the cannibalistic attack rate and disease transmission rate promotes Turing patterns.
\item  The movement of the prey and susceptible predator plays a significant role in the spatial distribution of the species. The movements cause stationary and non-stationary patterns.
\end{enumerate} 

\noindent The constant positive steady state is obtained by solving the system of nonlinear equations given in (\ref{eq2}) and the condition for the existence is given in (\ref{existenceequi}). From Theorems \ref{localtheorem} and \ref{globaltheorem}, it is clear that the constant positive solution $\textbf{u}^{*}$ is locally and globally stable under certain parametric restriction which implies the absence of stationary patterns. Further, we obtained the conditions for the nonexistence and existence of the nonconstant positive steady states of the steady state system (\ref{nonhomo}). From Theorems, \ref{theoremnon} and \ref{theoremconst}, we can conclude that large diffusivity may be one of the reasons for the appearance and non-appearance of stationary patterns. The discussion on Turing instability in Theorem \ref{turingtheorem} ensures the existence of stationary patterns.\\
\indent Lyapunov exponents are calculated at different parameter sets, and we observe that the system shows unstable as well as stable dynamics. It is observed that when we increase the movement of the susceptible predator by increasing the diffusivity coefficient, then all the population show spatial self-organization. When we increase the cannibalistic attack rate and also increase the movement of the susceptible predator, then spatiotemporal distribution converges to stationary patterns. In the absence of cannibalism and disease, the system fails to exhibit any stationary patterns while the introduction of the cannibalistic attack rate and disease transmission rate ensure stationary patterns.
 \section*{Acknowledgments}
  The research of first author (VK) is supported by CSIR, India [09/1058(0006)/2016-EMR-I].  The research of second author (NK) is funded by Science Engineering Research Board (SERB), under three separate grants 
[MTR/2018/000727], [EMR/2017/005203] and [MSC/2020000369]. 
 \bibliographystyle{unsrt}  
 
\newpage

 \renewcommand{\thefigure}{\arabic{figure}}
 \begin{figure} [h!]
 \centering
 \includegraphics[scale=0.32]{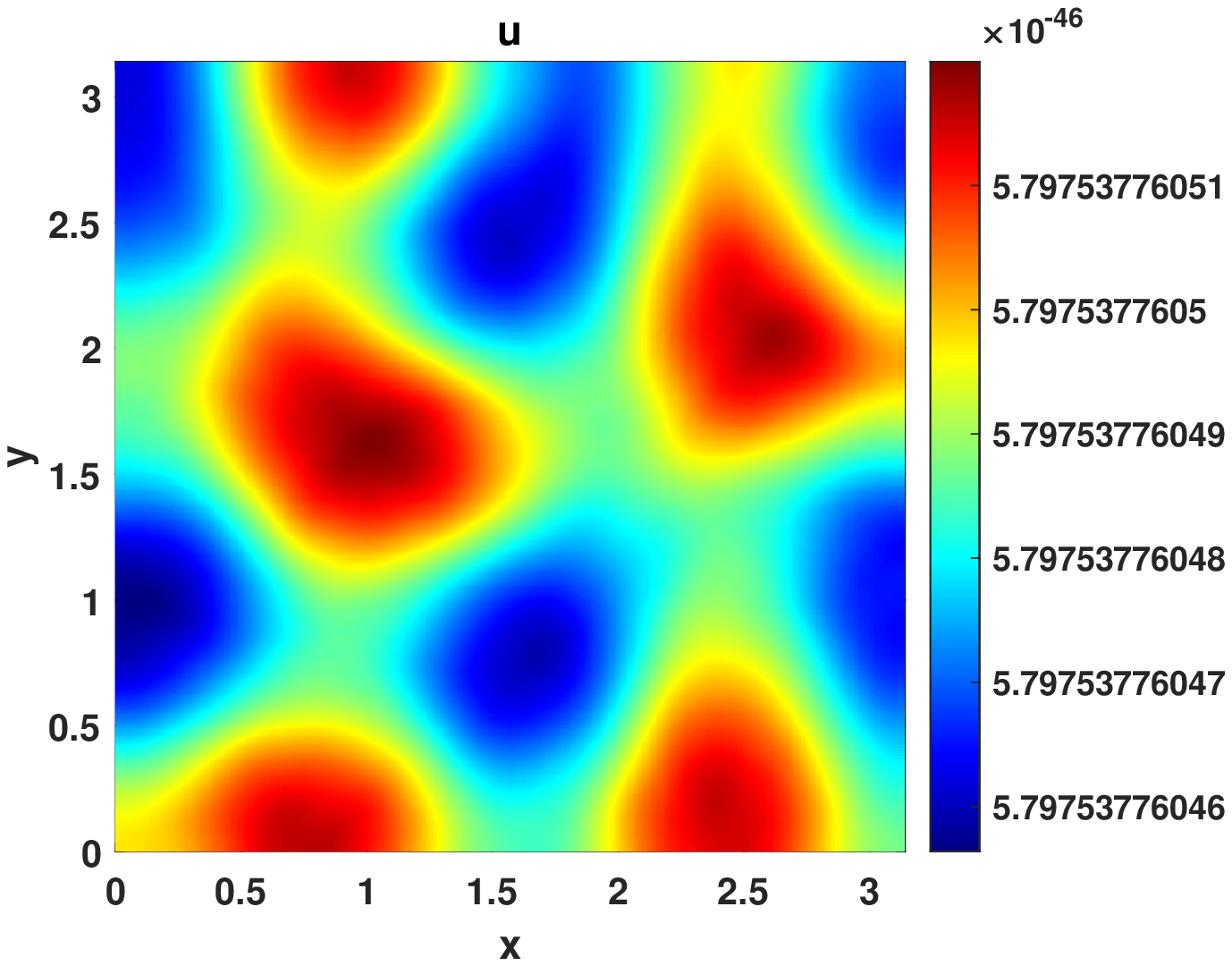}    \includegraphics[scale=0.32]{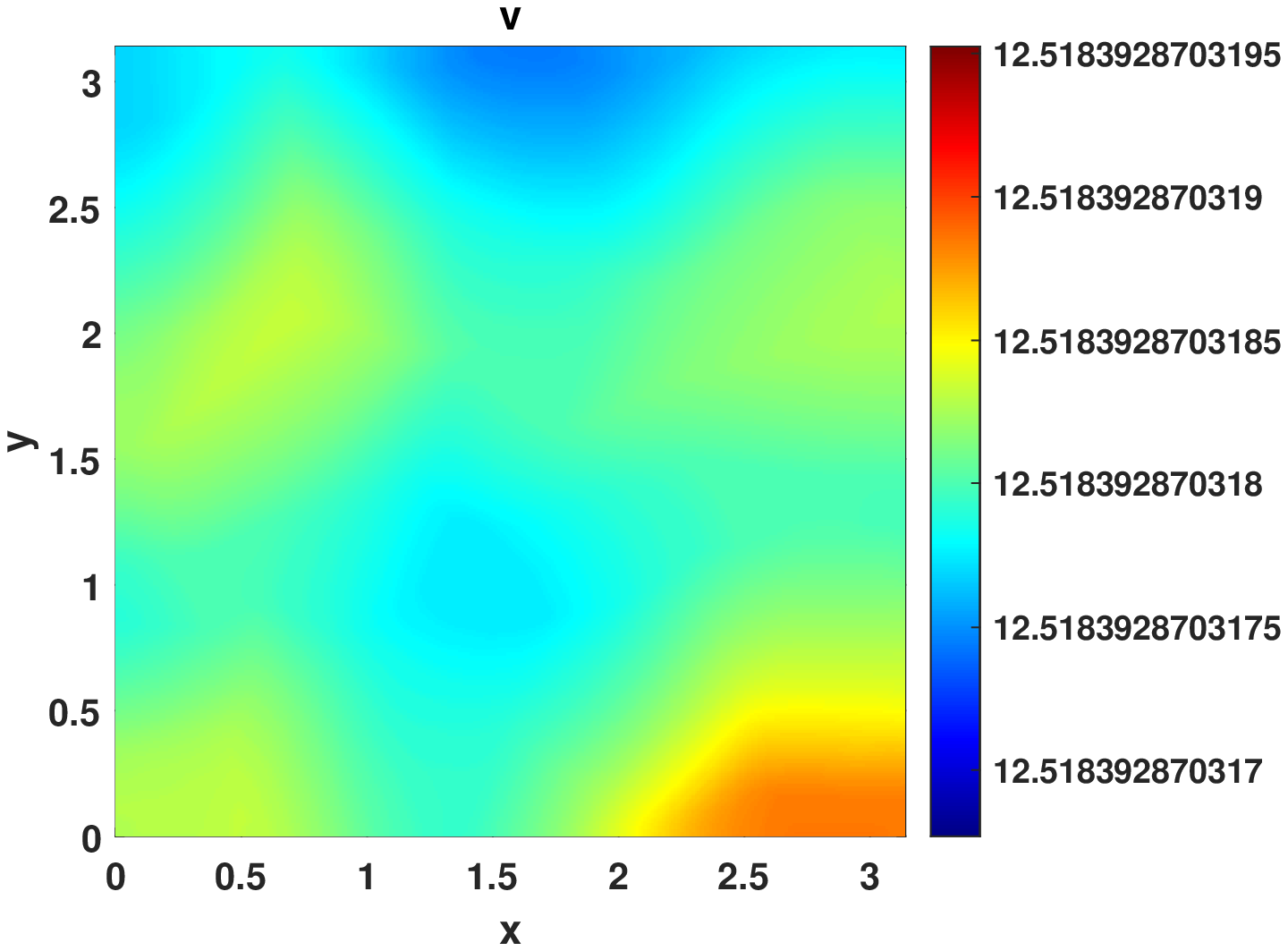}    \includegraphics[scale=0.32]{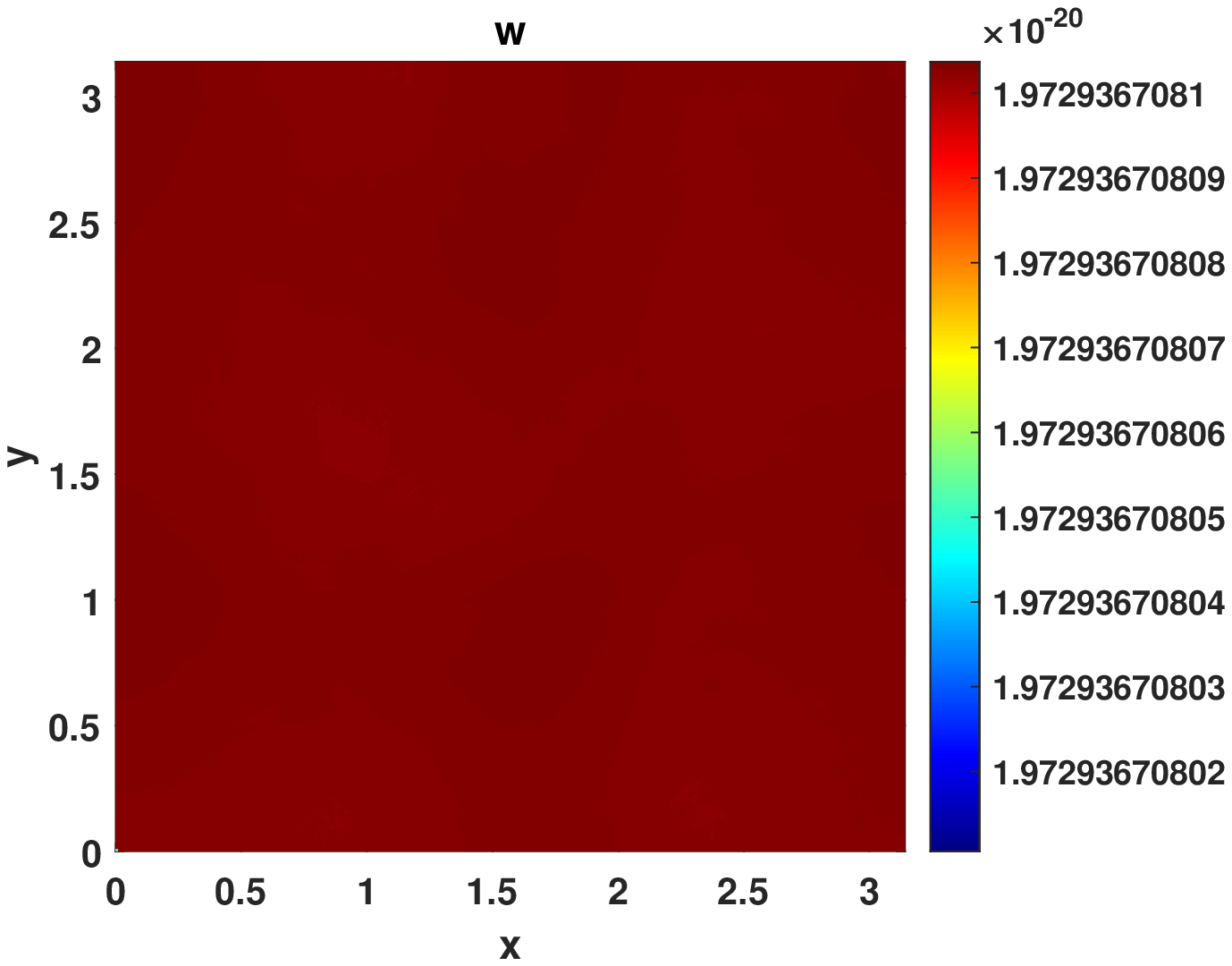}
\caption{Contour plot of the population densities at $t=5000$, when disease transmission rate $\lambda = 0$ and cannibalistic attack rate $\sigma = 0.005$. Other parameters are fixed as given in (\ref{paraeq}).}
\label{nopatterns1}
\end{figure}
\begin{figure}[!ht]
\centering
 {\includegraphics[scale=0.36]{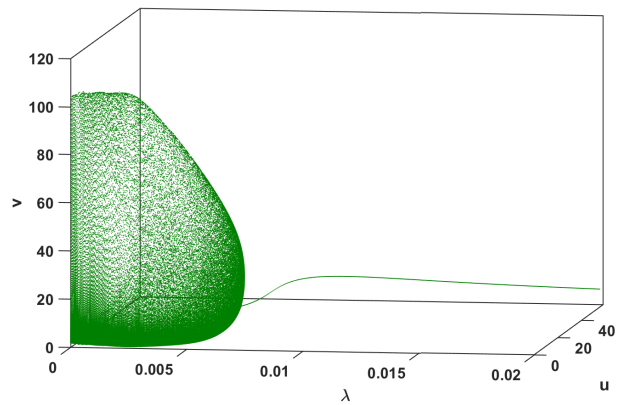}}
\caption{Bifurcation diagram with bifurcation parameter as disease transmission rate ($\lambda$). Here $0\leqslant \lambda \leqslant0.02$ and $\sigma=0.005$. Other parameters are fixed as given in (\ref{paraeq}).}
\label{bifurcation_diagram1}
\end{figure}

\begin{figure}[!ht]
\centering
\subfigure[]{\includegraphics[scale=0.38]{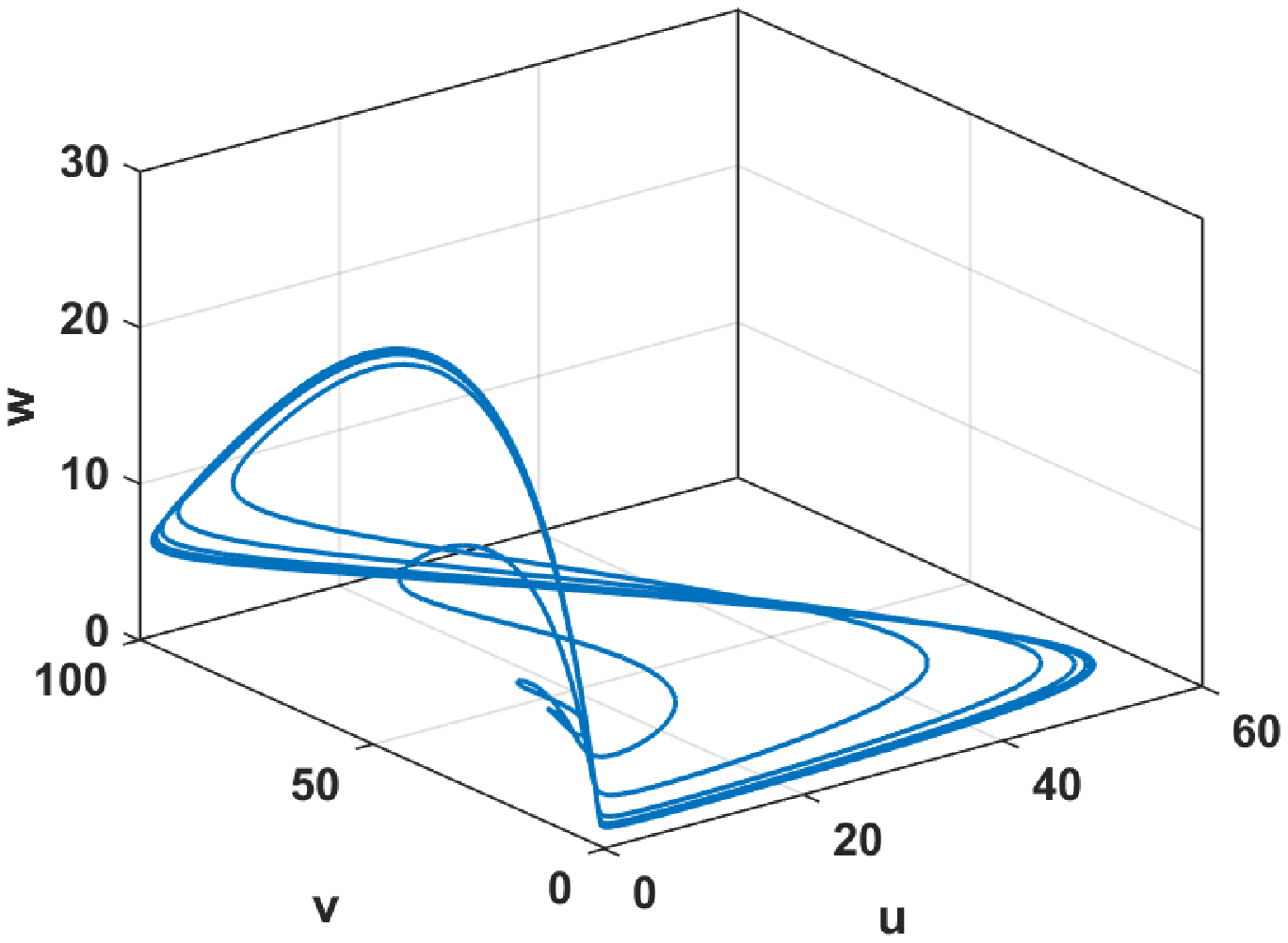}} \subfigure[]{\includegraphics[scale=0.26]{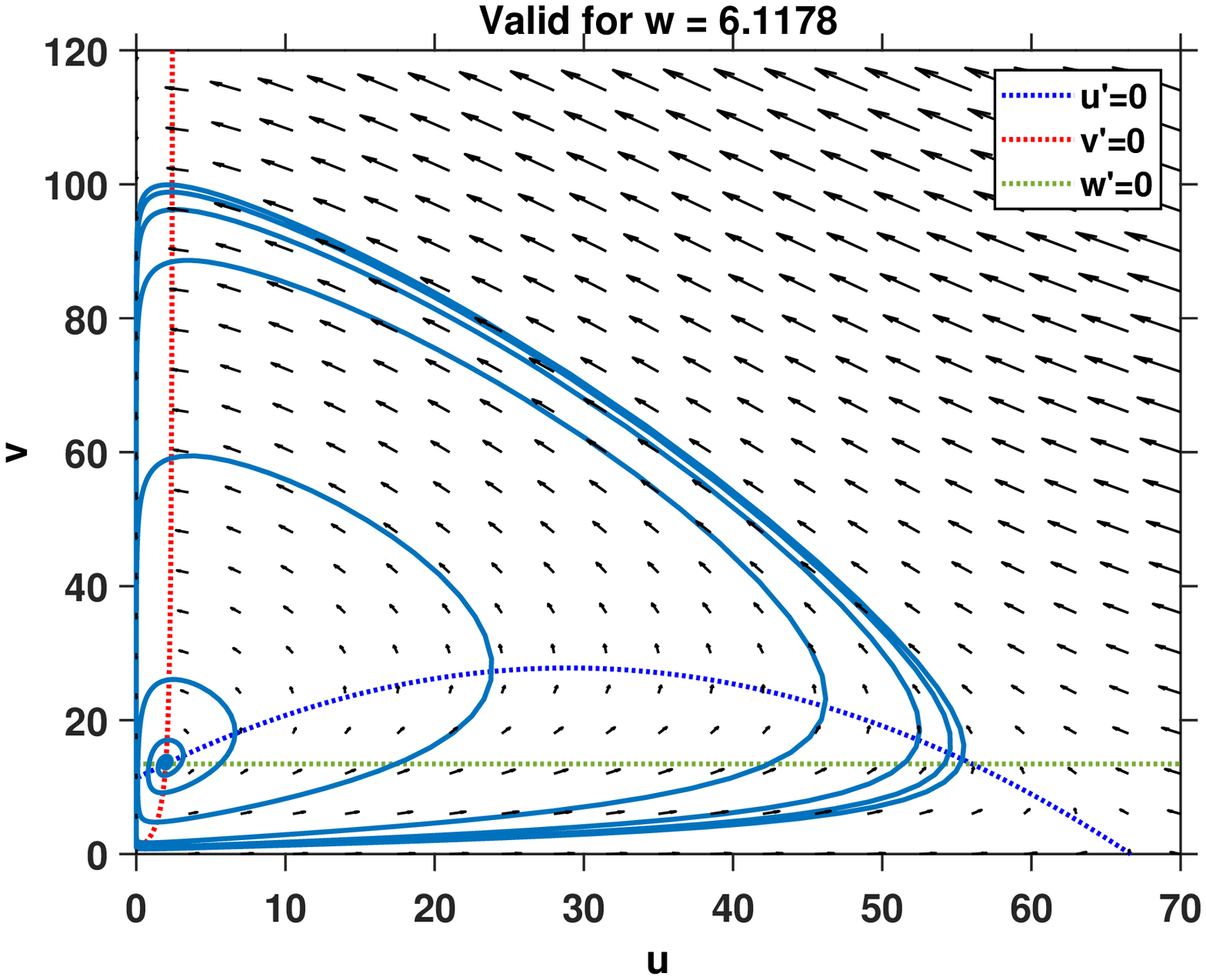}}\\
\subfigure[]{\includegraphics[scale=0.28]{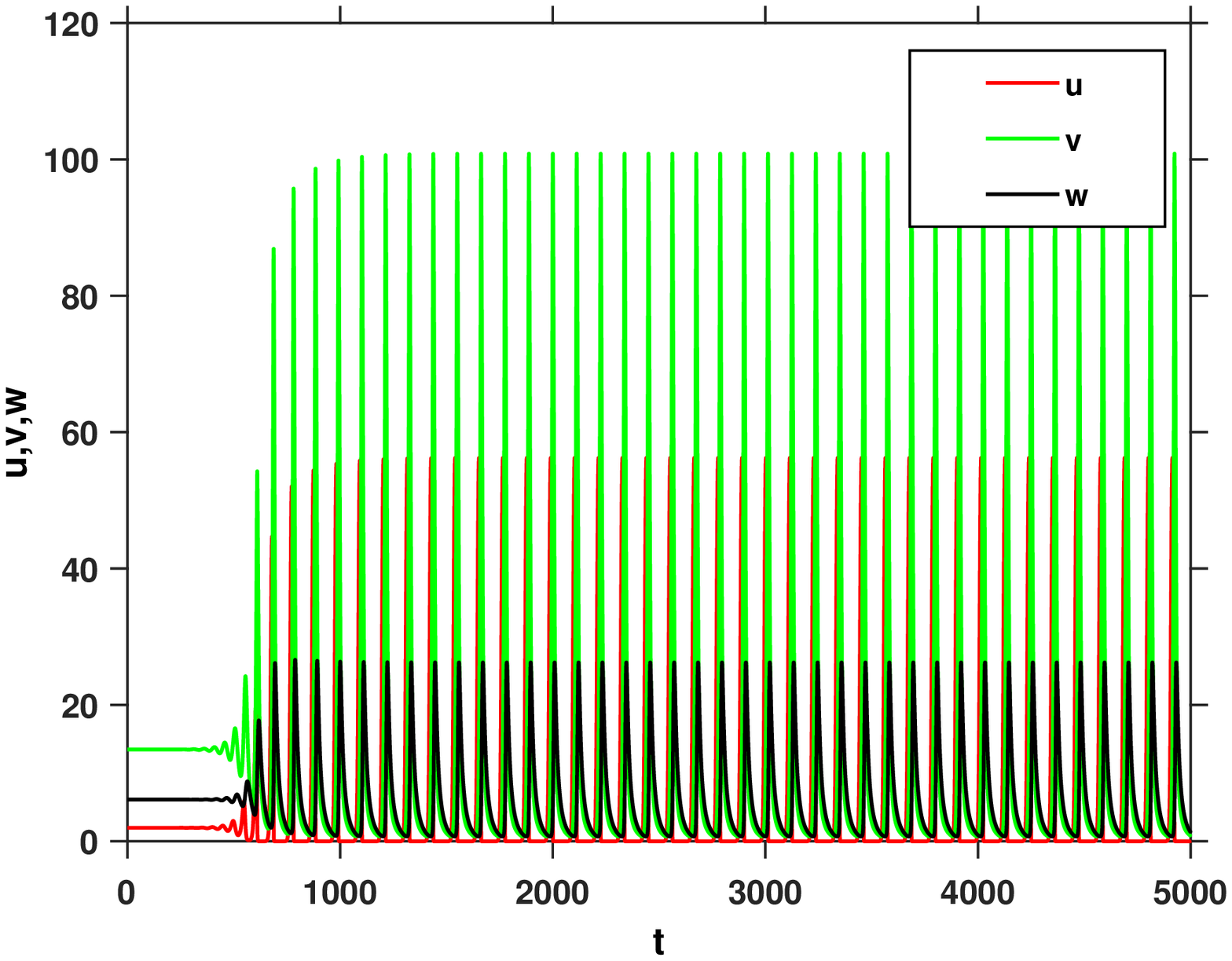}} 
\subfigure[]{\includegraphics[scale=0.39]{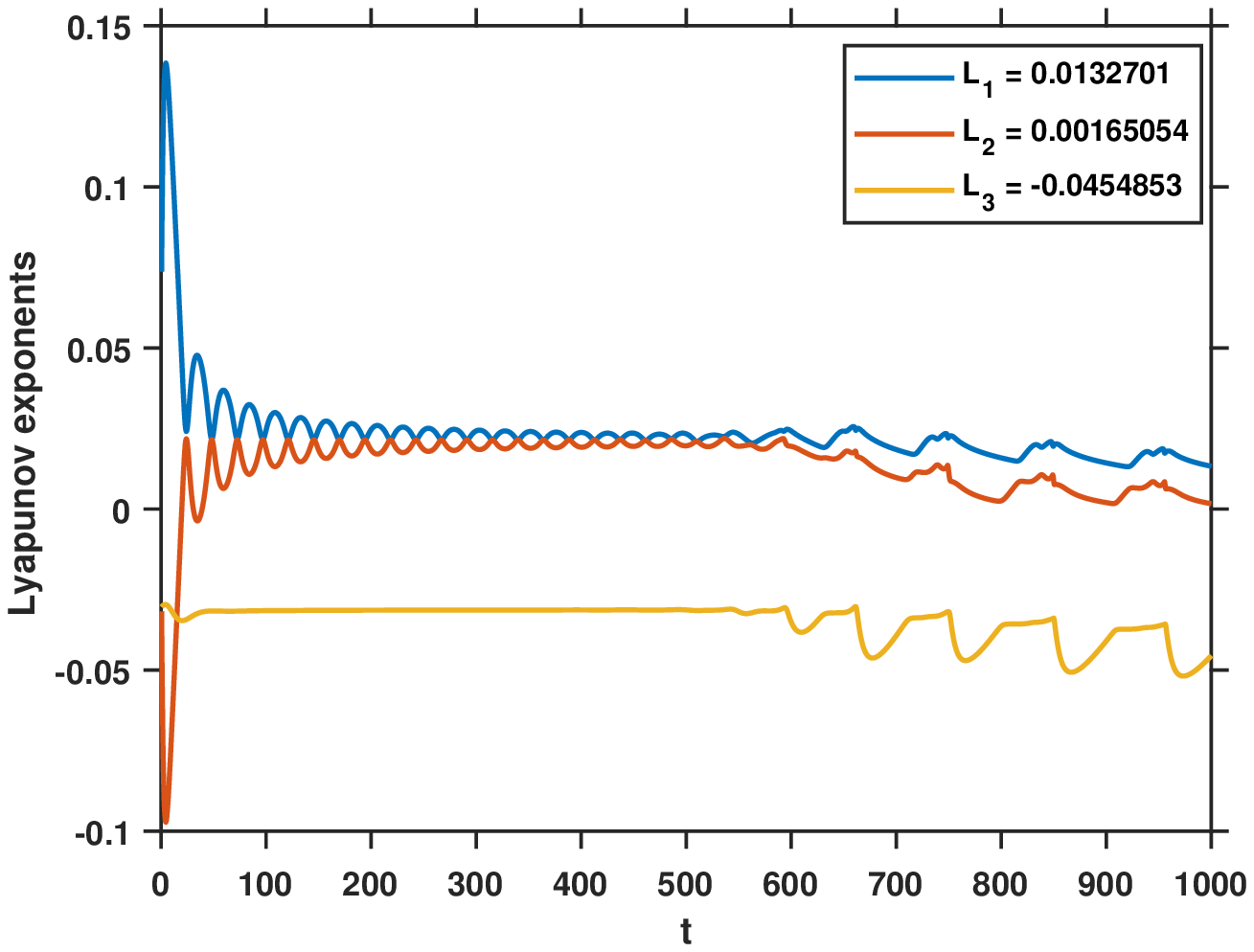}} 
\caption{Dynamics of model system (\ref{main}) at the parameter set, used for  the calculation of non-Turing patterns (given in (\ref{paraeq})). (a) Three-dimensional phase plot with direction field (b) Two-dimensional phase plot (c) Time series plot and (d) Lyapunov spectrum. }
\label{spatiodynamics}
\end{figure}
\renewcommand{\thefigure}{\arabic{figure}}
 \begin{figure} [!ht]
 \centering
\hspace{-0.5cm}   \includegraphics[scale=0.32]{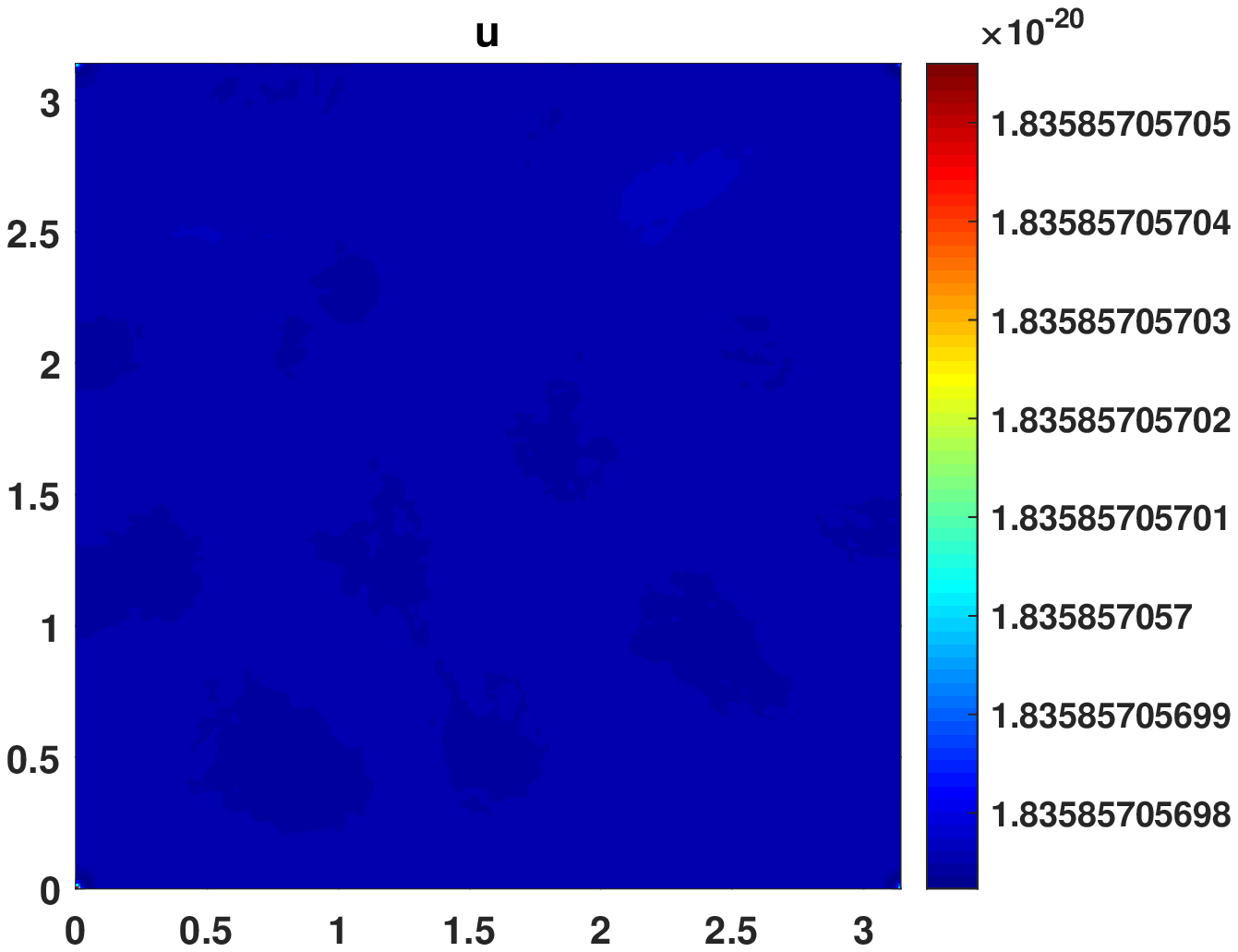}      \includegraphics[scale=0.32]{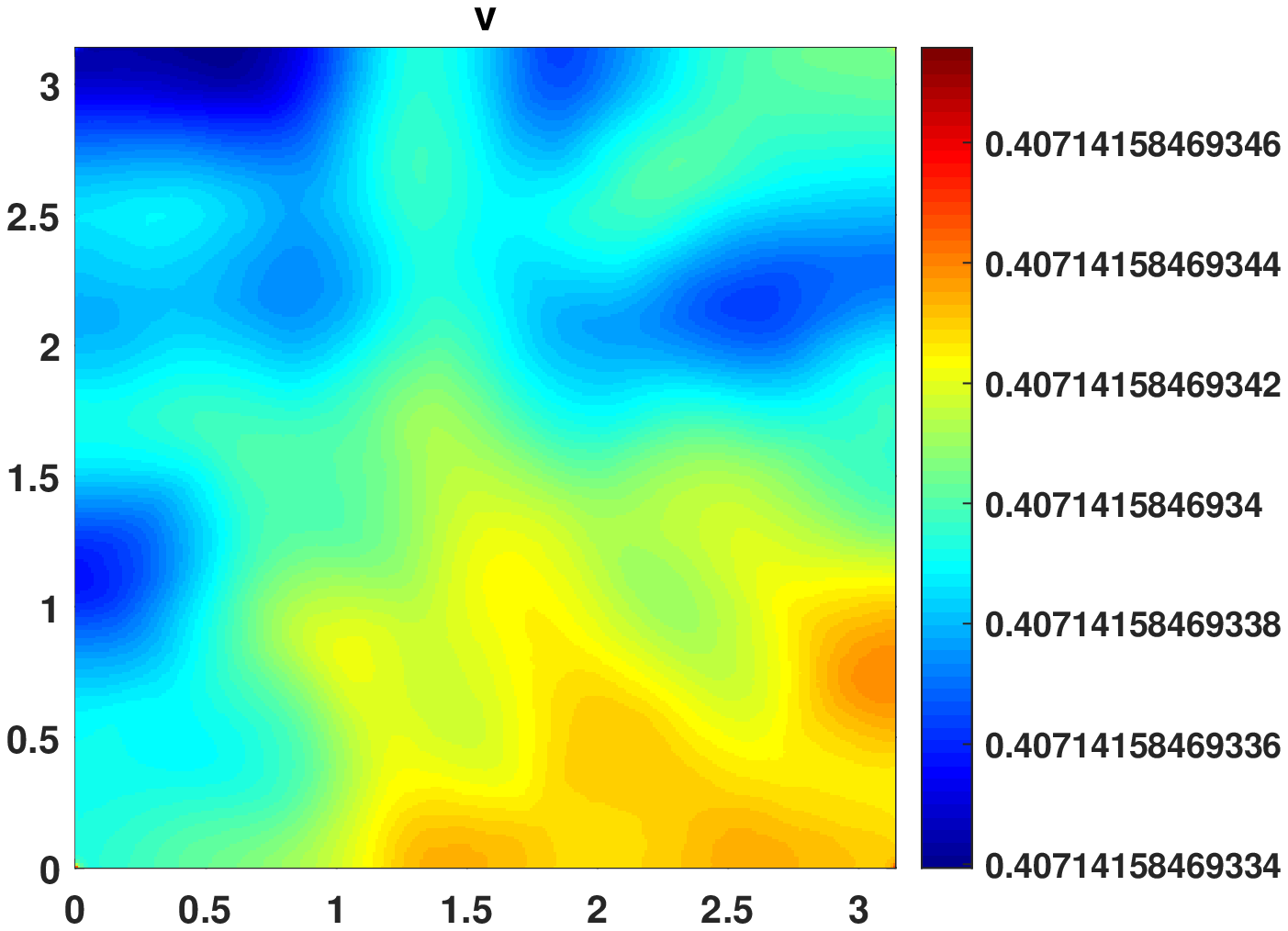}     \includegraphics[scale=0.32]{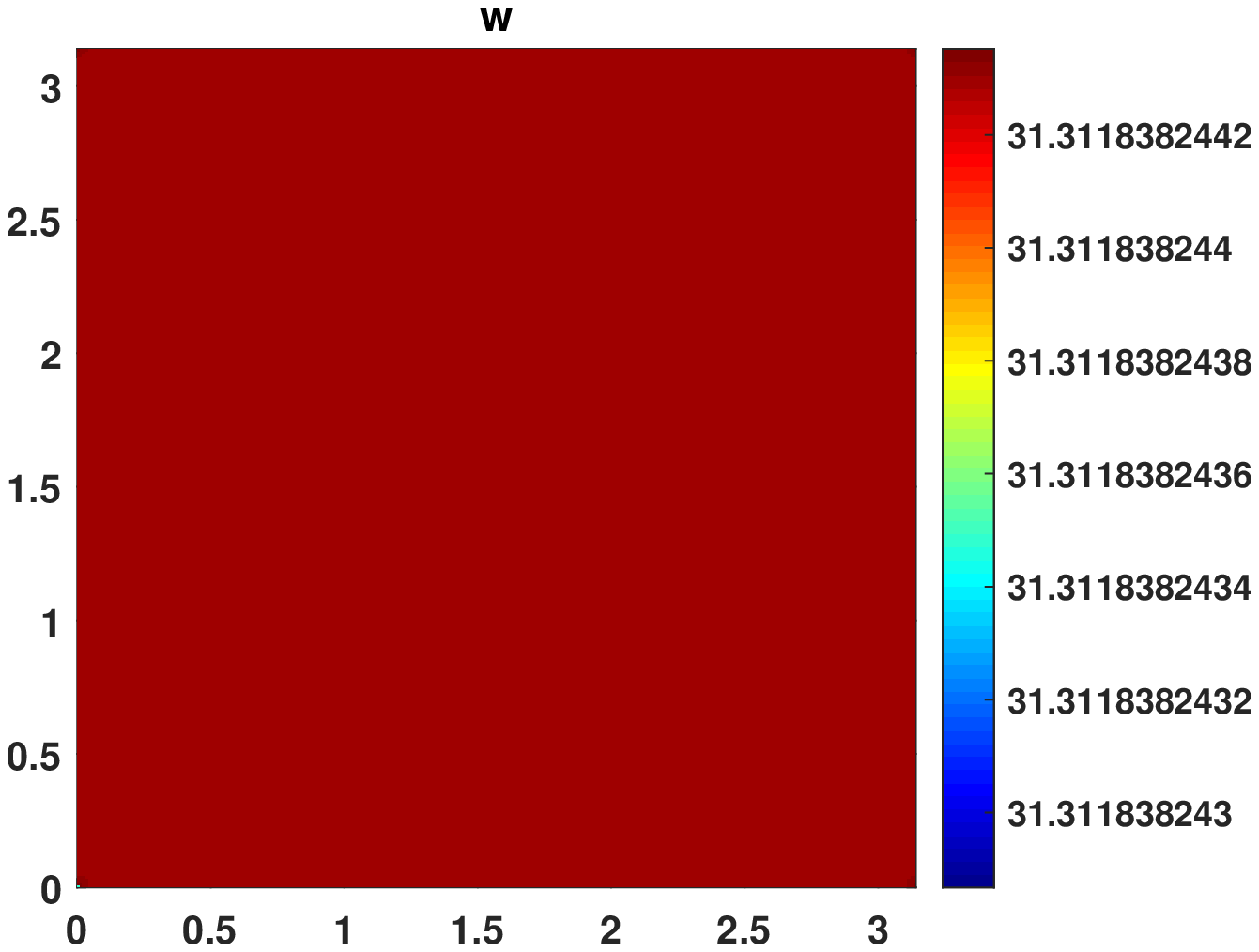} 
\caption{Contour plot of the population densities at $t=5000$, when $\sigma = 0$ and $\lambda=0.003$. Other parameters are fixed as given in (\ref{paraeq}).}
\label{nopatterns}
\end{figure}
\begin{figure}[!ht]
\centering
{\includegraphics[scale=0.36]{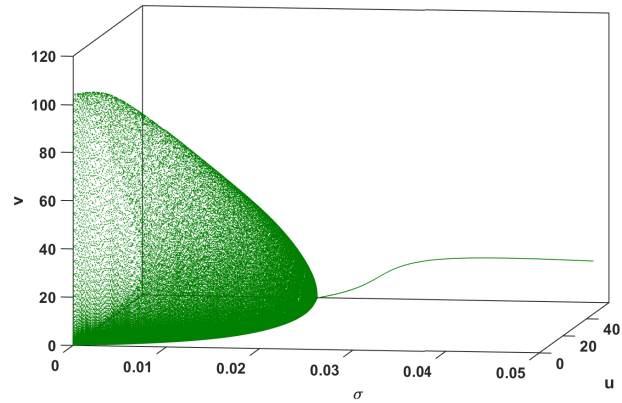}} 
\caption{Bifurcation diagram with bifurcation parameter as cannibalistic attack rate ($\sigma$). The range for $\sigma$ is $0\leqslant \sigma \leqslant0.05$ and $\lambda=0.003$. Other parameters are fixed as given in (\ref{paraeq}).}
\label{bifurcation_diagram}
\end{figure}
\begin{figure}[!ht]
\centering
\subfigure[]{\includegraphics[scale=0.38]{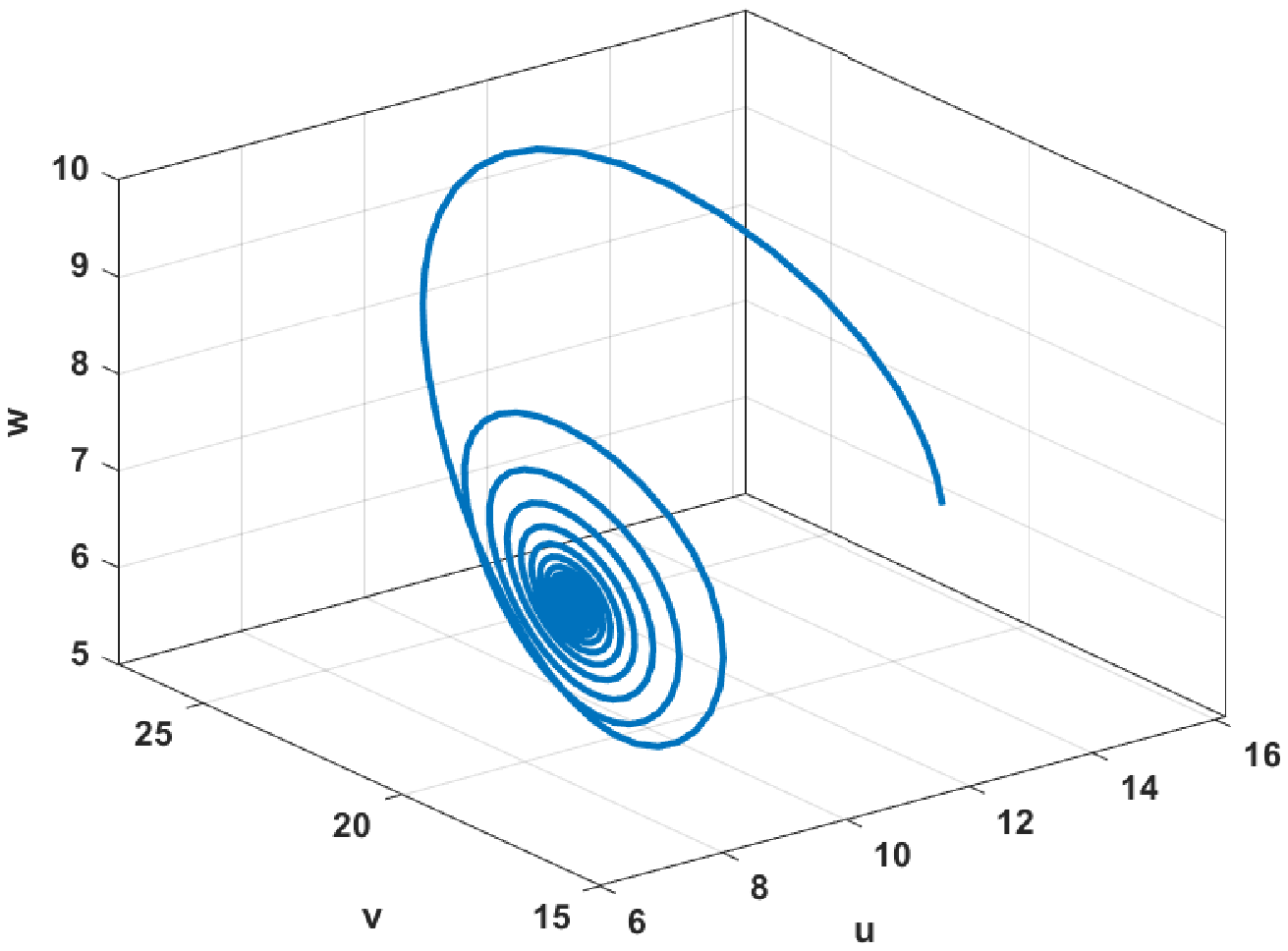}} \subfigure[]{\includegraphics[scale=0.38]{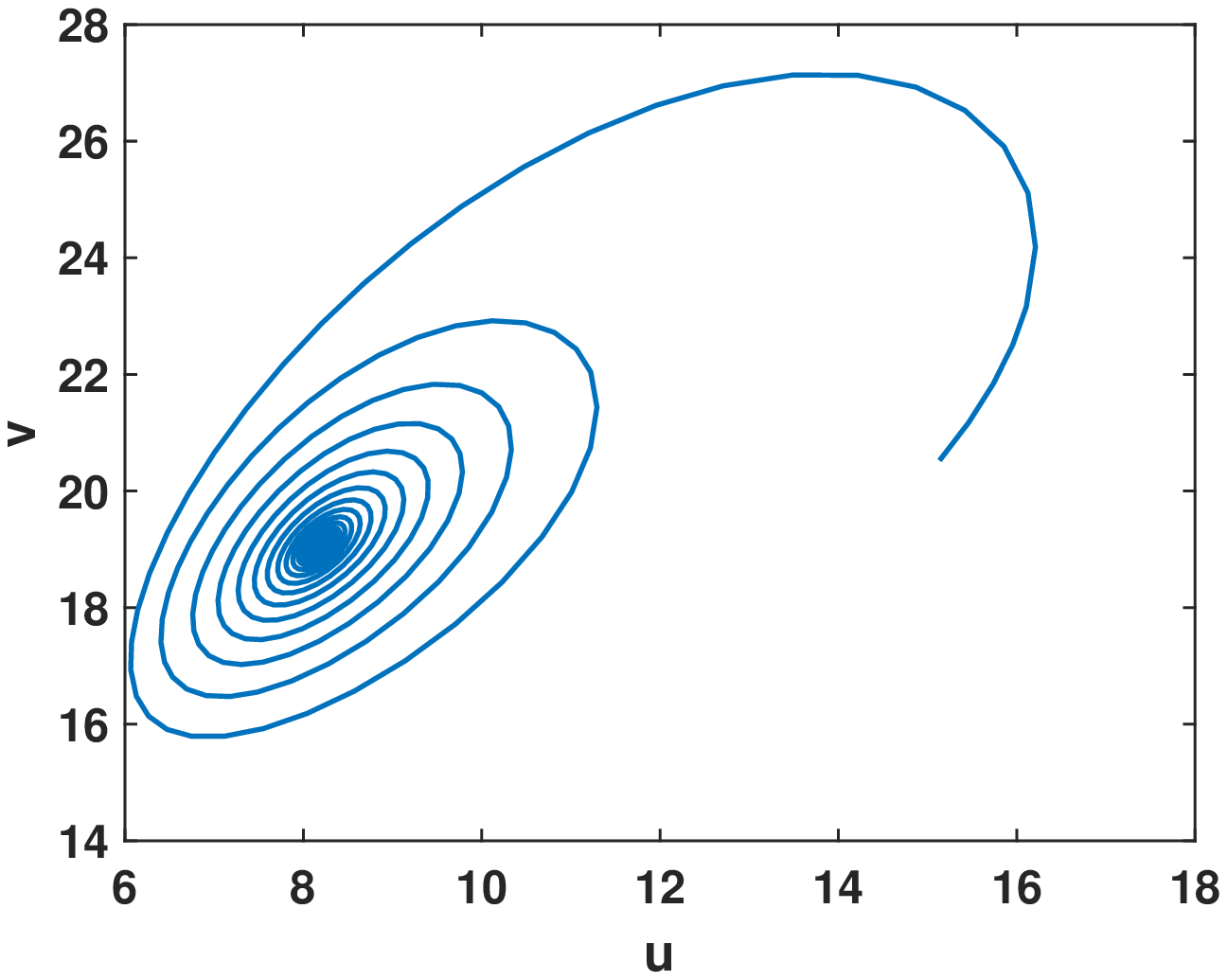}} \\
\subfigure[]{\includegraphics[scale=0.38]{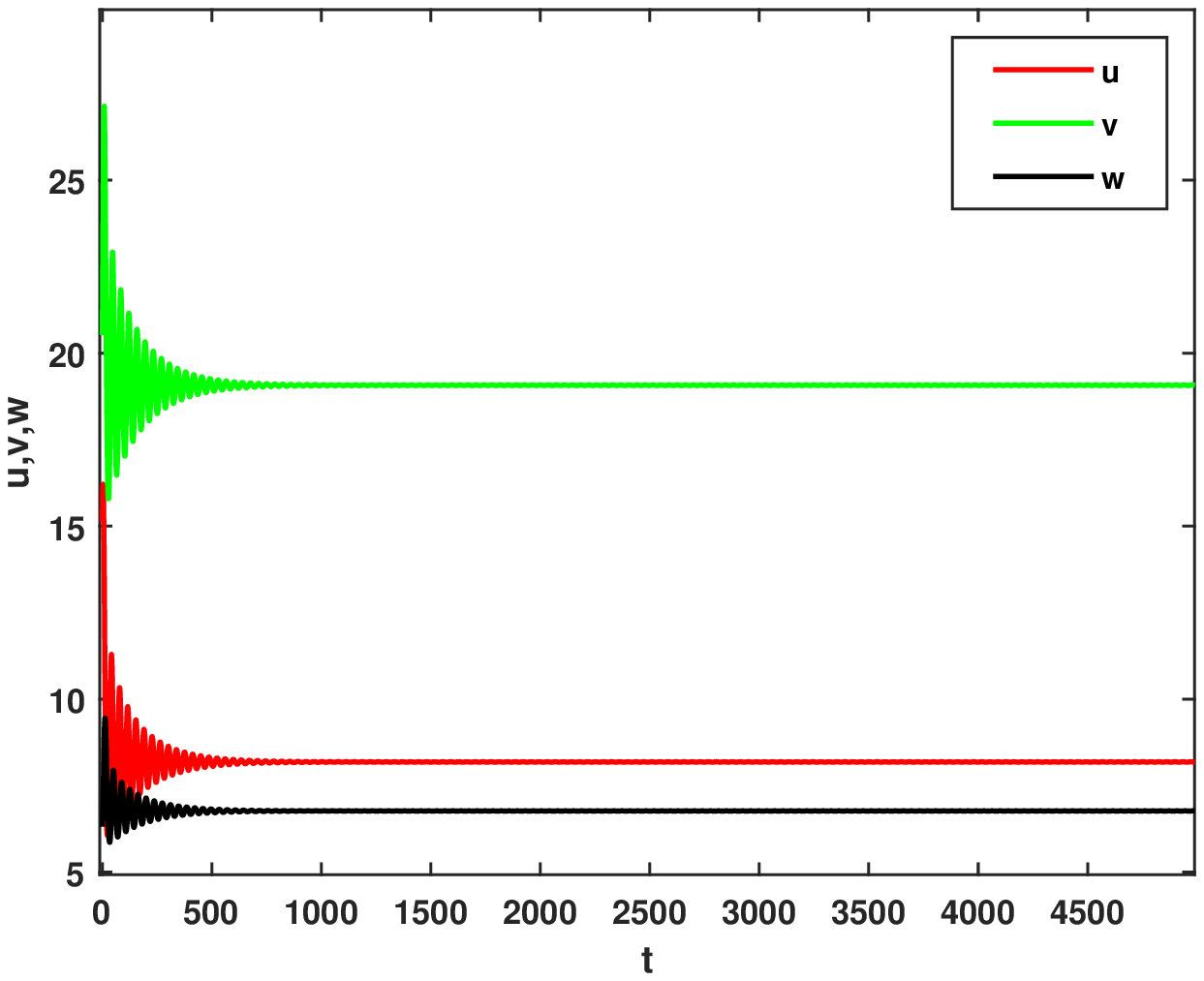}}
\subfigure[]{\includegraphics[scale=0.38]{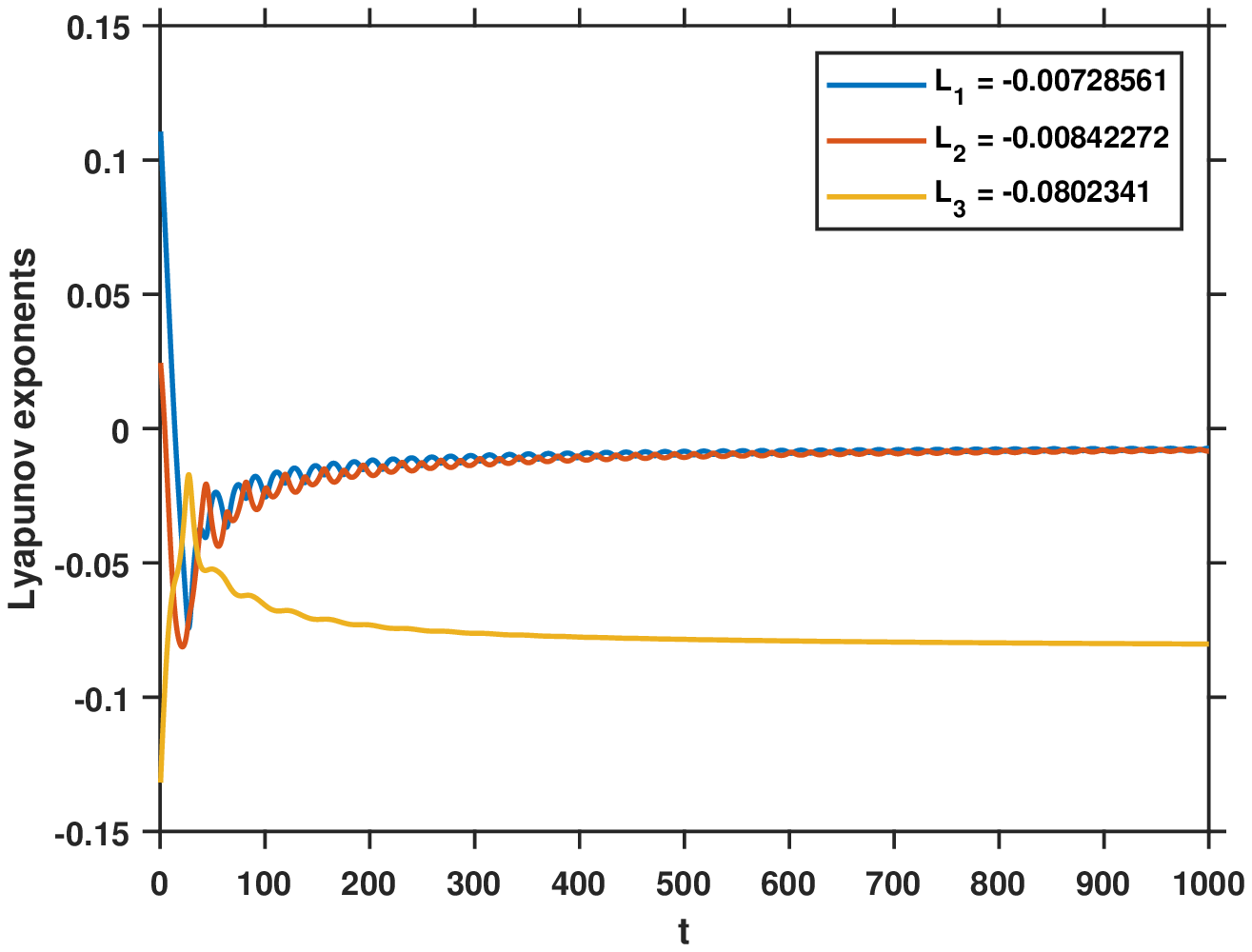}} 
\caption{ Dynamics of model system (\ref{main}) when $\sigma=0.026$ and other parameter values are fixed as given in (\ref{paraeq}). (a) Three-dimensional phase attractor (b) Two-dimensional phase plot (c) Time series plot and (d) Lyapunov spectrum. This is the parameter set at which Turing patterns are seen.}
\label{stablefocusplot}
\end{figure}
                      \begin{figure}[!ht]
   \centering\hspace{-0.3cm}
  \subfigure[]  {\includegraphics[width=5cm]{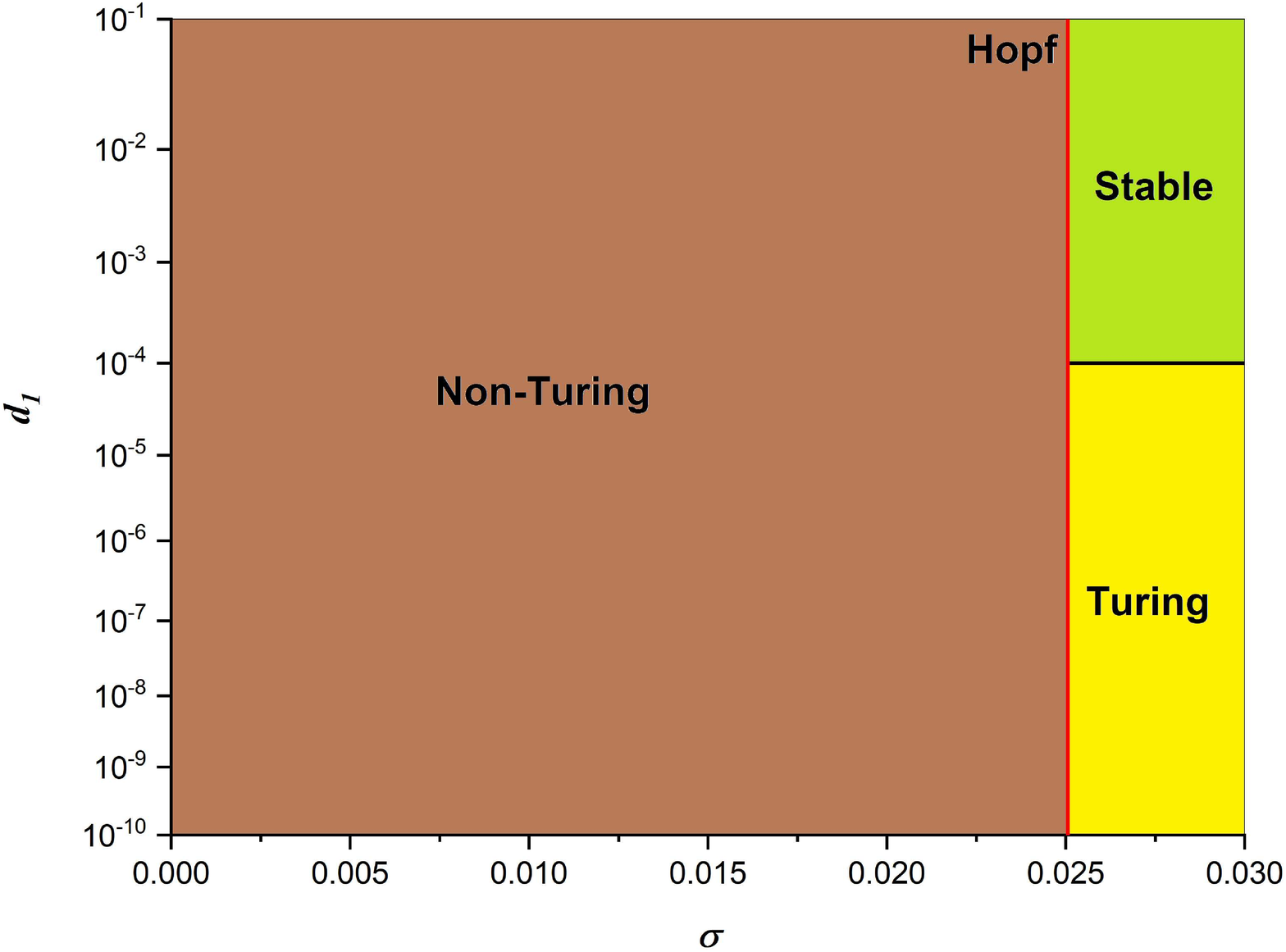}}   
   \subfigure[]  {\includegraphics[width=5cm]{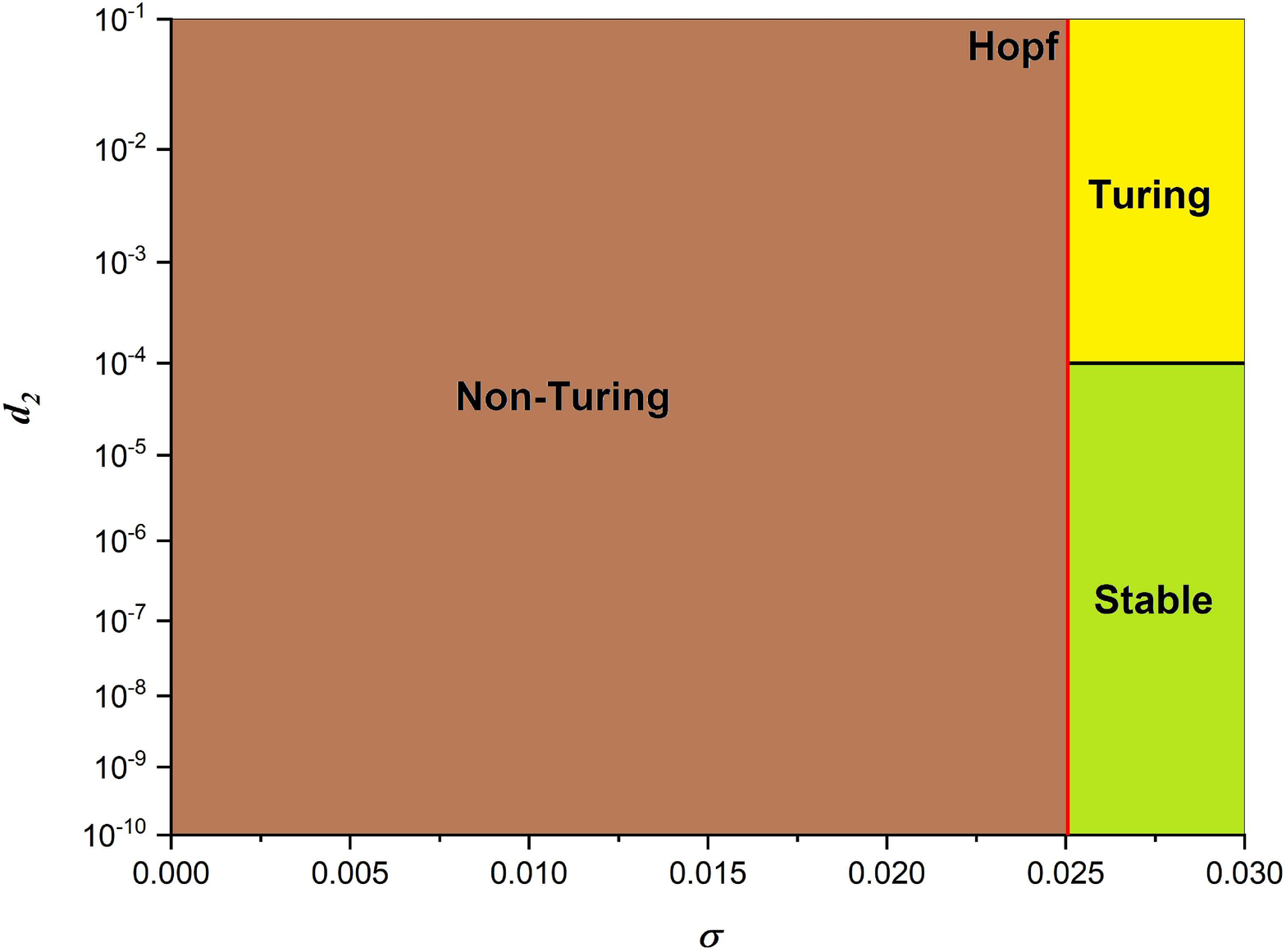}} \\ 
    \subfigure[]  {\includegraphics[width=6cm]{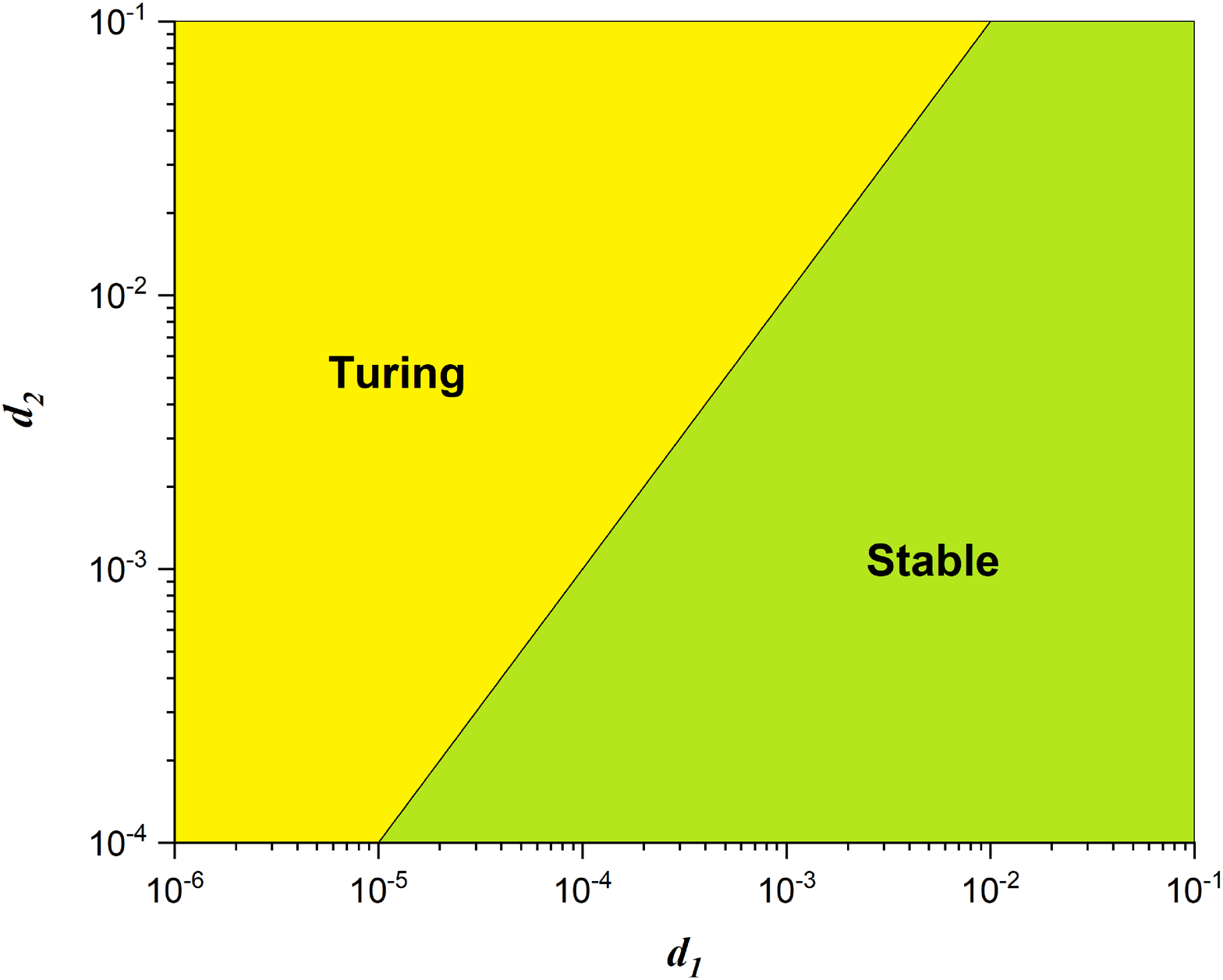}}
 \caption {Yellow region shows Turing instability, green region shows planar stablity and brown region shows non-Turing region. Red line represents Hopf bifurcation which separates the Turing and non-Turing regions. Here (a) $d_2=10^{-3},d_3=10^{-10}$ (b) $d_1=10^{-5},d_3=10^{-10}$ and (c) $\sigma=0.026,d_3=10^{-10}$. Other parameters are same as given in Table \ref{paratable}. }  \label{regiont}
 \end{figure}
 
  \renewcommand{\thefigure}{\arabic{figure}}
 \begin{figure} [!ht]
 \centering
\begin{tabular}{cccc} \vspace{-1cm}
\begin{tabular}[c]{@{}c@{}c@{}c@{}c@{}}A \\\\ \\\\\\ \\\\\\\\\end{tabular}\hspace{-0.3cm} &
 \includegraphics[scale=0.26]{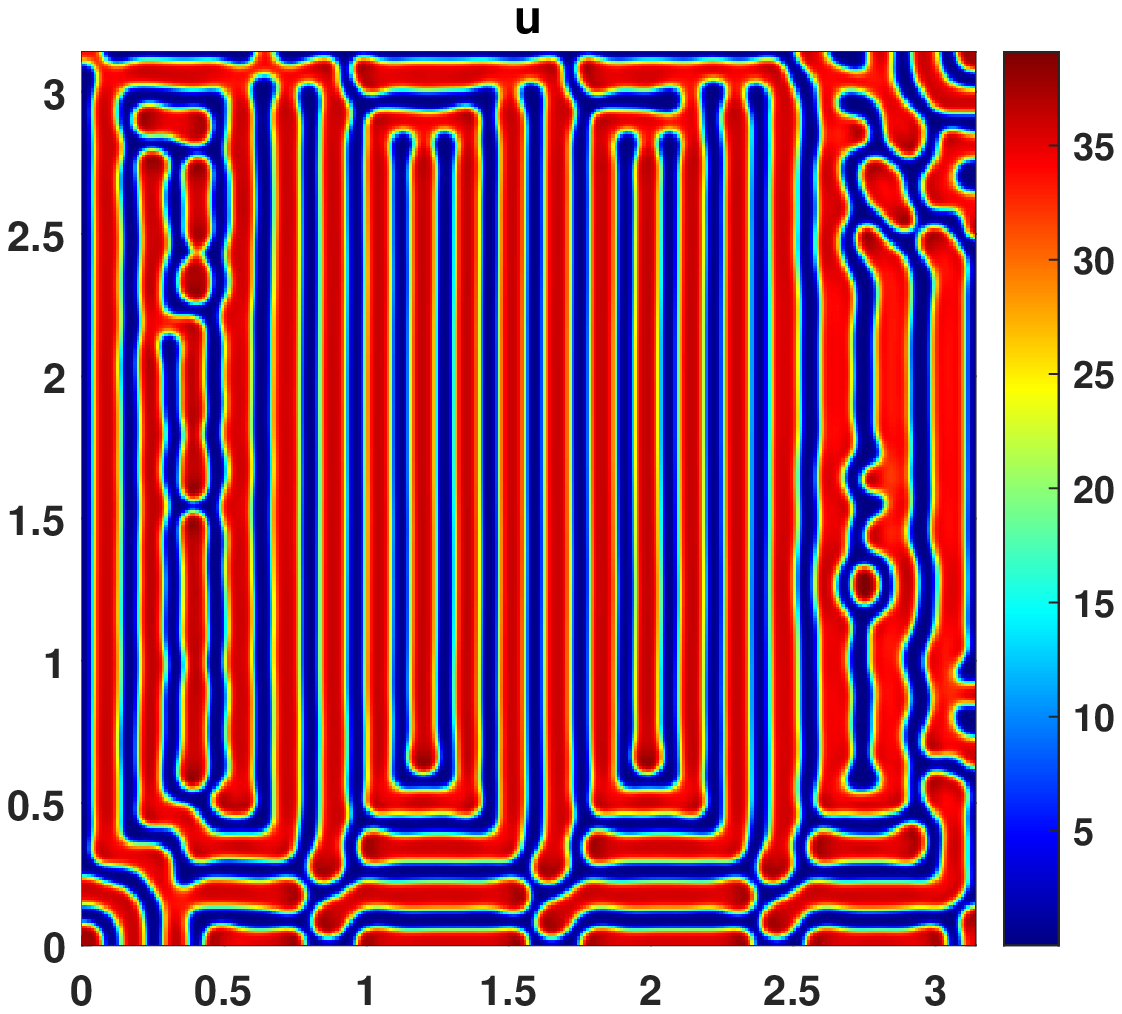} \hspace{-0.5cm} &    \includegraphics[scale=0.26]{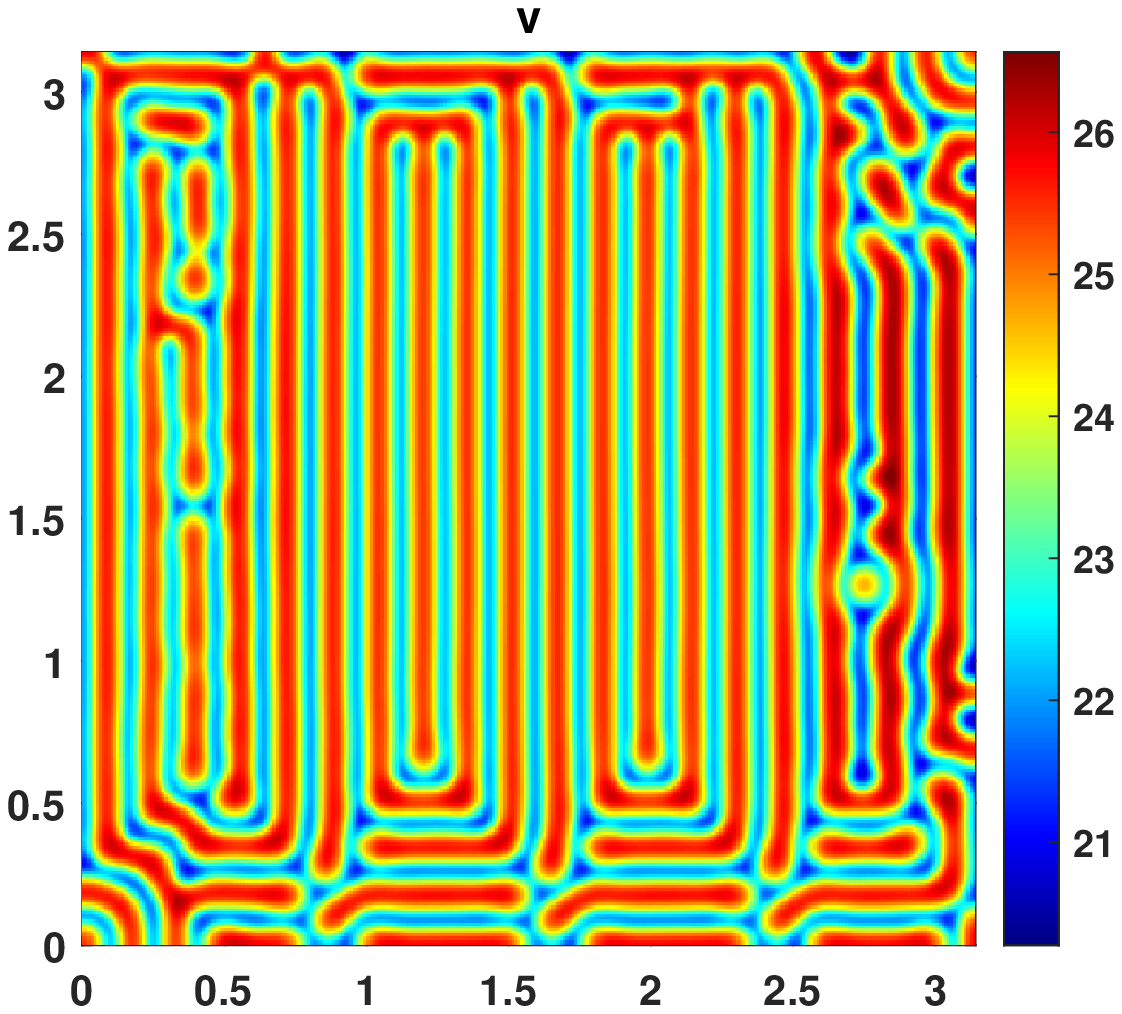} \hspace{-0.5cm} &  \includegraphics[scale=0.26]{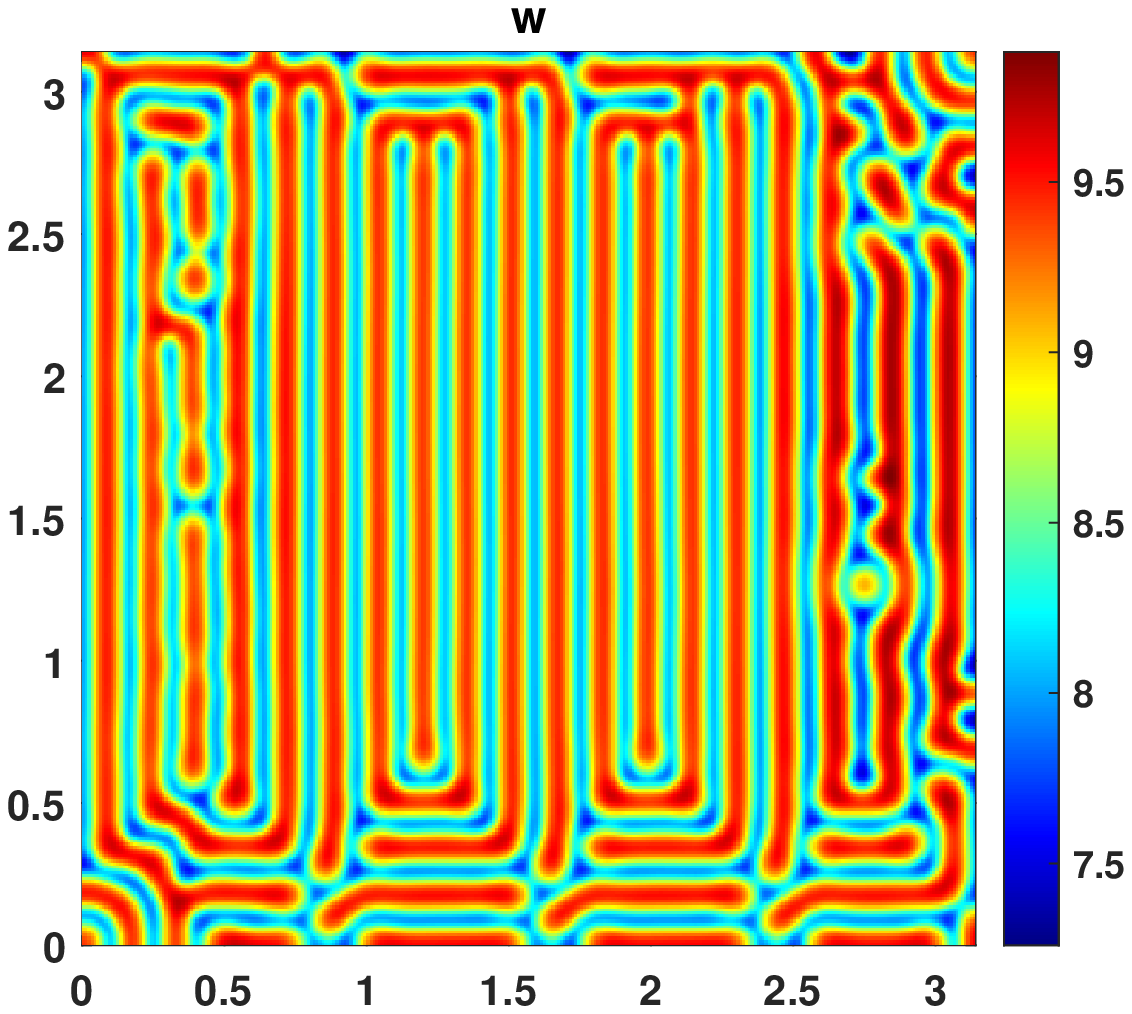}\\\vspace{-1cm}
\begin{tabular}[c]{@{}c@{}c@{}c@{}c@{}}B  \\\\\\\\\\  \\\\\\\\\end{tabular}\hspace{-0.3cm} & \includegraphics[scale=0.26]{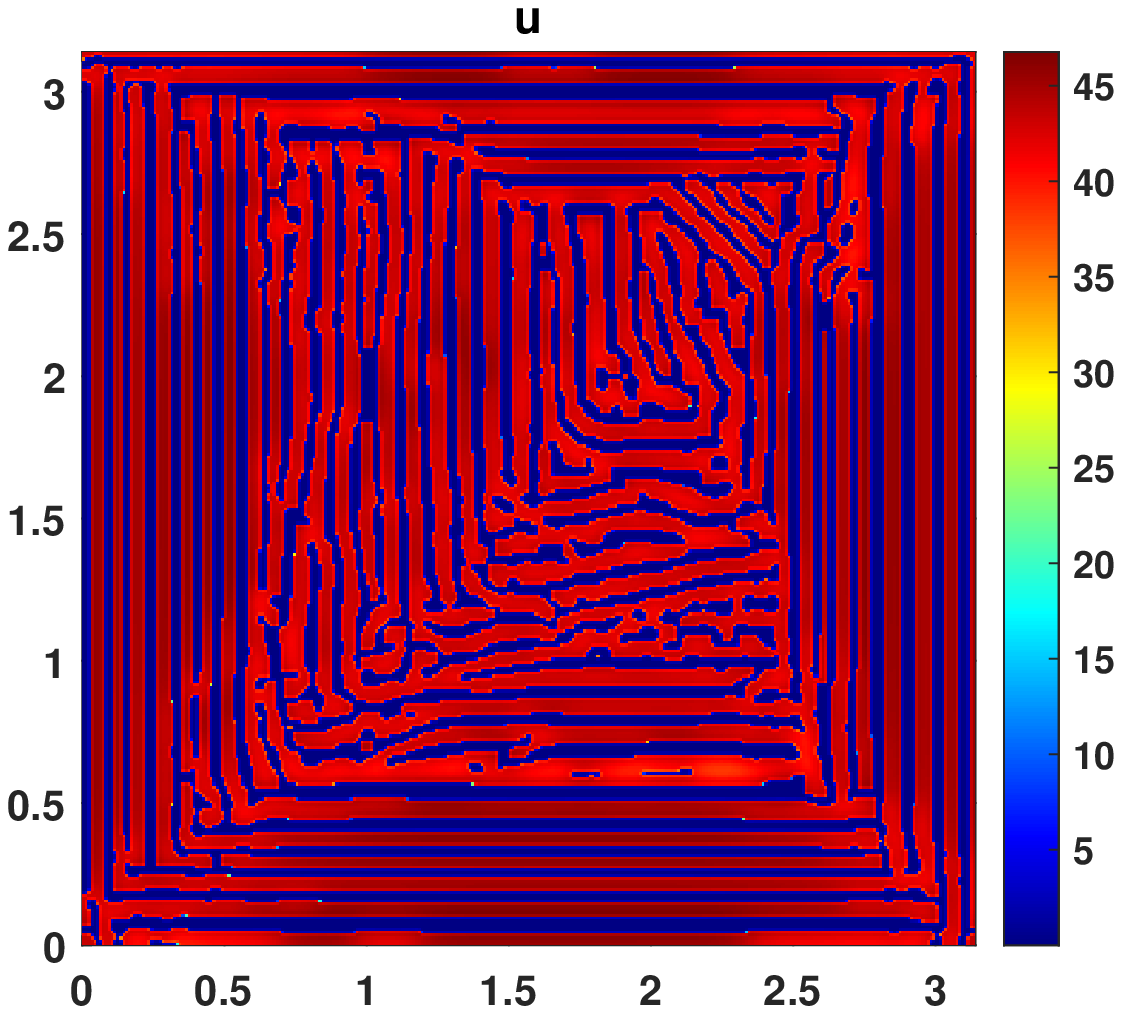}\hspace{-0.5cm}  &   \includegraphics[scale=0.26]{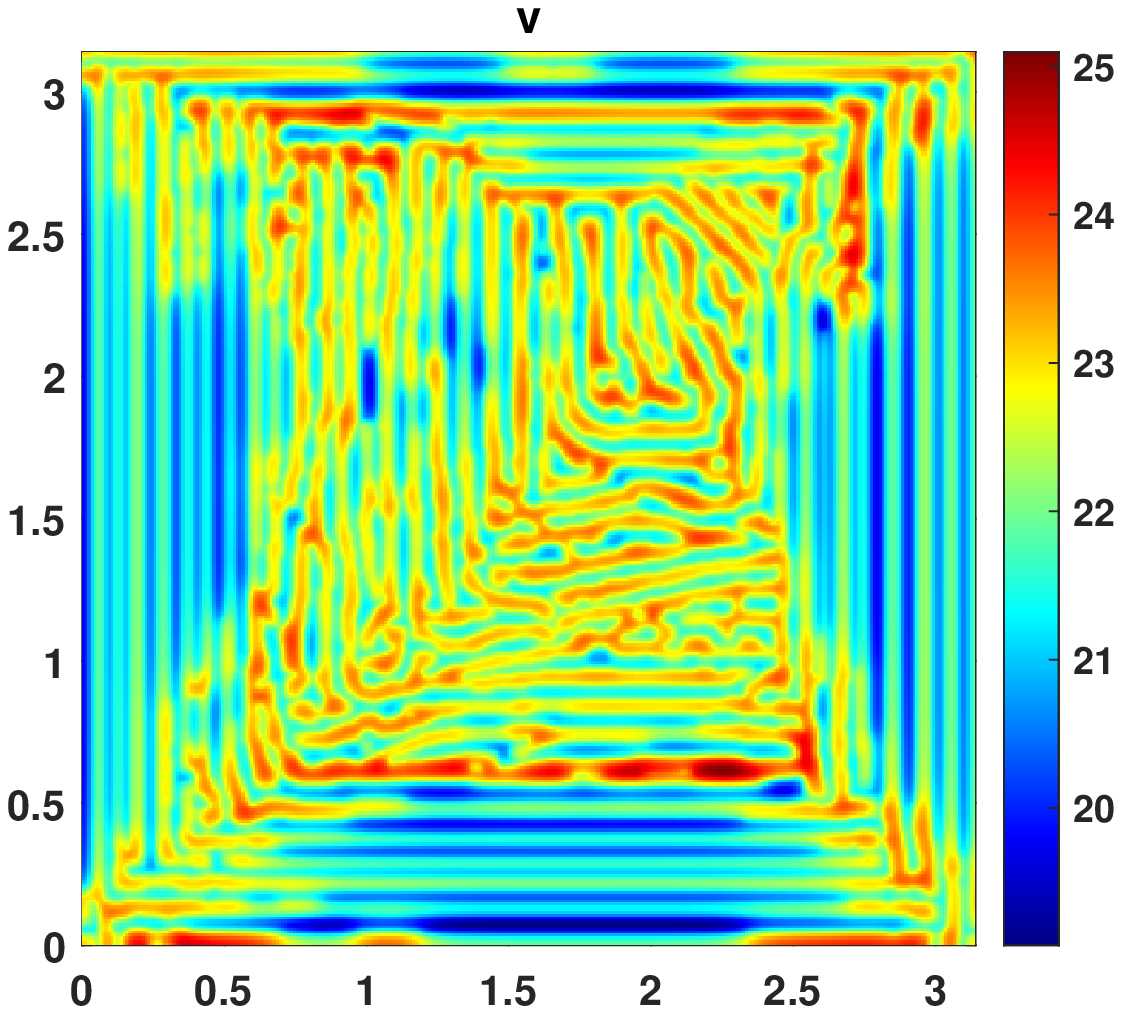} \hspace{-0.5cm} &  \includegraphics[scale=0.26]{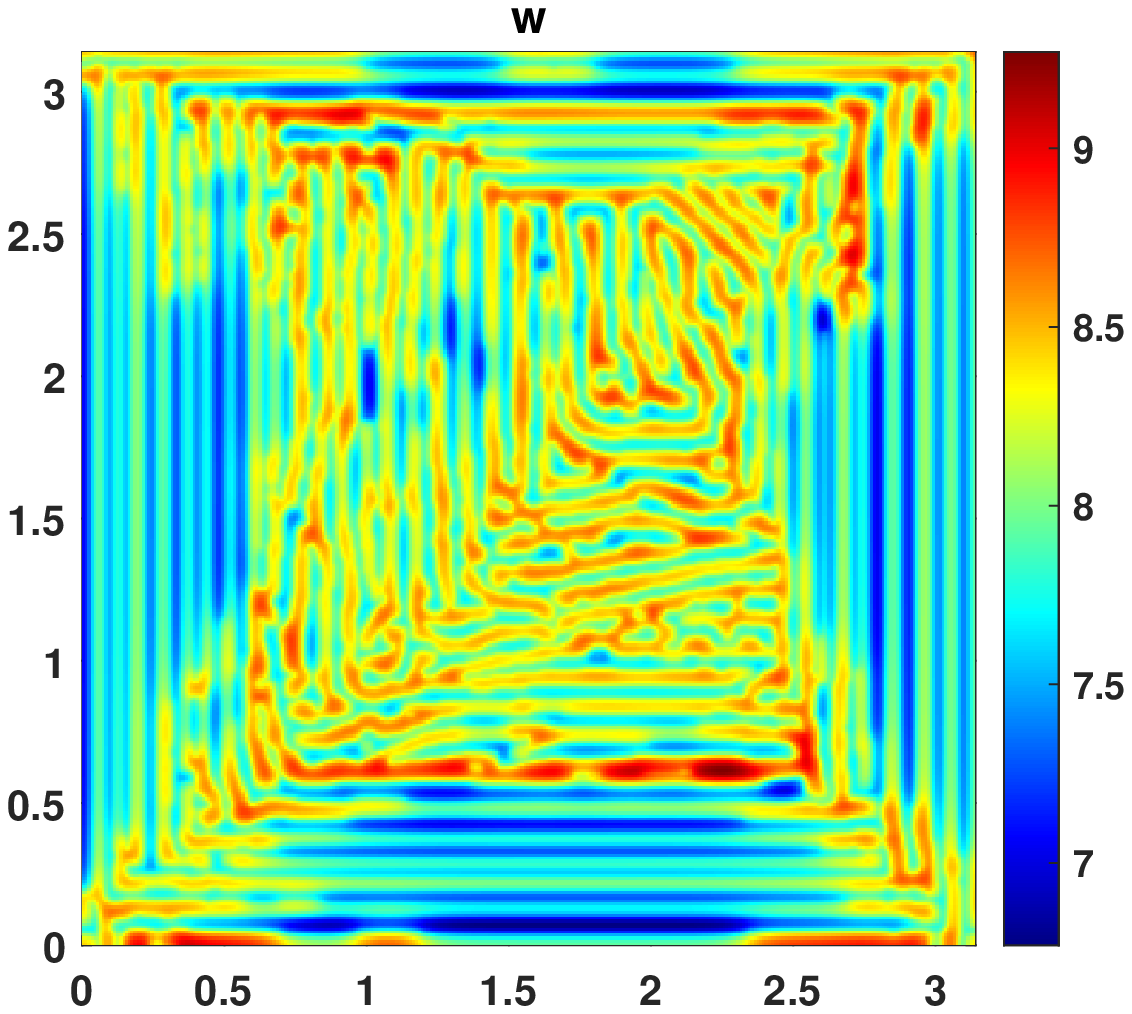} \end{tabular}
  \caption{Turing patterns seen in prey, susceptible predator and infected predator densities of model system (\ref{eq2}) at (A) $t =1000$, the diffusive rate of prey is $10^{-5}$.~(B) $t =1000$ and diffusive rate of prey is  $10^{-6}$. Other parameter values are given in Table \ref{paratable}. }
\label{table8}
\end{figure}  
 \renewcommand{\thefigure}{\arabic{figure}}
 \begin{figure} [!ht]
 \centering
\begin{tabular}{cccc} \vspace{-1cm}
\begin{tabular}[c]{@{}c@{}c@{}c@{}c@{}}A  \\\\ \\\\\\ \\\\\\\\\end{tabular}\hspace{-0.3cm} &
 \includegraphics[scale=0.26]{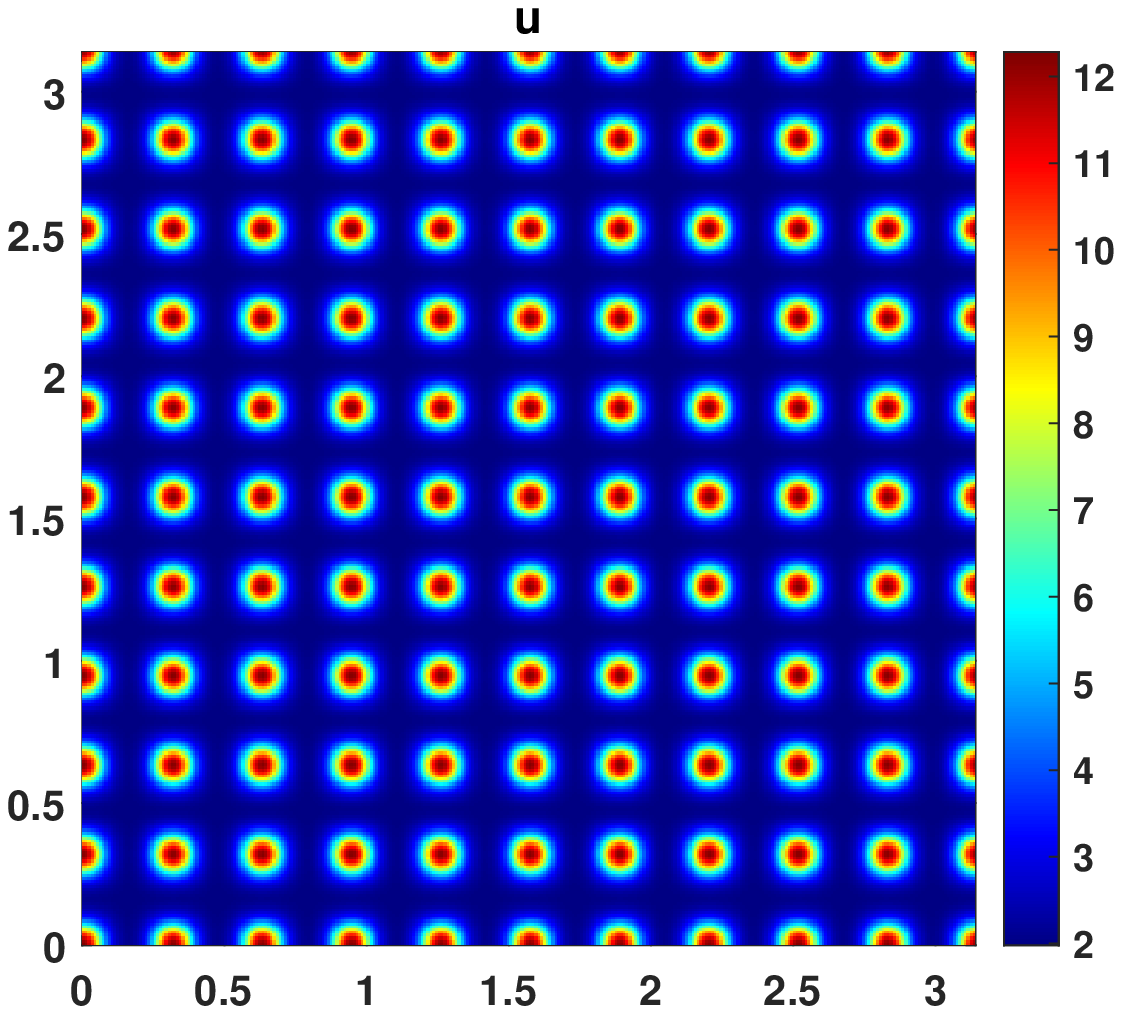} \hspace{-0.5cm} &    \includegraphics[scale=0.26]{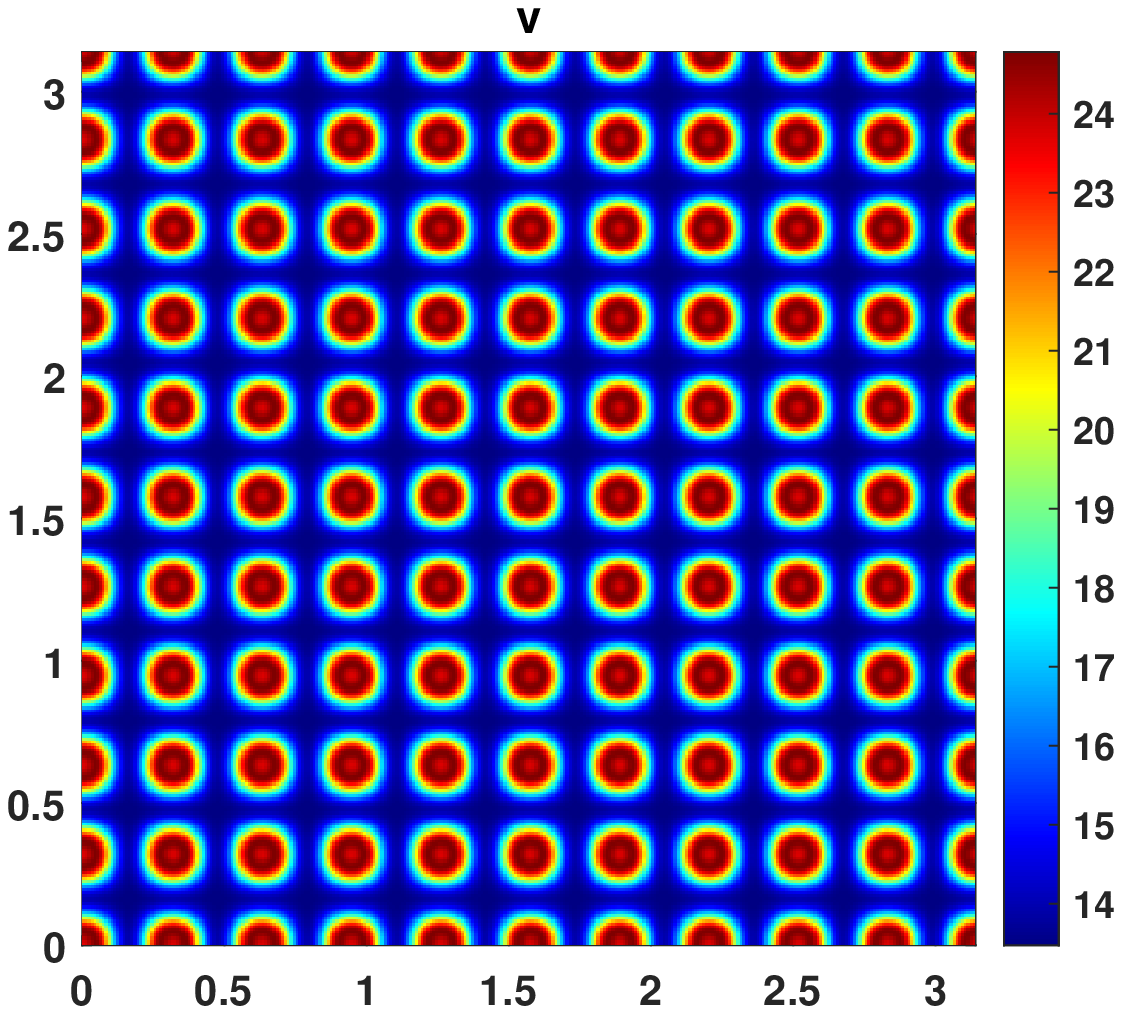} \hspace{-0.5cm} &  \includegraphics[scale=0.26]{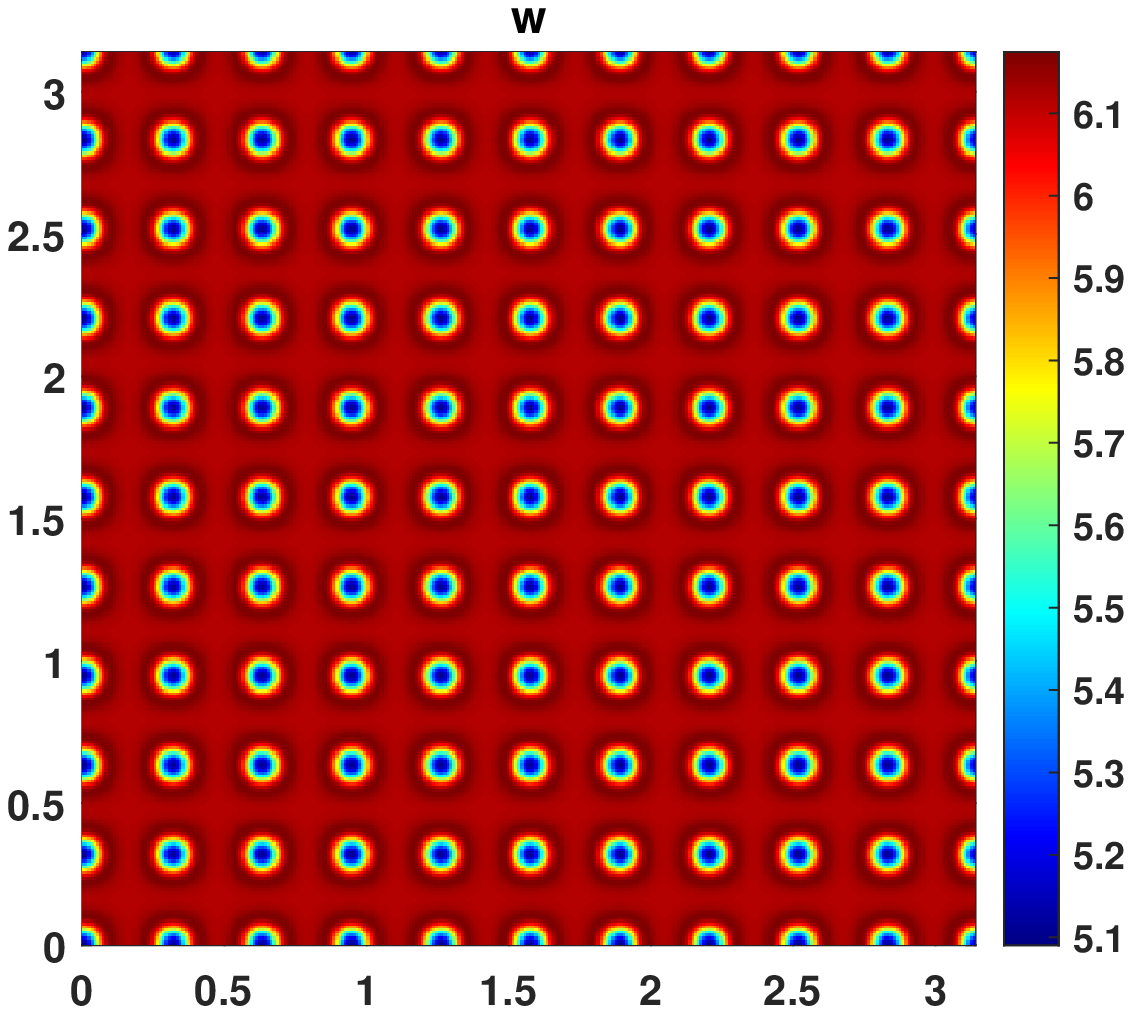}\\\vspace{-1cm}
\begin{tabular}[c]{@{}c@{}c@{}c@{}c@{}}B  \\\\\\\\ \\ \\\\\\\\\end{tabular}\hspace{-0.3cm} & \includegraphics[scale=0.26]{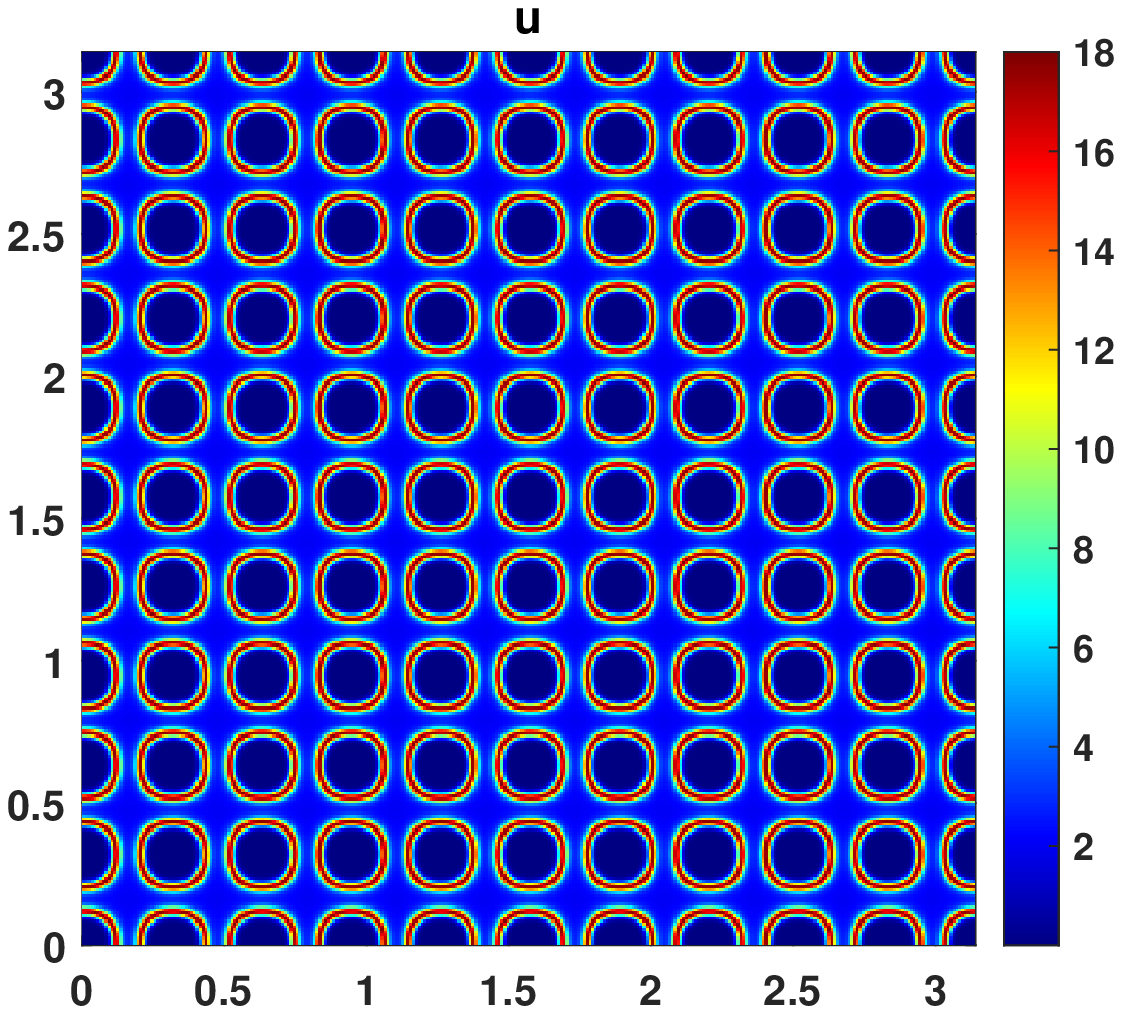}\hspace{-0.5cm}  &   \includegraphics[scale=0.26]{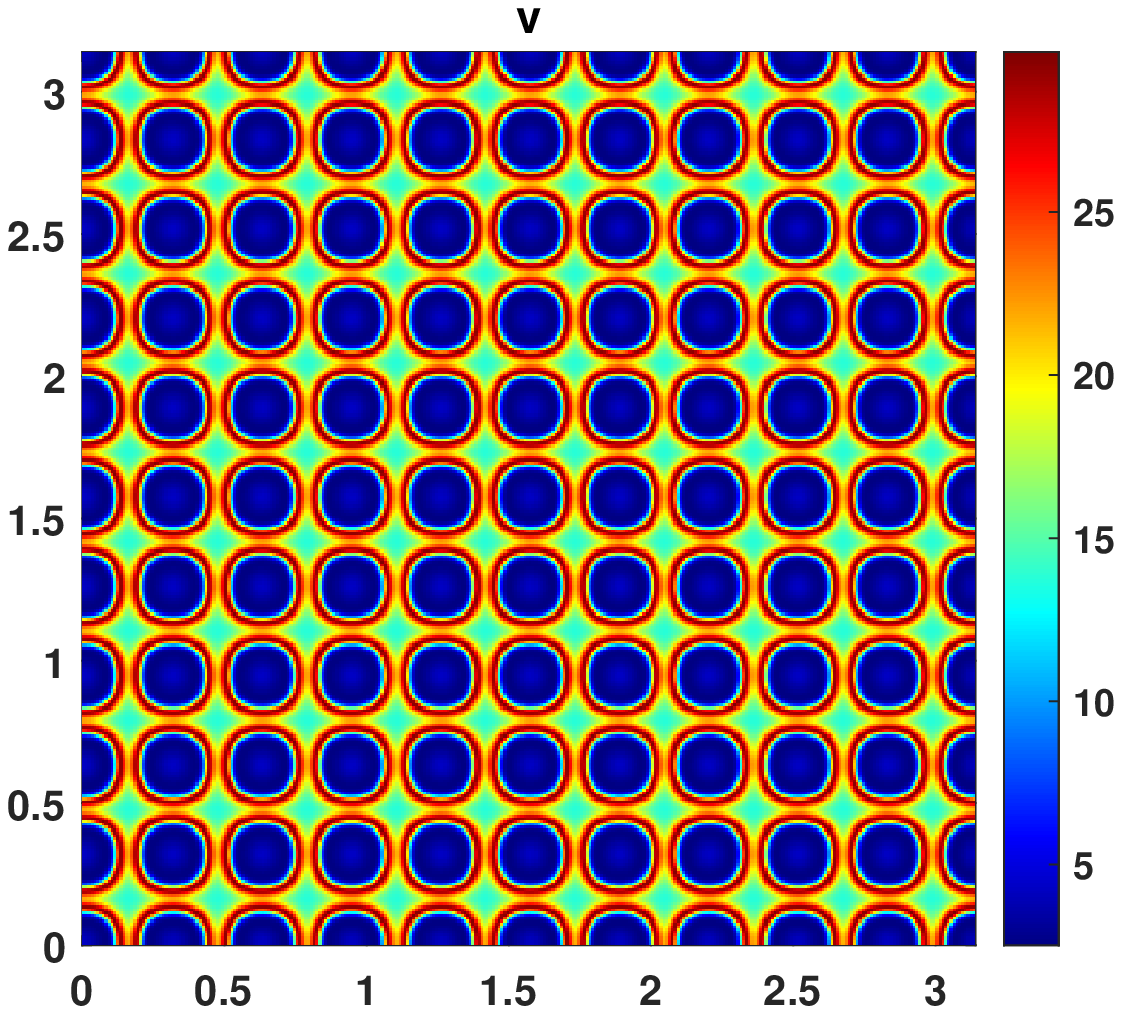} \hspace{-0.5cm} &  \includegraphics[scale=0.26]{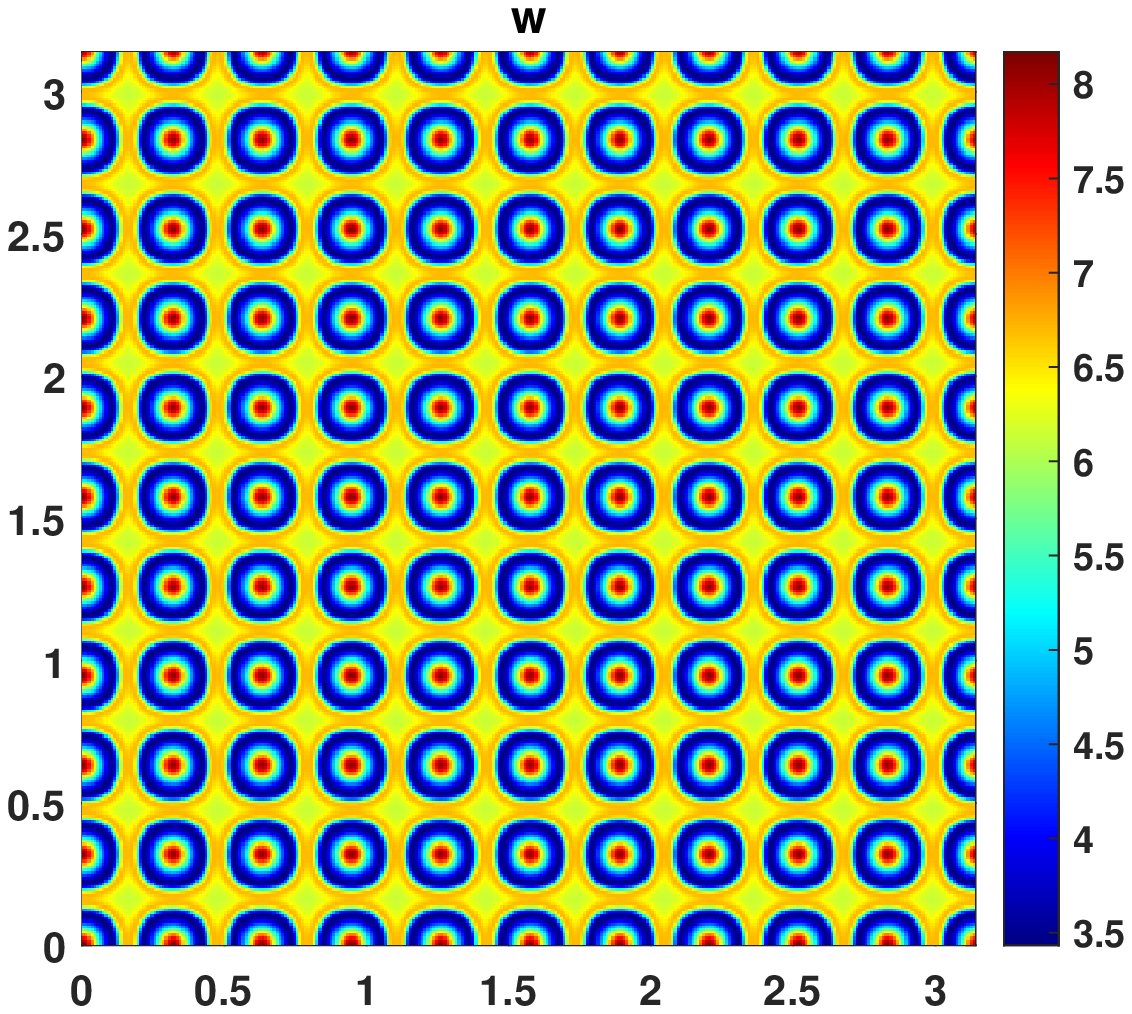}\\\vspace{-1cm} 
\begin{tabular}[c]{@{}c@{}c@{}c@{}c@{}}C \\\\\\\\ \\ \\\\\\\\\end{tabular}\hspace{-0.3cm} & \includegraphics[scale=0.26]{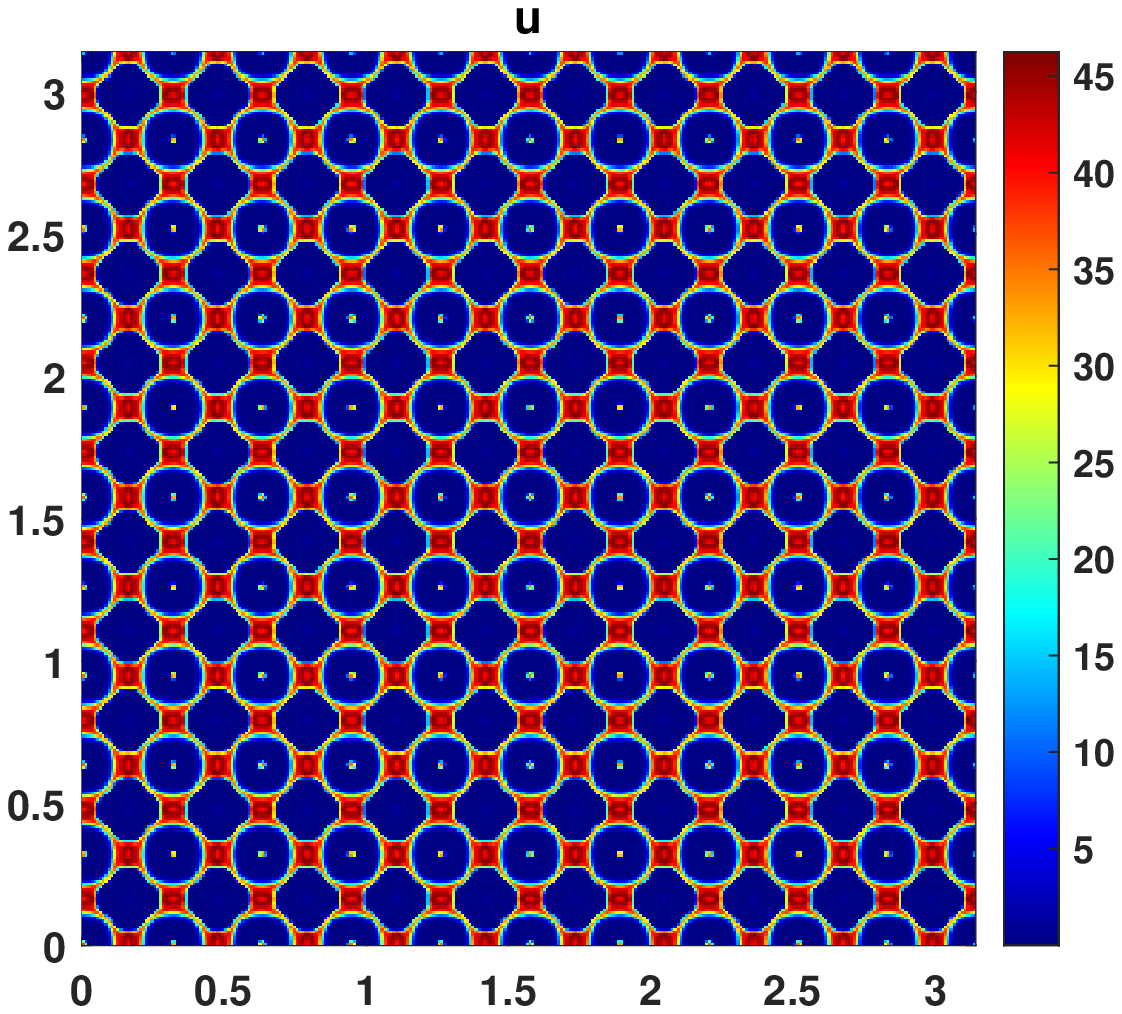}\hspace{-0.5cm}  &   \includegraphics[scale=0.26]{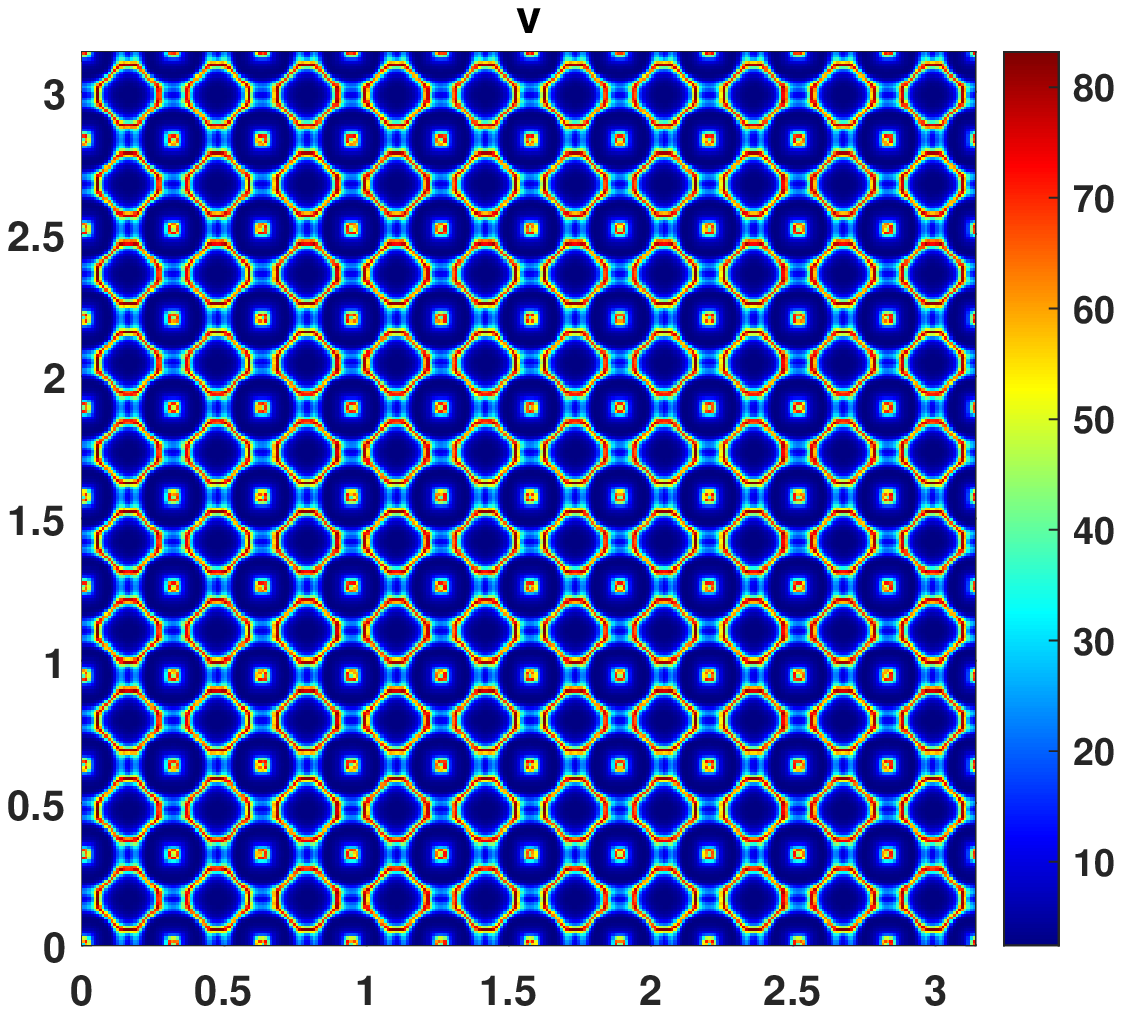} \hspace{-0.5cm} &  \includegraphics[scale=0.26]{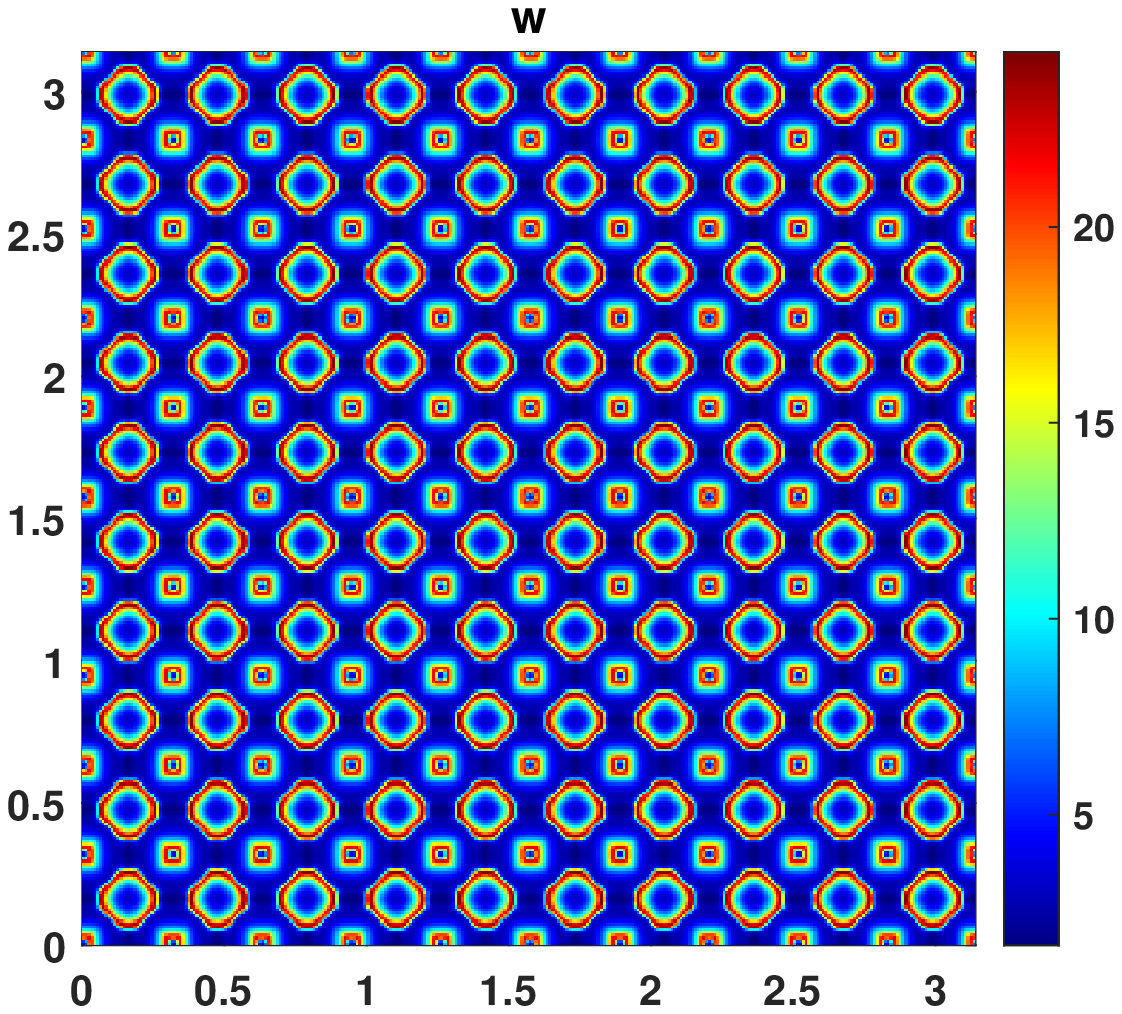} \\\vspace{-1cm} 
\begin{tabular}[c]{@{}c@{}c@{}c@{}c@{}}D \\\\\\ \\  \\\\\\\\\\\end{tabular}\hspace{-0.3cm} & \includegraphics[scale=0.26]{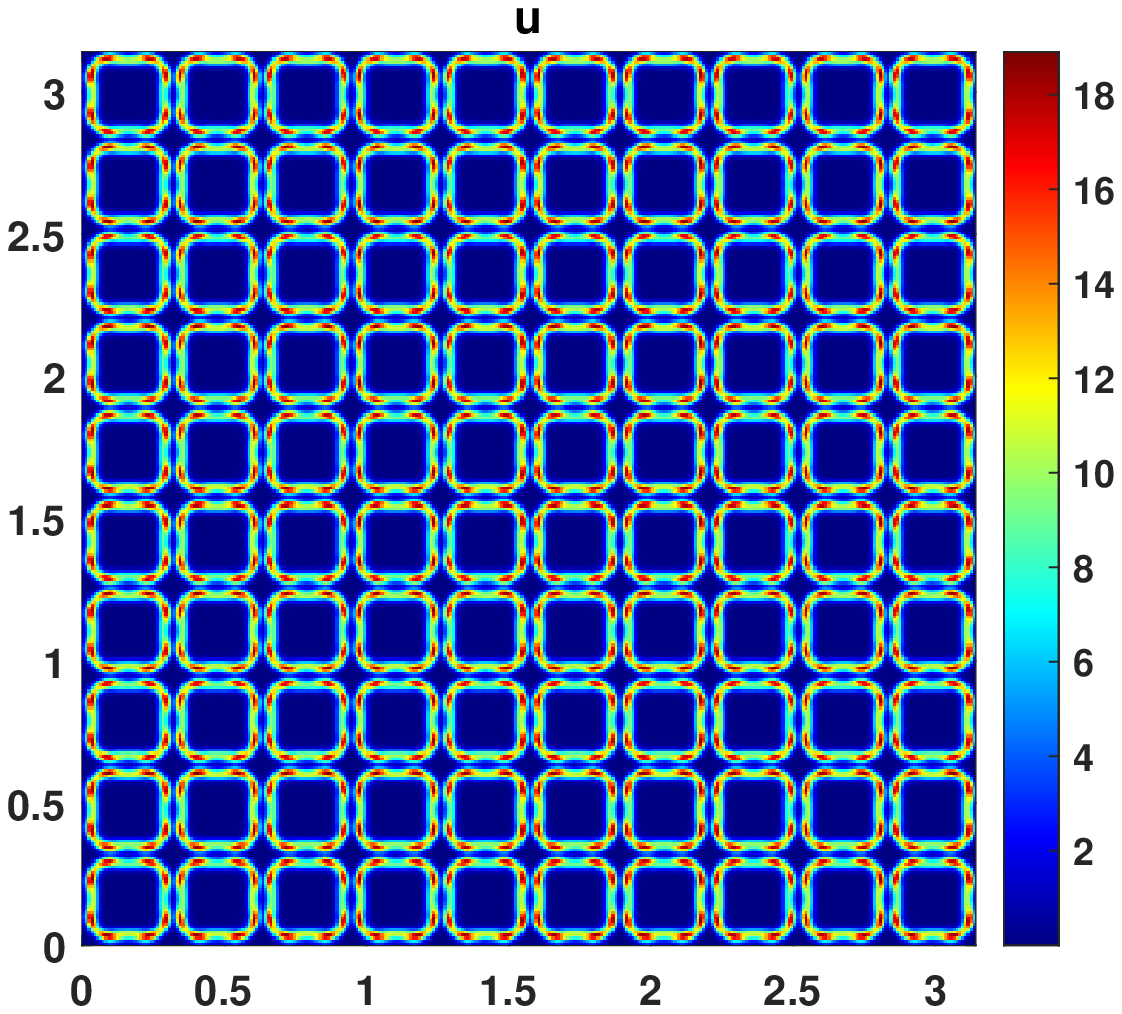}\hspace{-0.5cm}  &   \includegraphics[scale=0.26]{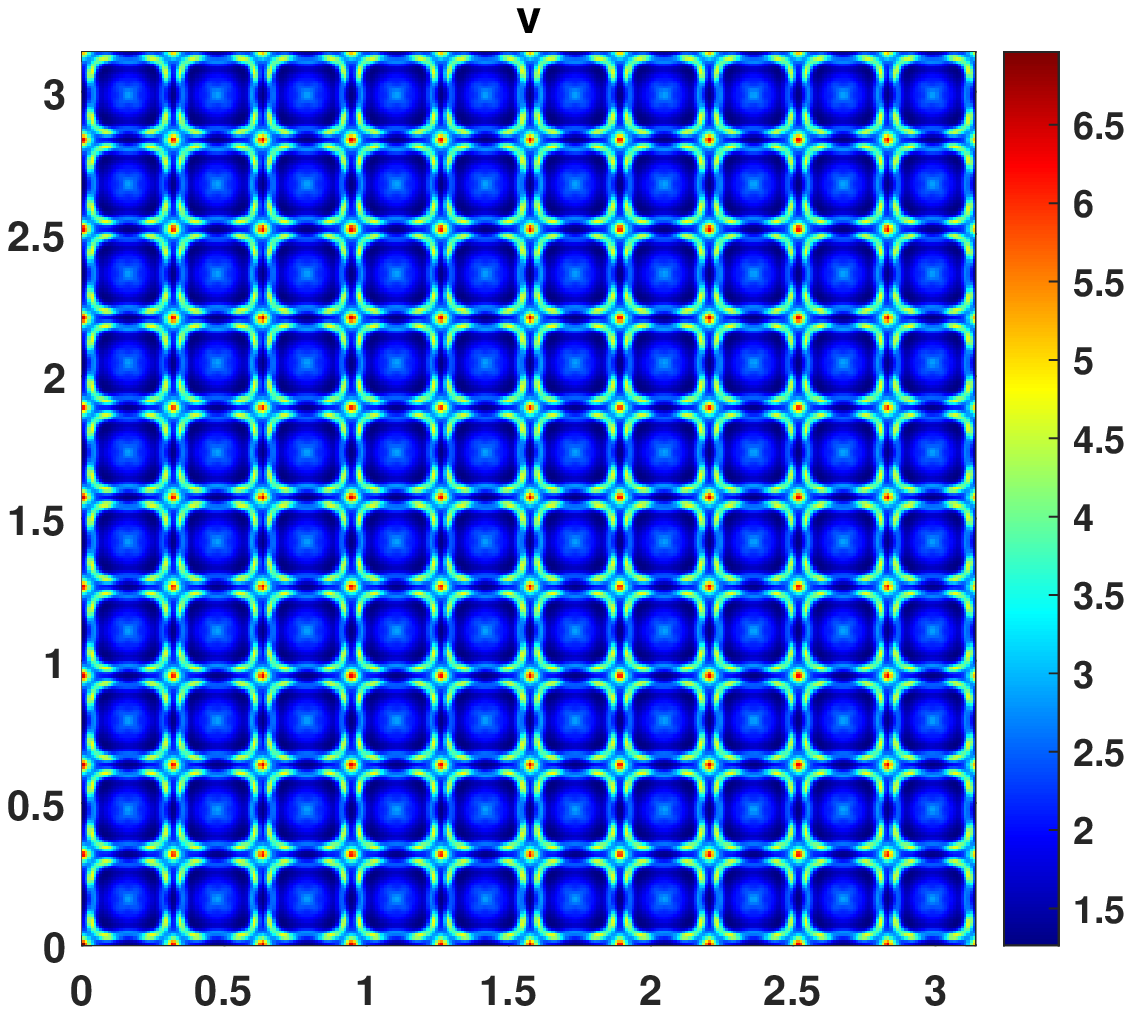} \hspace{-0.5cm} &  \includegraphics[scale=0.26]{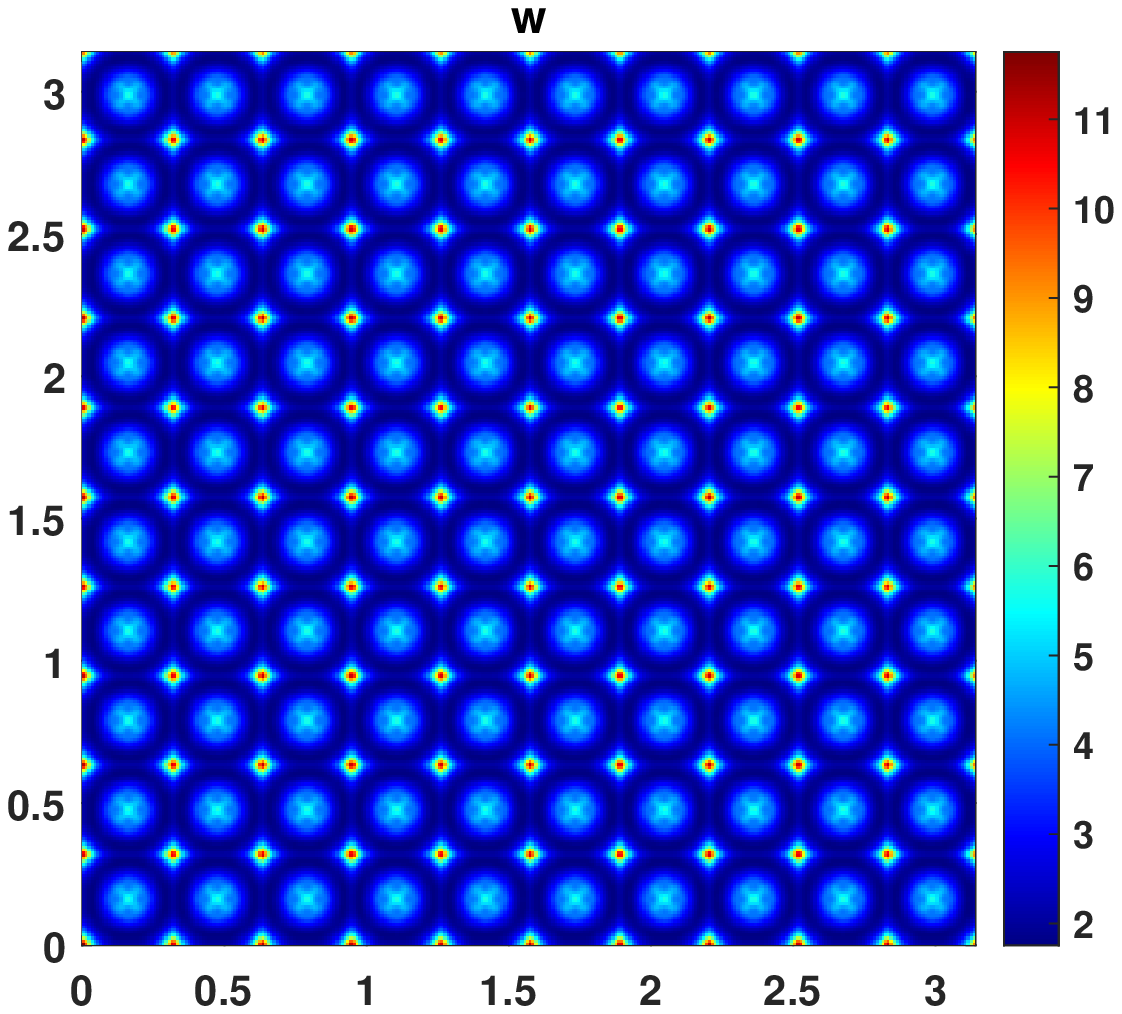} 
\end{tabular}
  \caption{Non-stationary non-Turing patterns seen in prey, susceptible predator and infected predator densities of the model system (\ref{eq2}), at (A) $t =200$,~(B)~ $t=300$,~(C)~ $t=500$, and ~(D)~ $t=1000$. The diffusive constants are $d_1=10^{-6}, d_2=10^{-6}$ and $d_3=10^{-10}$. Other parameter values are given in Table \ref{paratable}.  }
\label{table9}
\end{figure}

    \renewcommand{\thefigure}{\arabic{figure}}
 \begin{figure} [!ht]
 \centering
\begin{tabular}{cccc} \vspace{-1cm}
\begin{tabular}[c]{@{}c@{}c@{}c@{}c@{}}A\\\\\\  \\ \\ \\\\\\\\\end{tabular}\hspace{-0.3cm} &
 \includegraphics[scale=0.26]{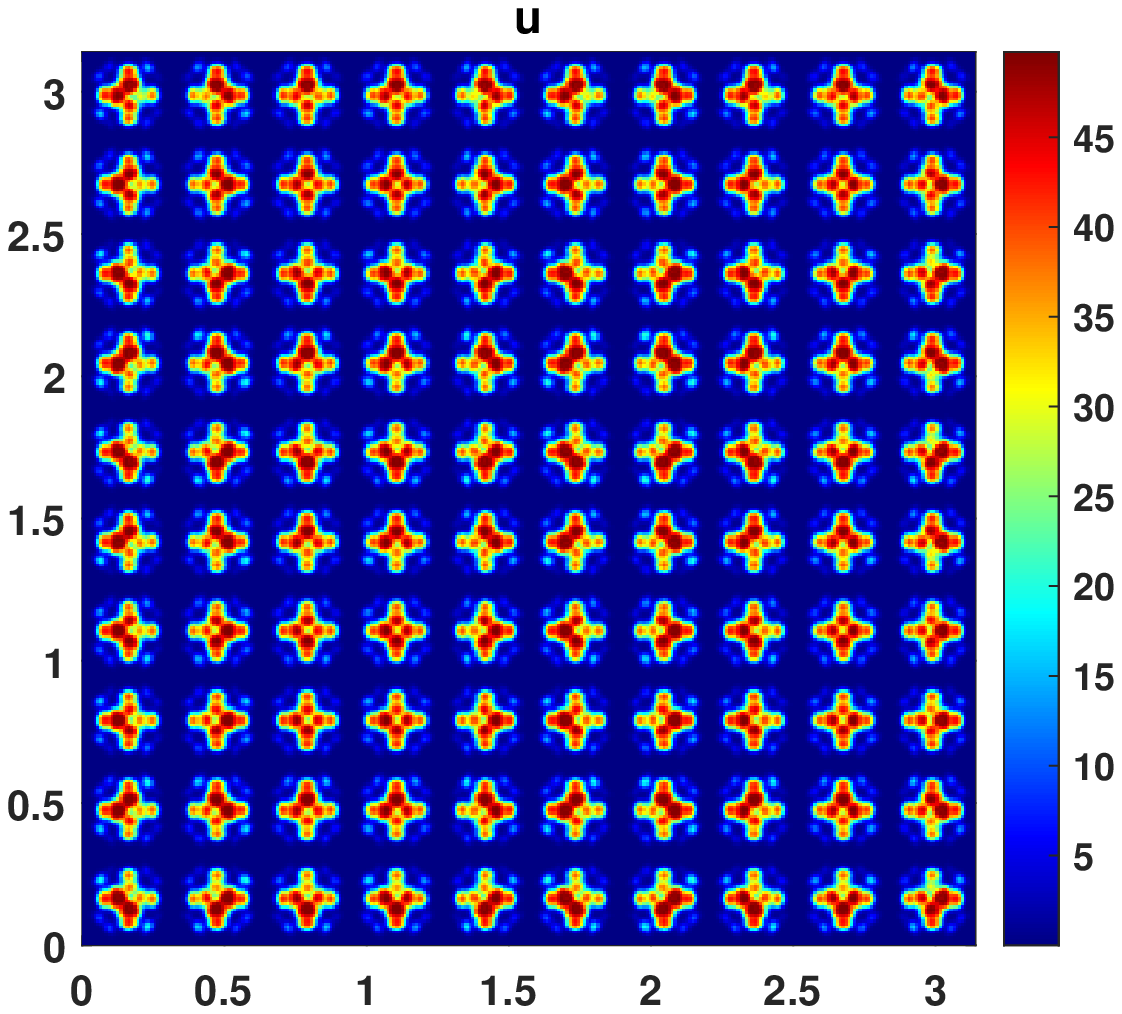} \hspace{-0.5cm} &    \includegraphics[scale=0.26]{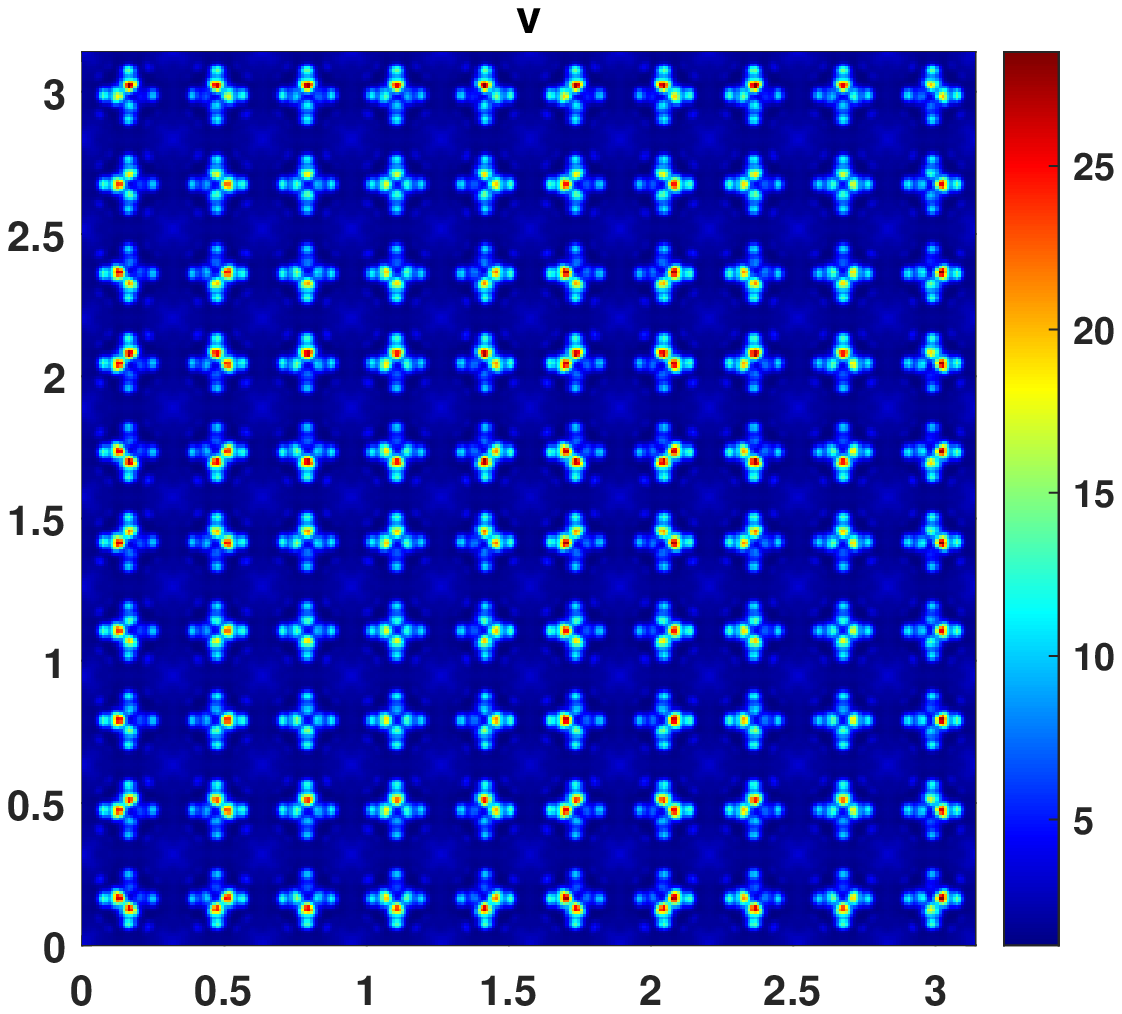} \hspace{-0.5cm} &  \includegraphics[scale=0.26]{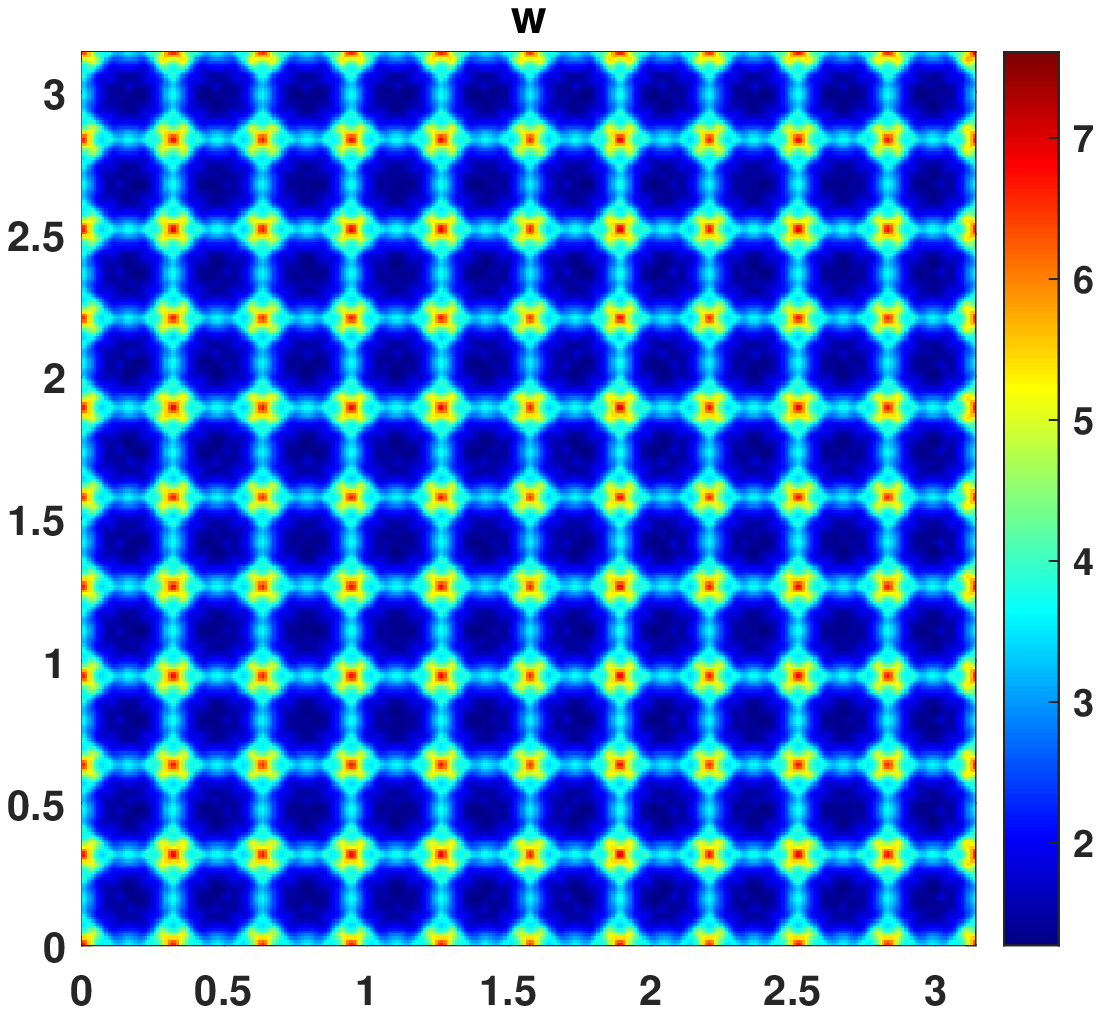}\\\vspace{-1cm}
\begin{tabular}[c]{@{}c@{}c@{}c@{}c@{}}B \\\\\\ \\ \\ \\\\\\\\\end{tabular}\hspace{-0.3cm} & \includegraphics[scale=0.26]{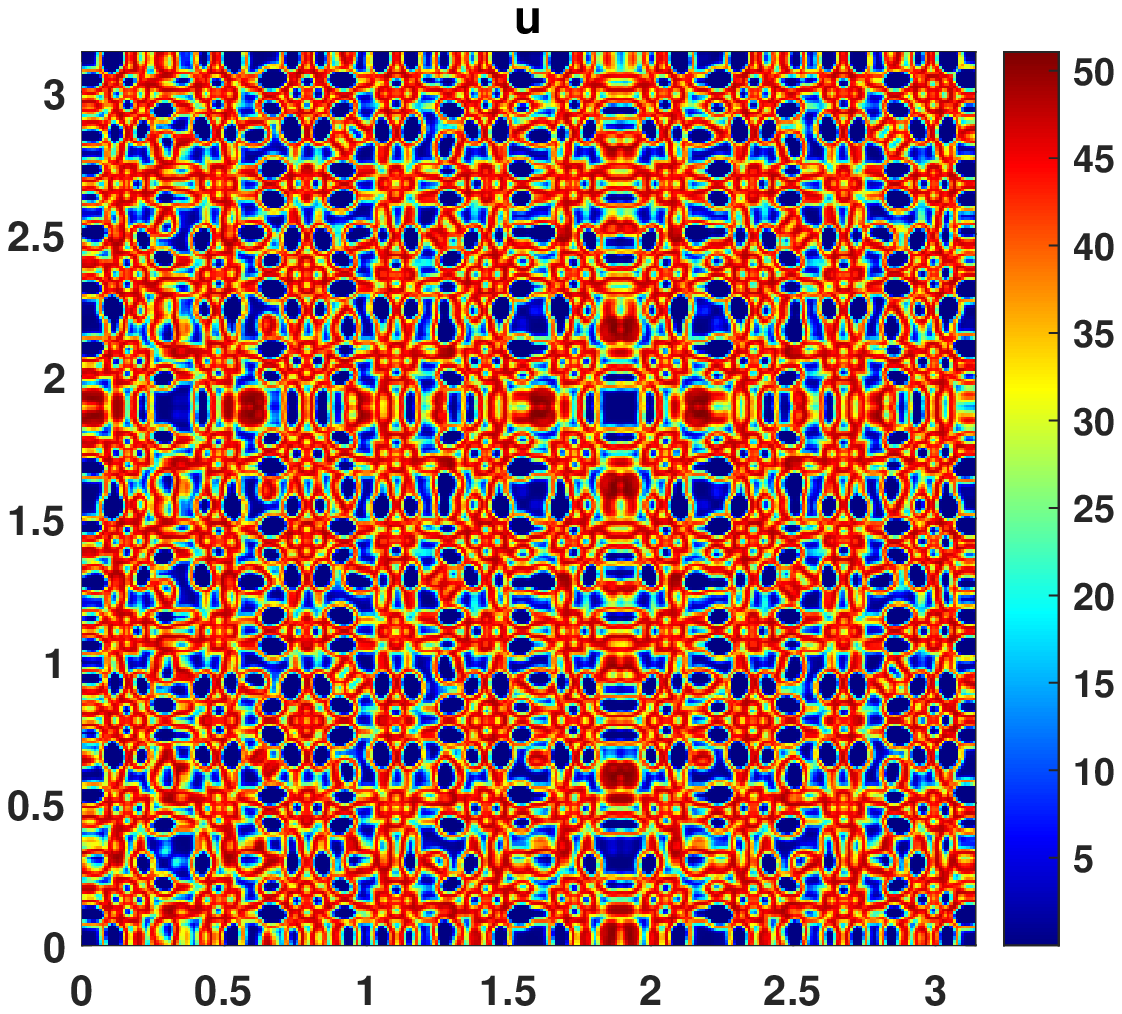}\hspace{-0.5cm}  &   \includegraphics[scale=0.26]{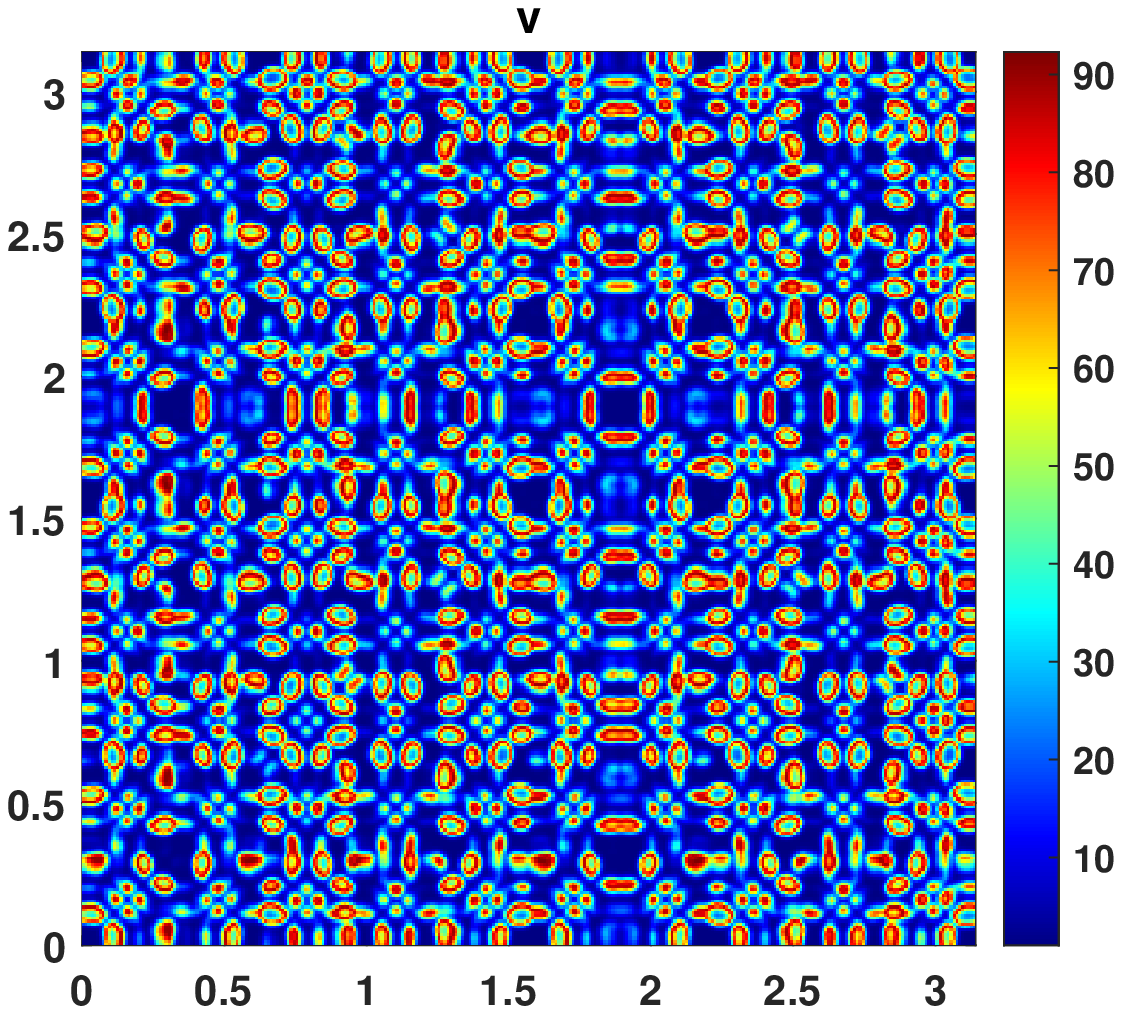} \hspace{-0.5cm} &  \includegraphics[scale=0.26]{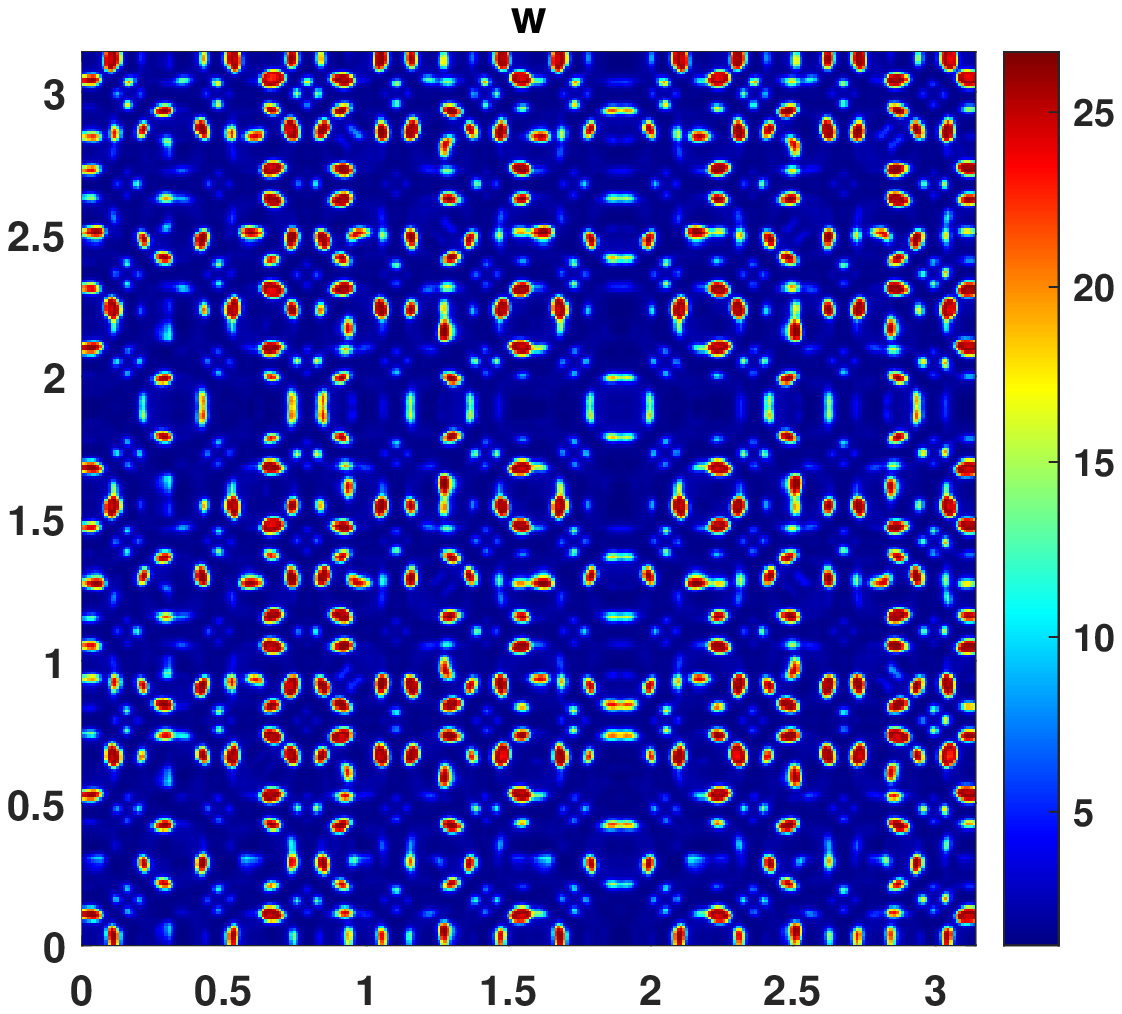}\\\vspace{-1cm} 
\begin{tabular}[c]{@{}c@{}c@{}c@{}c@{}}C \\\\\\ \\ \\ \\\\\\\\\end{tabular}\hspace{-0.3cm} & \includegraphics[scale=0.26]{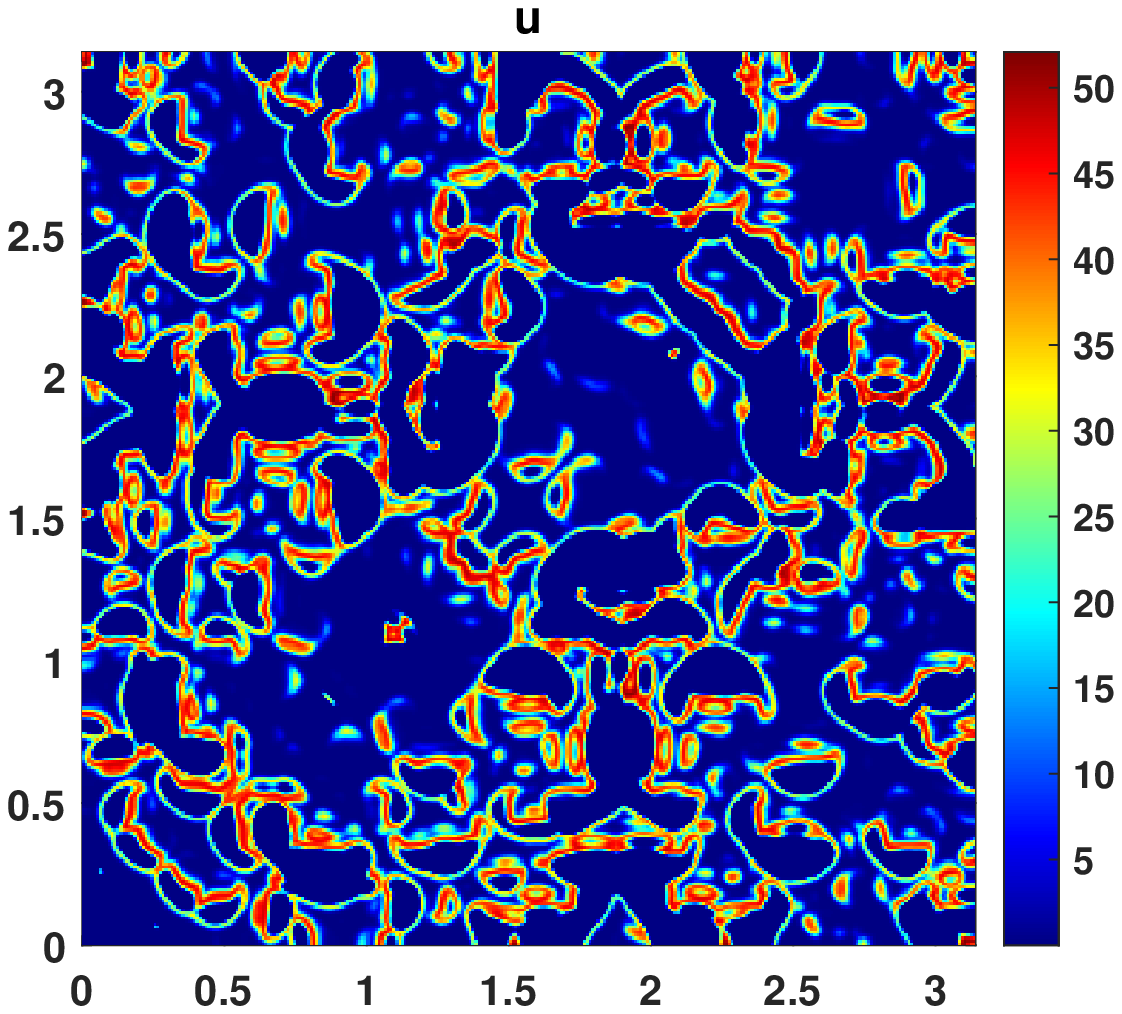}\hspace{-0.5cm}  &   \includegraphics[scale=0.26]{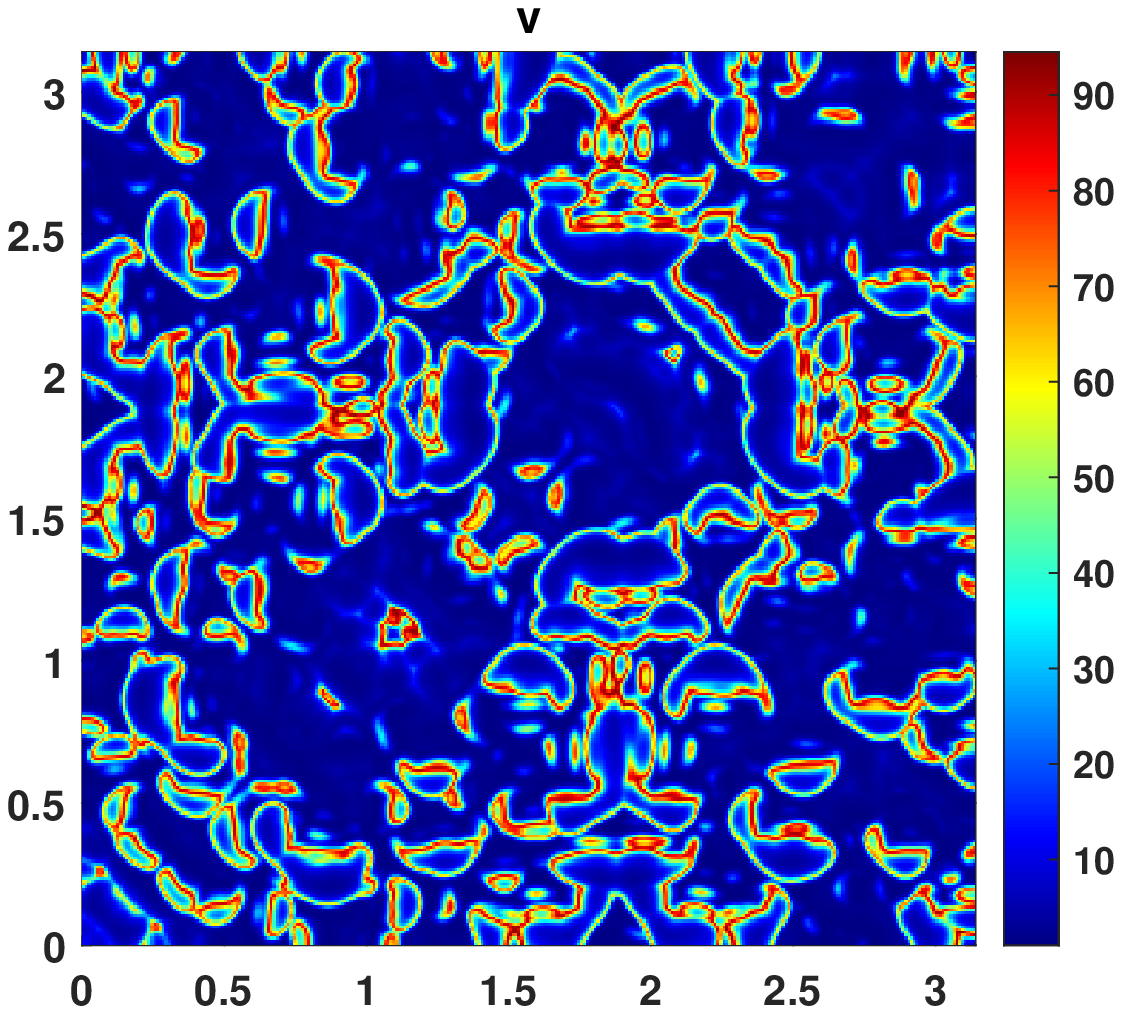} \hspace{-0.5cm} &  \includegraphics[scale=0.26]{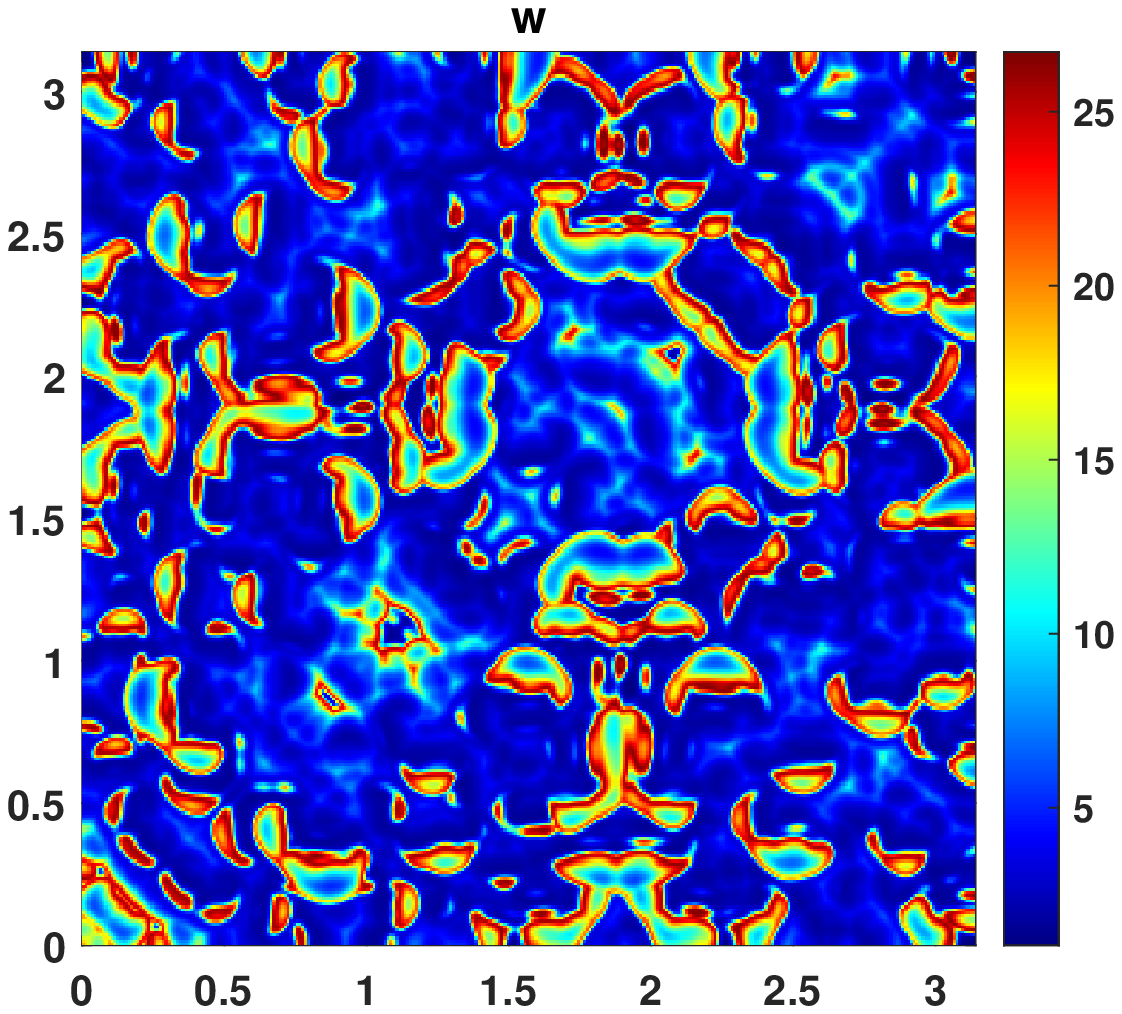} \\\vspace{-1cm} 
\begin{tabular}[c]{@{}c@{}c@{}c@{}c@{}}D\\\\\\  \\ \\ \\\\\\\\\end{tabular}\hspace{-0.3cm} & \includegraphics[scale=0.26]{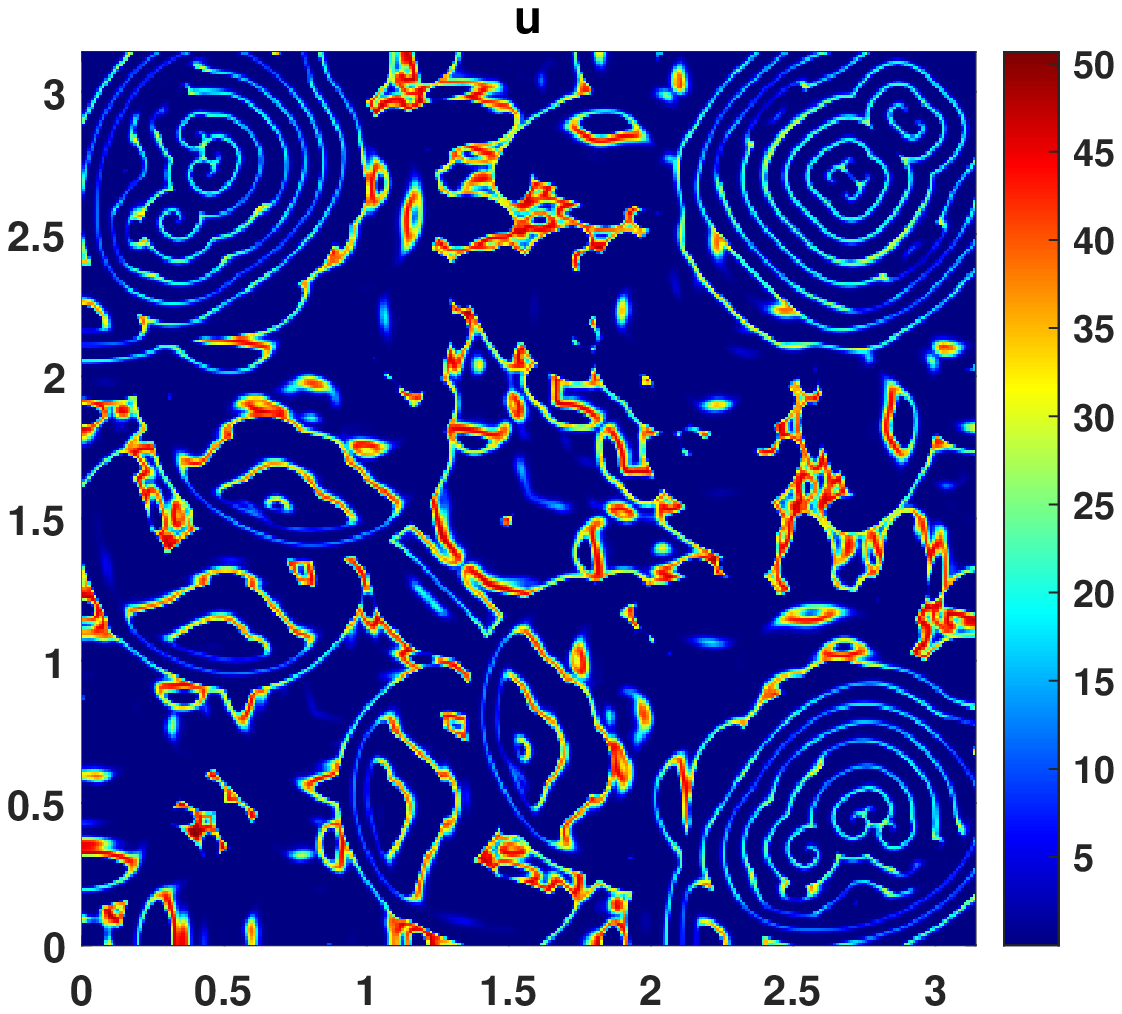}\hspace{-0.5cm}  &   \includegraphics[scale=0.26]{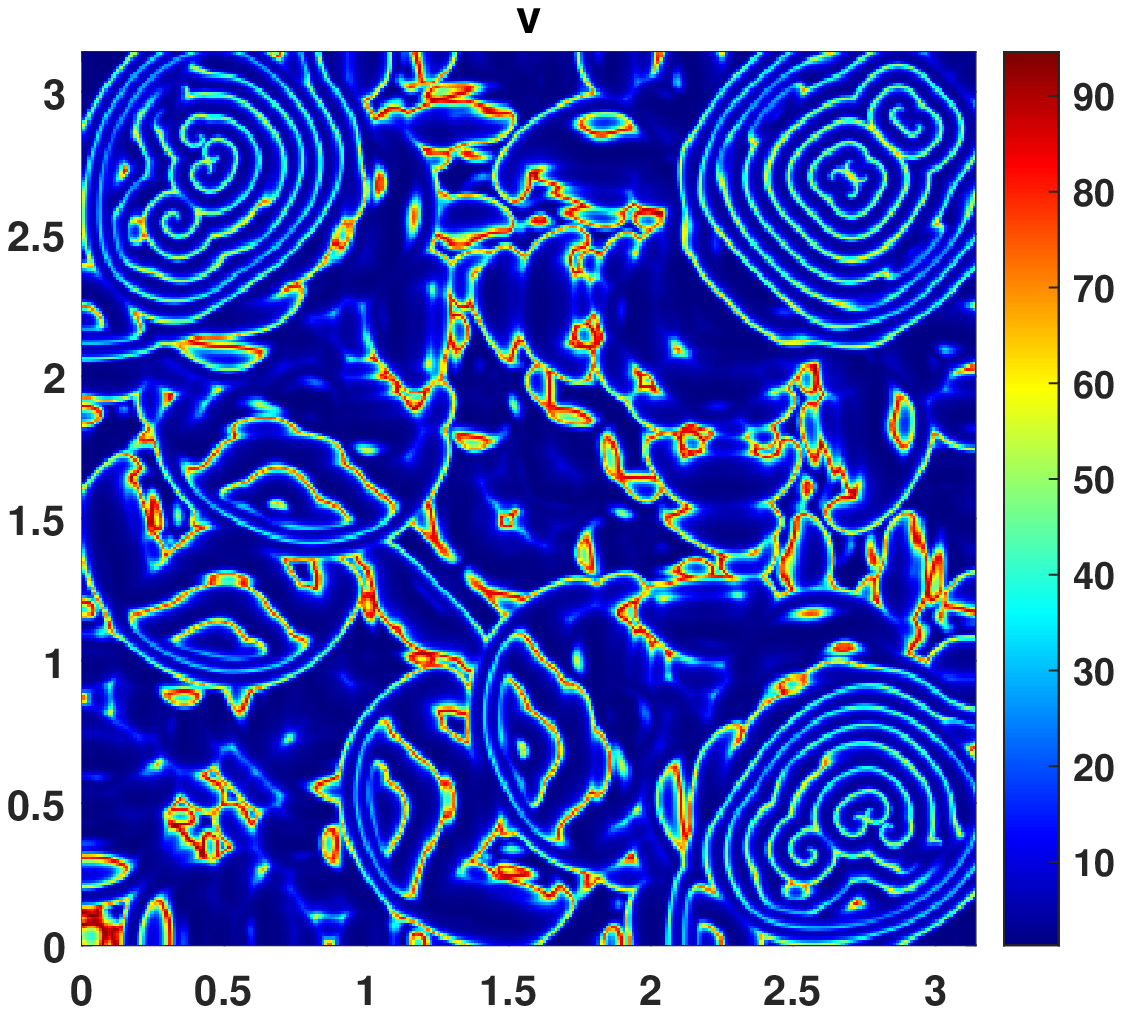} \hspace{-0.5cm} &  \includegraphics[scale=0.26]{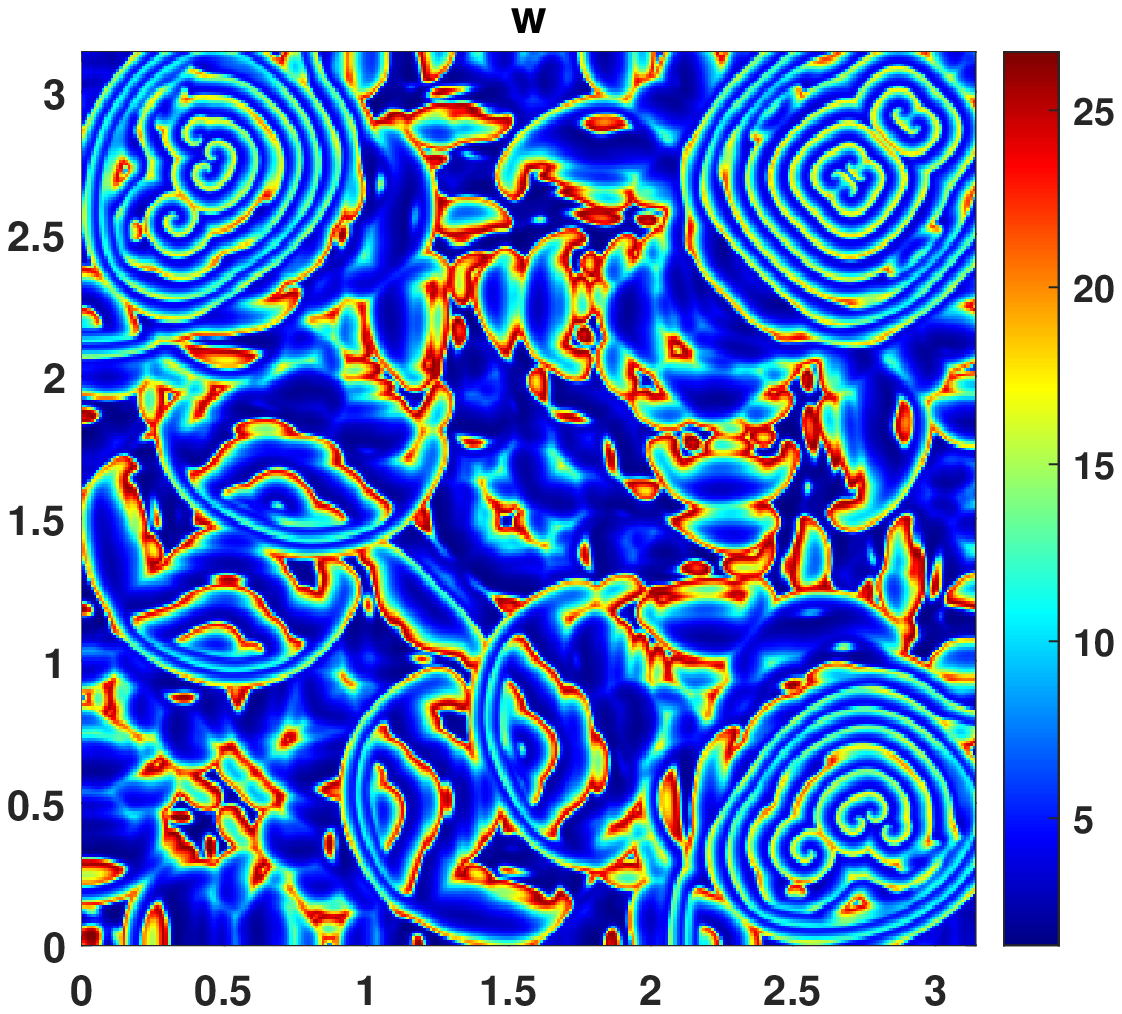} 
\end{tabular}
  \caption{Non-stationary non-Turing patterns seen in prey, susceptible predator and infected predator densities of the model system (\ref{eq2}), at (A)~ $t=1500$, ~(B)~$t=2000$,~(C)~$t=4000$, and ~(D)~ $t=10000.$ The diffusive constants are $d_1=10^{-6}, d_2=10^{-6}$ and $d_3=10^{-10}$. Other parameter values are given in Table \ref{paratable}.  }
\label{table91}
\end{figure} 

 \renewcommand{\thefigure}{\arabic{figure}}
 \begin{figure} [!ht]
 \centering
\begin{tabular}{cccc} \vspace{-1cm}
\begin{tabular}[c]{@{}c@{}c@{}c@{}c@{}}A  \\\\\\\\ \\ \\\\\\\\\end{tabular}\hspace{-0.3cm} &
 \includegraphics[scale=0.26]{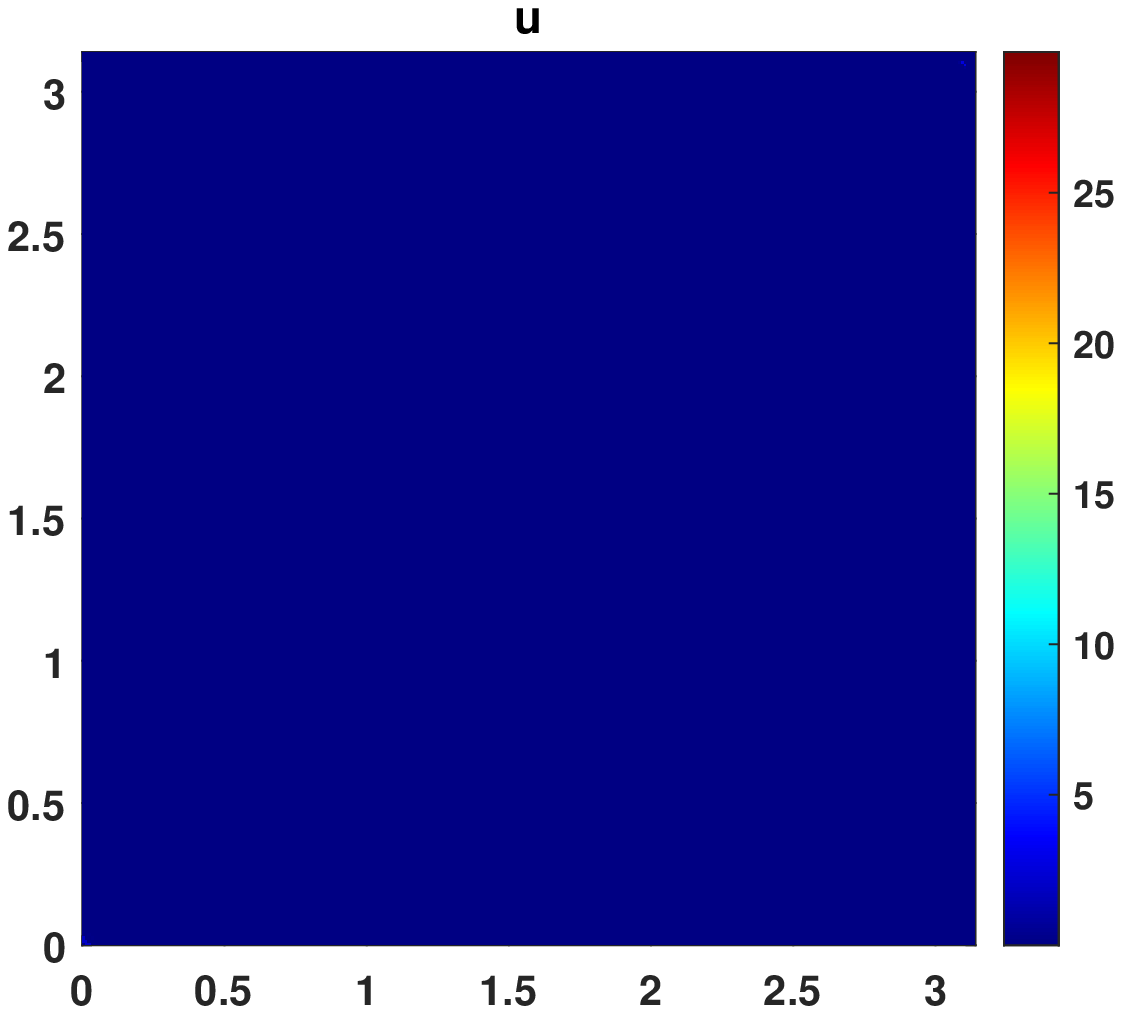} \hspace{-0.5cm} &    \includegraphics[scale=0.26]{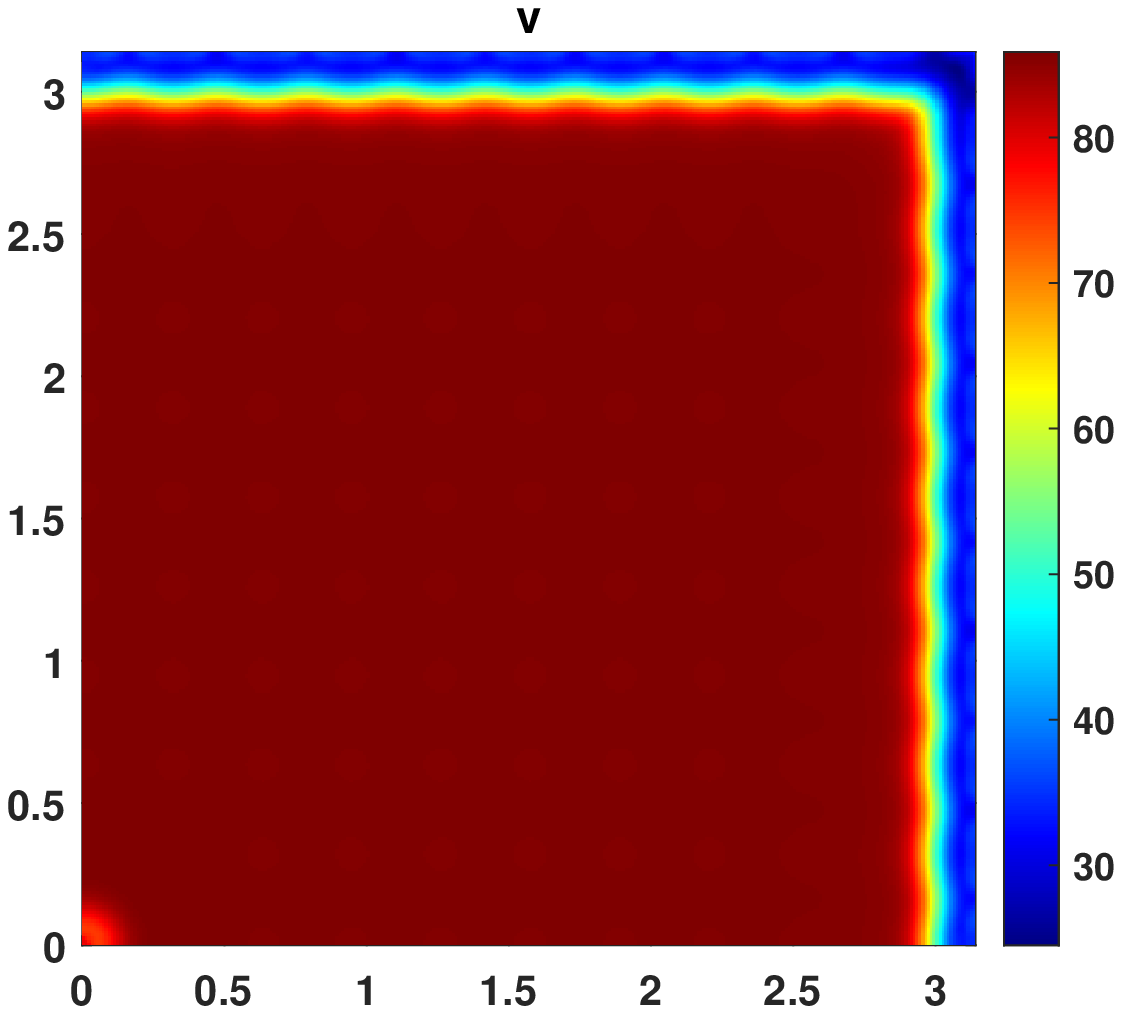} \hspace{-0.5cm} &  \includegraphics[scale=0.26]{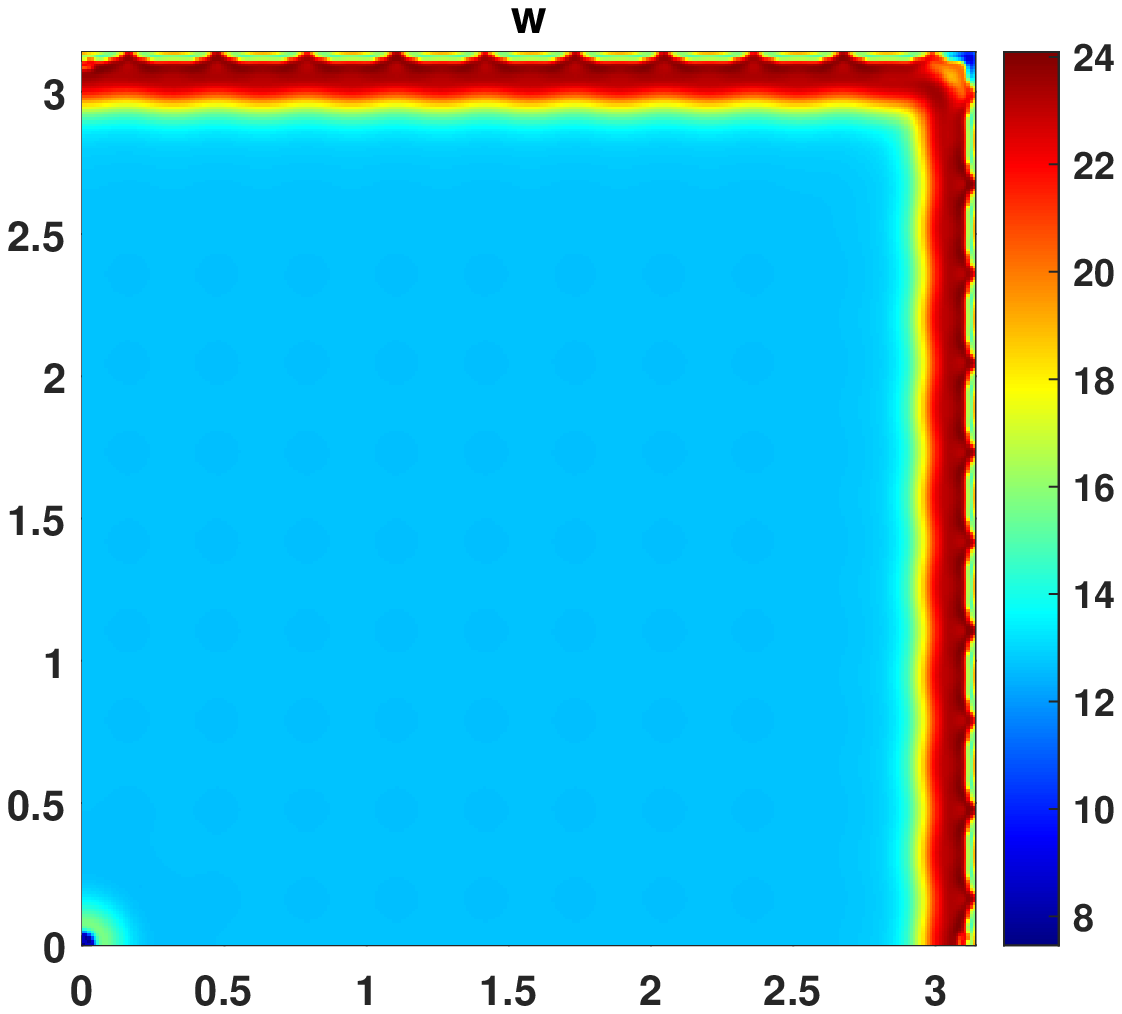}\\\vspace{-1cm}
 \begin{tabular}[c]{@{}c@{}c@{}c@{}c@{}}B  \\\\ \\\\\\ \\\\\\\\\end{tabular}\hspace{-0.3cm} &
 \includegraphics[scale=0.26]{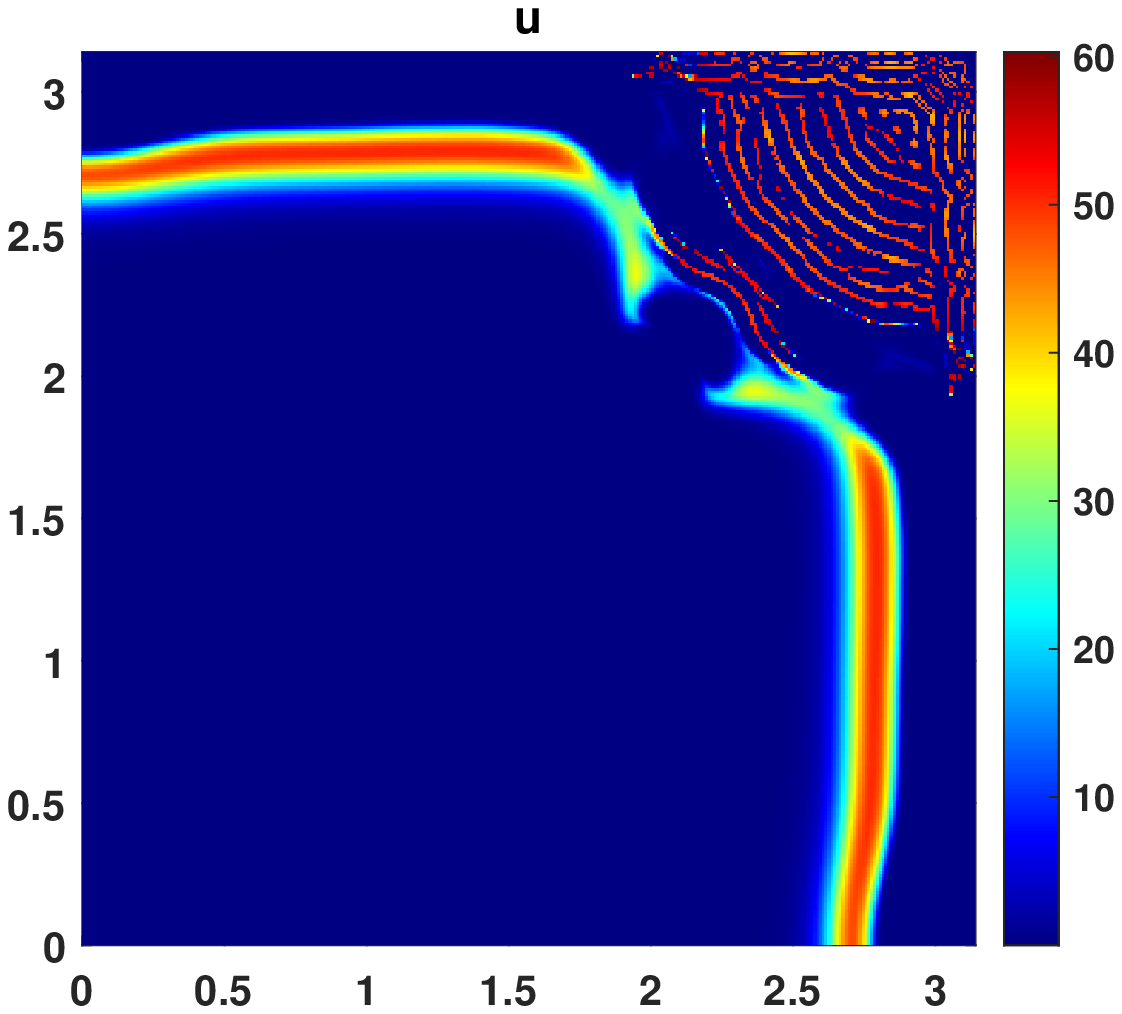} \hspace{-0.5cm} &    \includegraphics[scale=0.26]{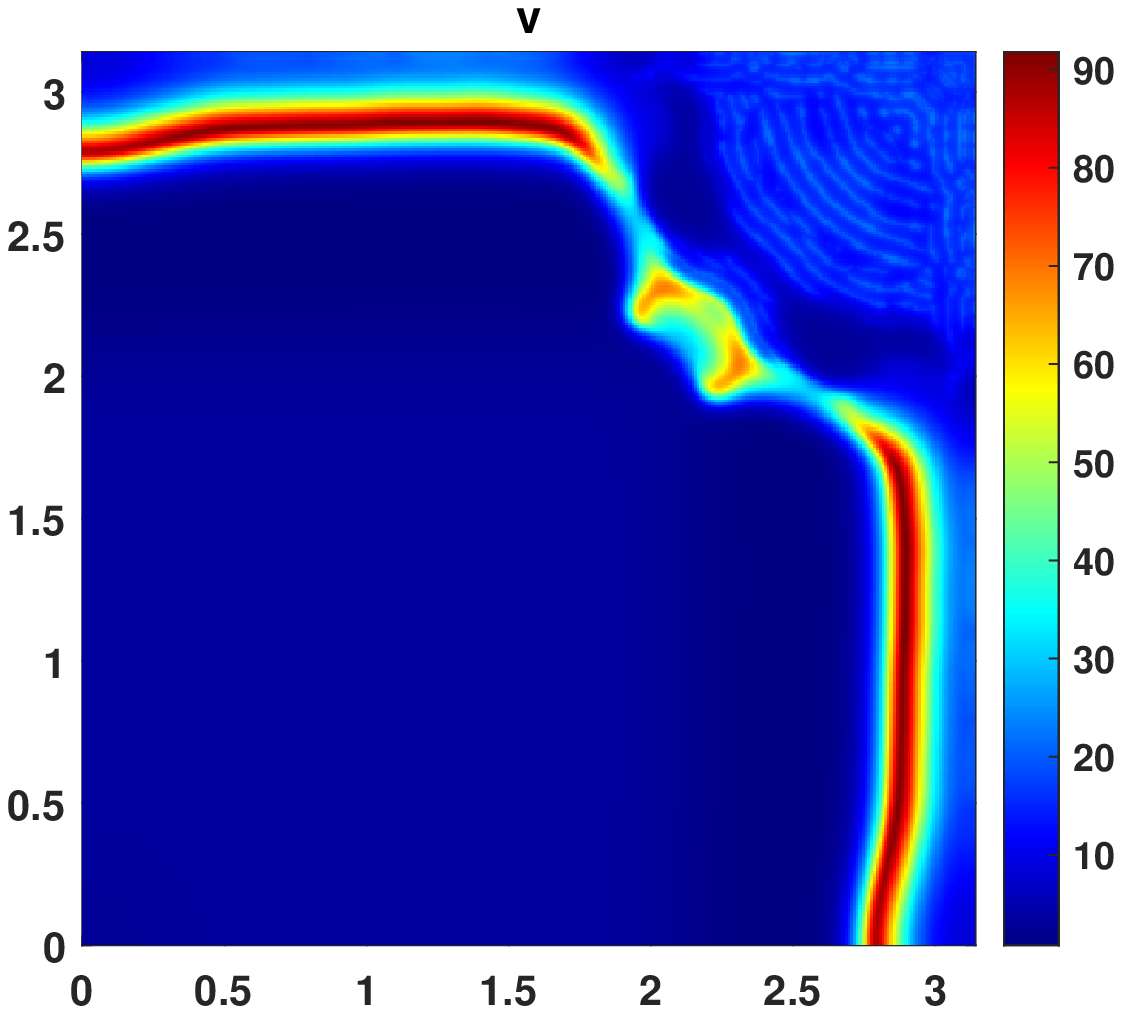} \hspace{-0.5cm} &  \includegraphics[scale=0.26]{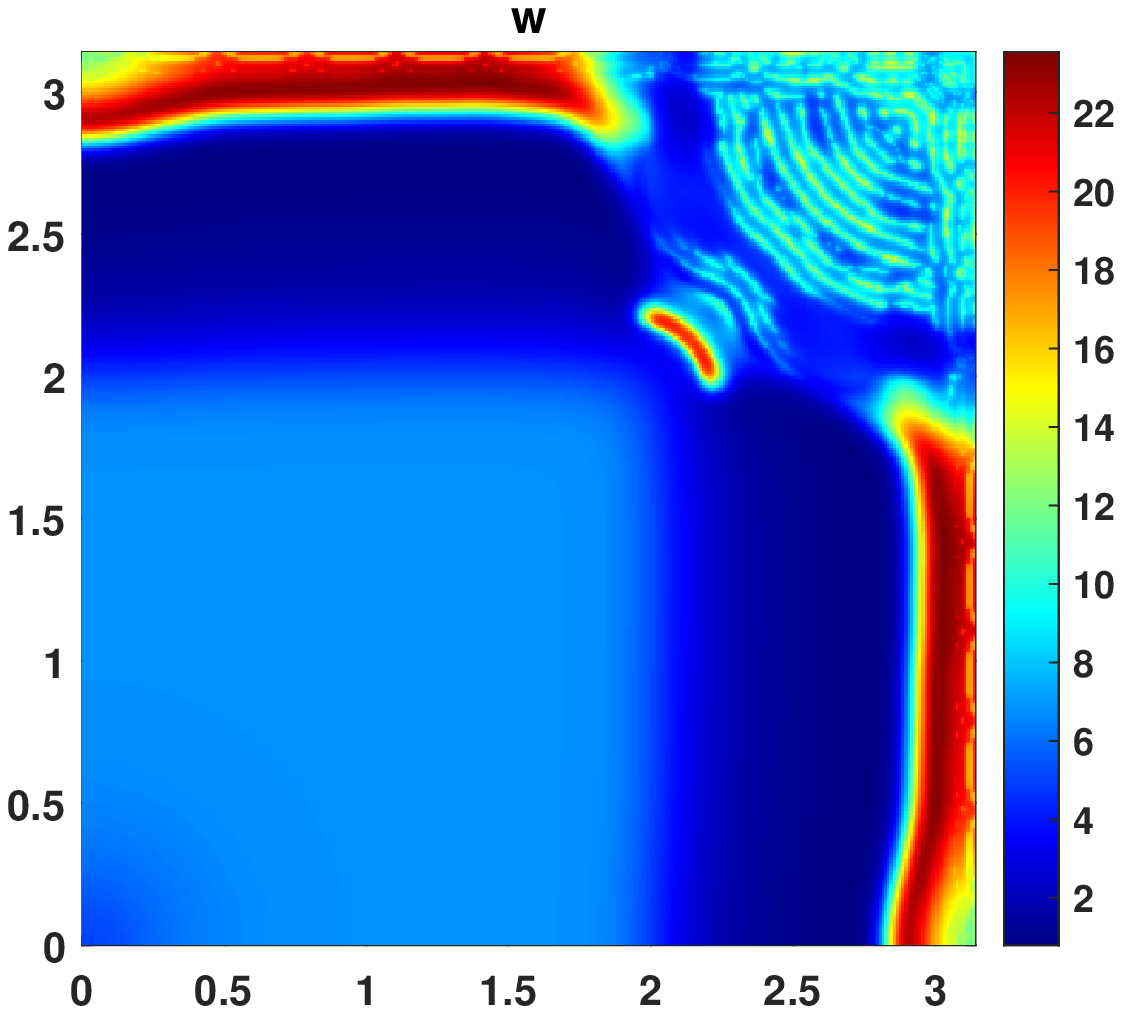}\\\vspace{-1cm}
\begin{tabular}[c]{@{}c@{}c@{}c@{}c@{}}C  \\\\  \\\\\\\\\\\\\\\end{tabular}\hspace{-0.3cm} & \includegraphics[scale=0.26]{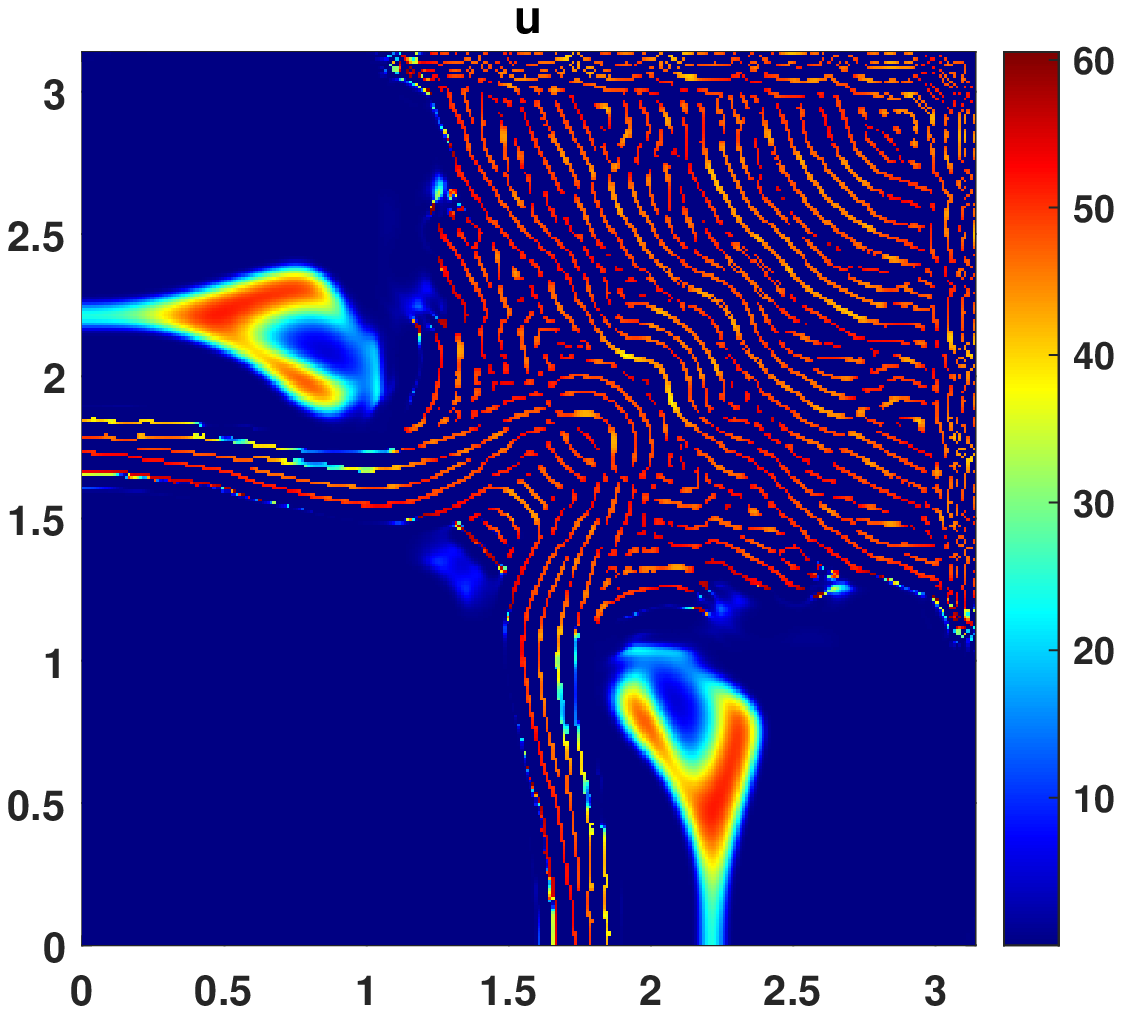}\hspace{-0.5cm}  &   \includegraphics[scale=0.26]{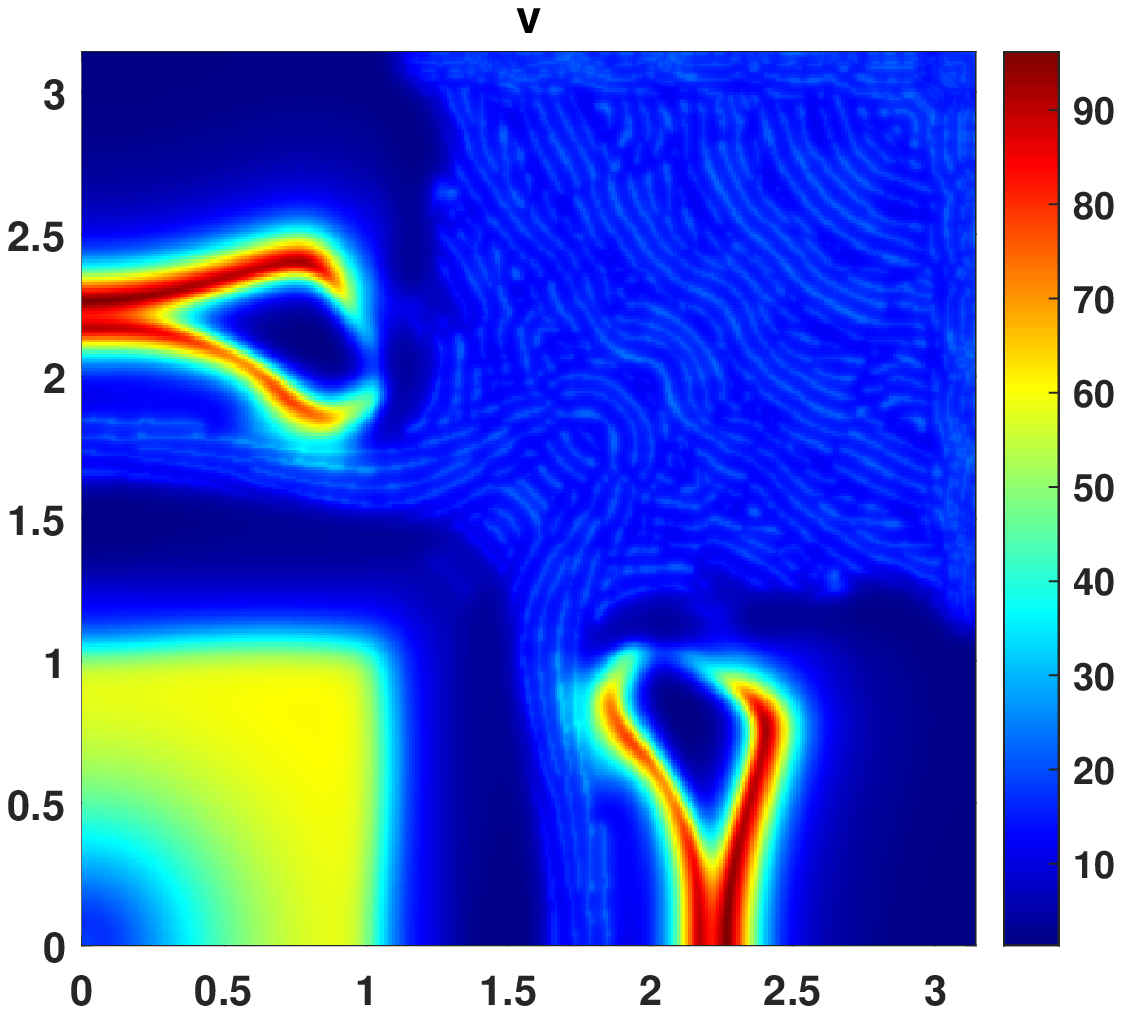} \hspace{-0.5cm} &  \includegraphics[scale=0.26]{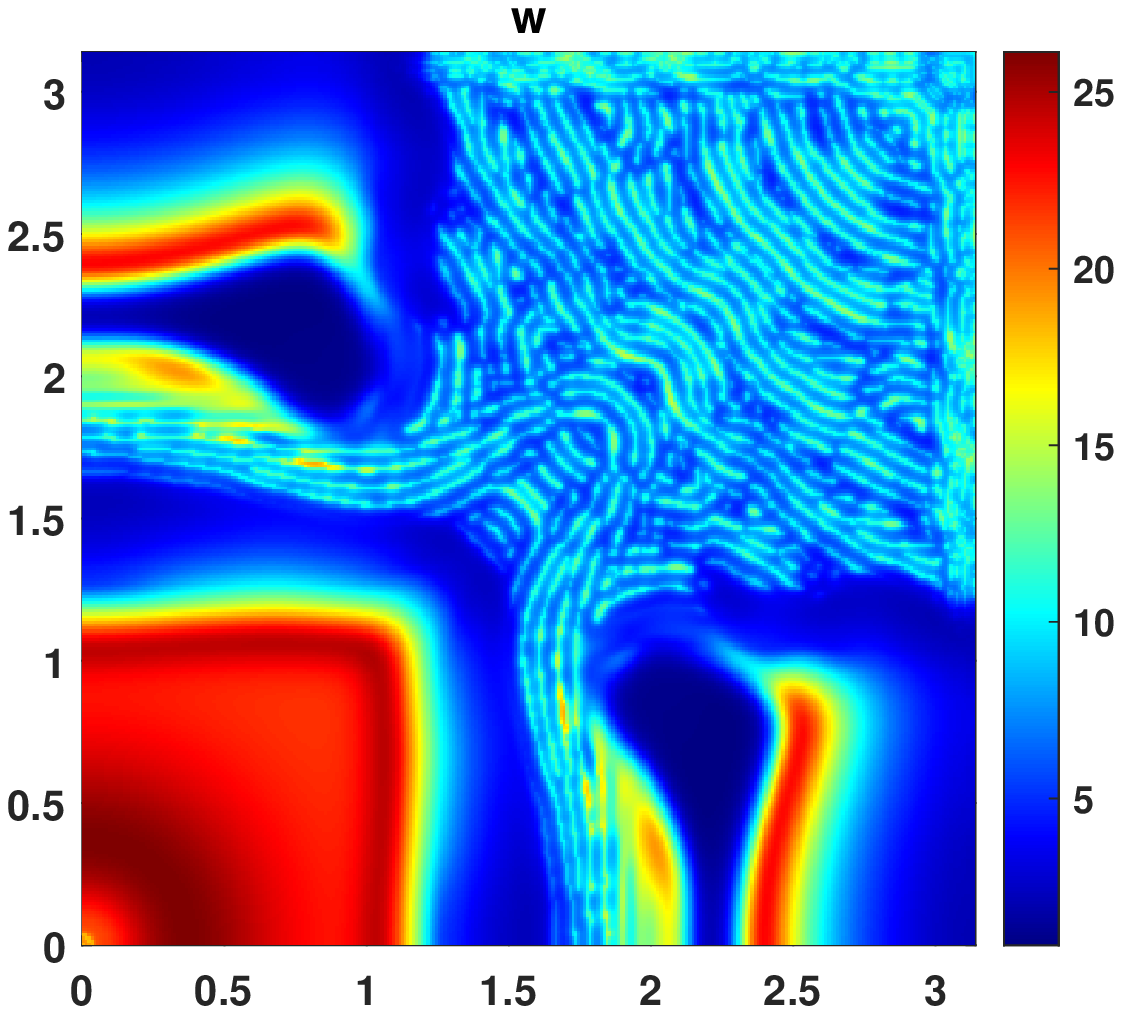}\\\vspace{-1cm} 
\begin{tabular}[c]{@{}c@{}c@{}c@{}c@{}}D  \\\\\\\\ \\ \\\\\\\\\end{tabular}\hspace{-0.3cm} & \includegraphics[scale=0.26]{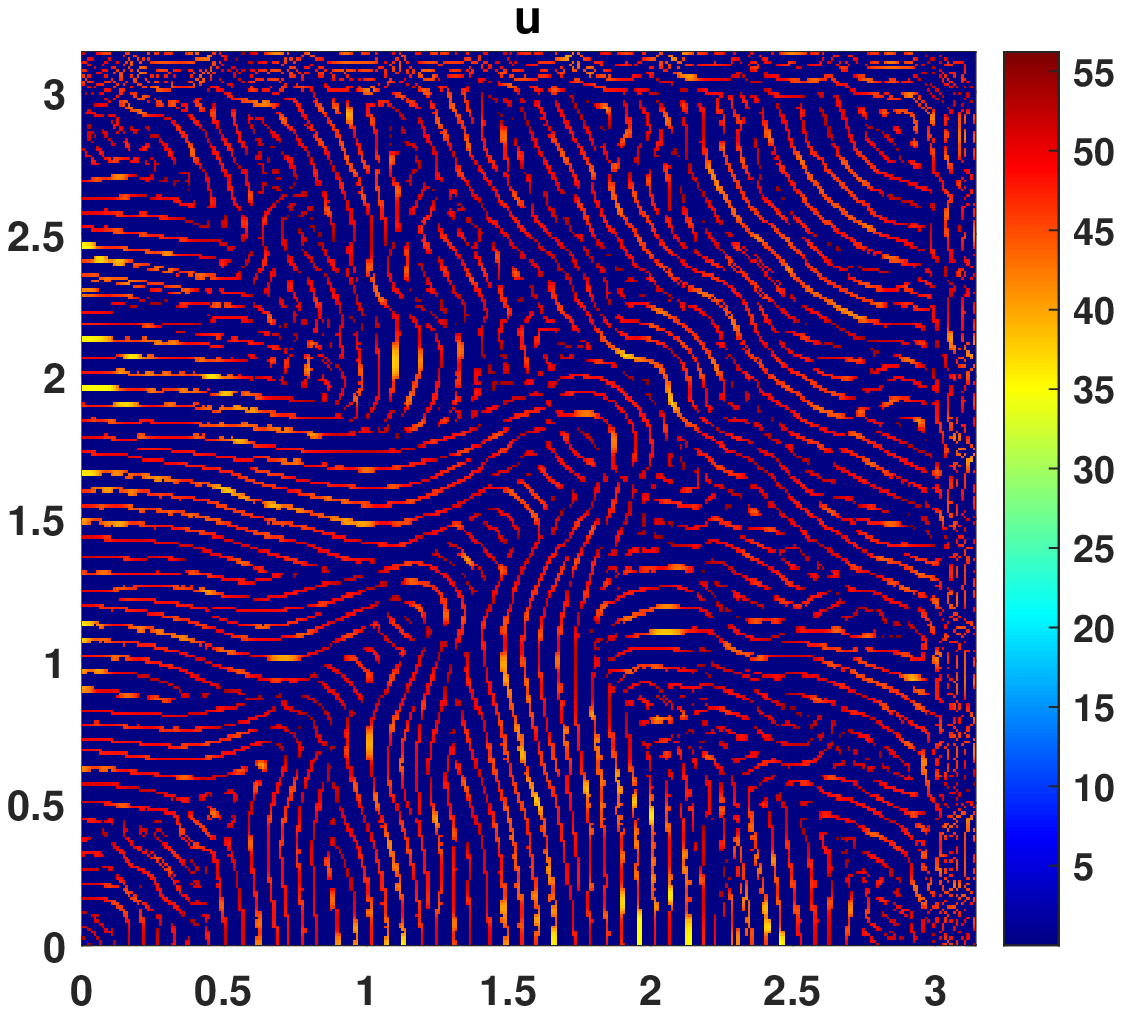}\hspace{-0.5cm}  &   \includegraphics[scale=0.26]{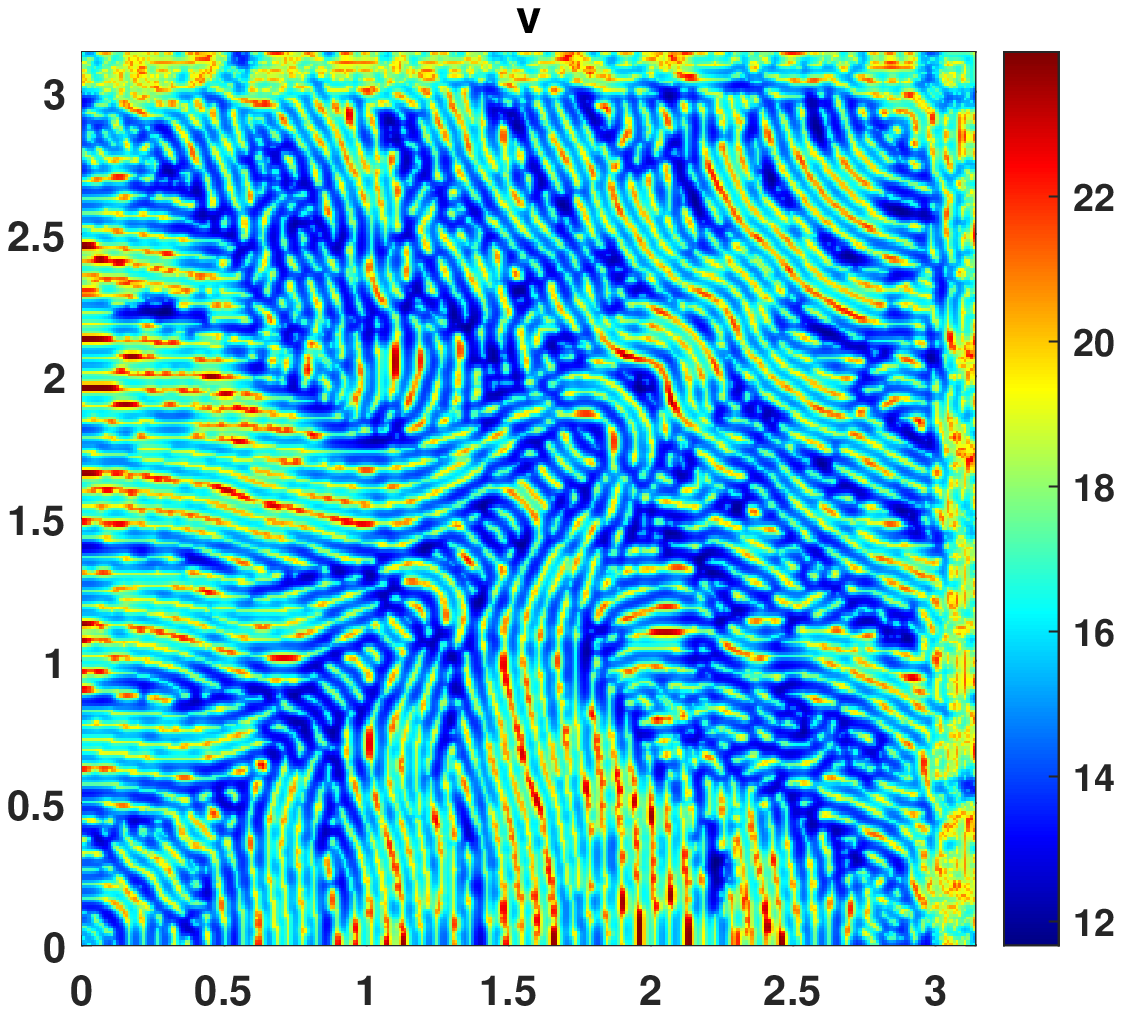} \hspace{-0.5cm} &  \includegraphics[scale=0.26]{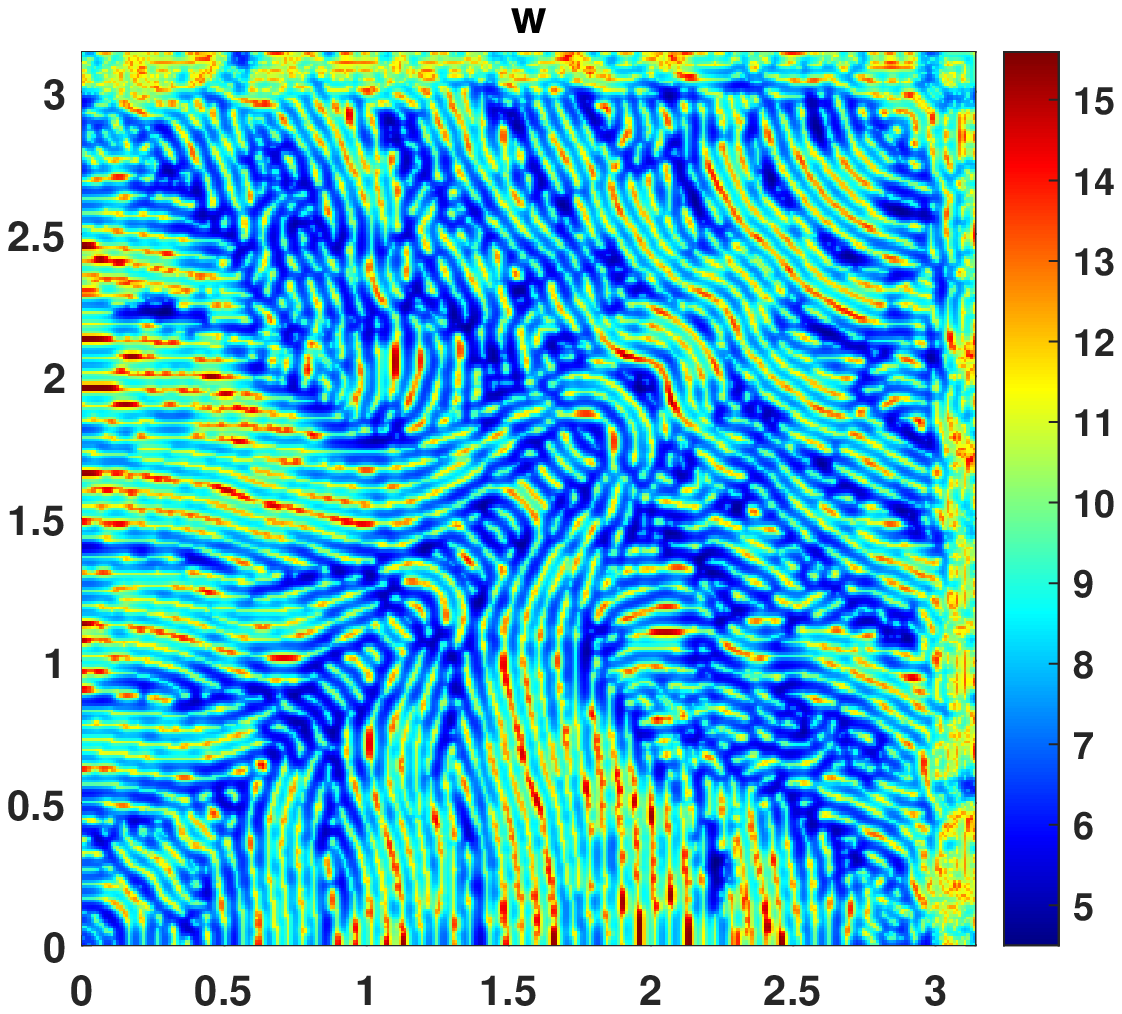} 
\end{tabular}
 \caption{Stationary non-Turing patterns seen in prey, susceptible predator and infected predator densities of the model system (\ref{eq2}), at (A) $t =500$~ (B) $t =2000$,~(C)~ $t=3000$, and (D)~ $t=5000$. The diffusive constants are $d_1=10^{-10}, d_2=10^{-4}$ and $d_3=10^{-10}$. Other parameter values are given in Table \ref{paratable}. }
\label{table10}
\end{figure}

 \renewcommand{\thefigure}{\arabic{figure}}
 \begin{figure} [!ht]
 \centering
\begin{tabular}{cccc} \vspace{-1cm}
\begin{tabular}[c]{@{}c@{}c@{}c@{}c@{}}A  \\ \\\\\\\\ \\\\\\\\\end{tabular}\hspace{-0.3cm} &
 \includegraphics[scale=0.26]{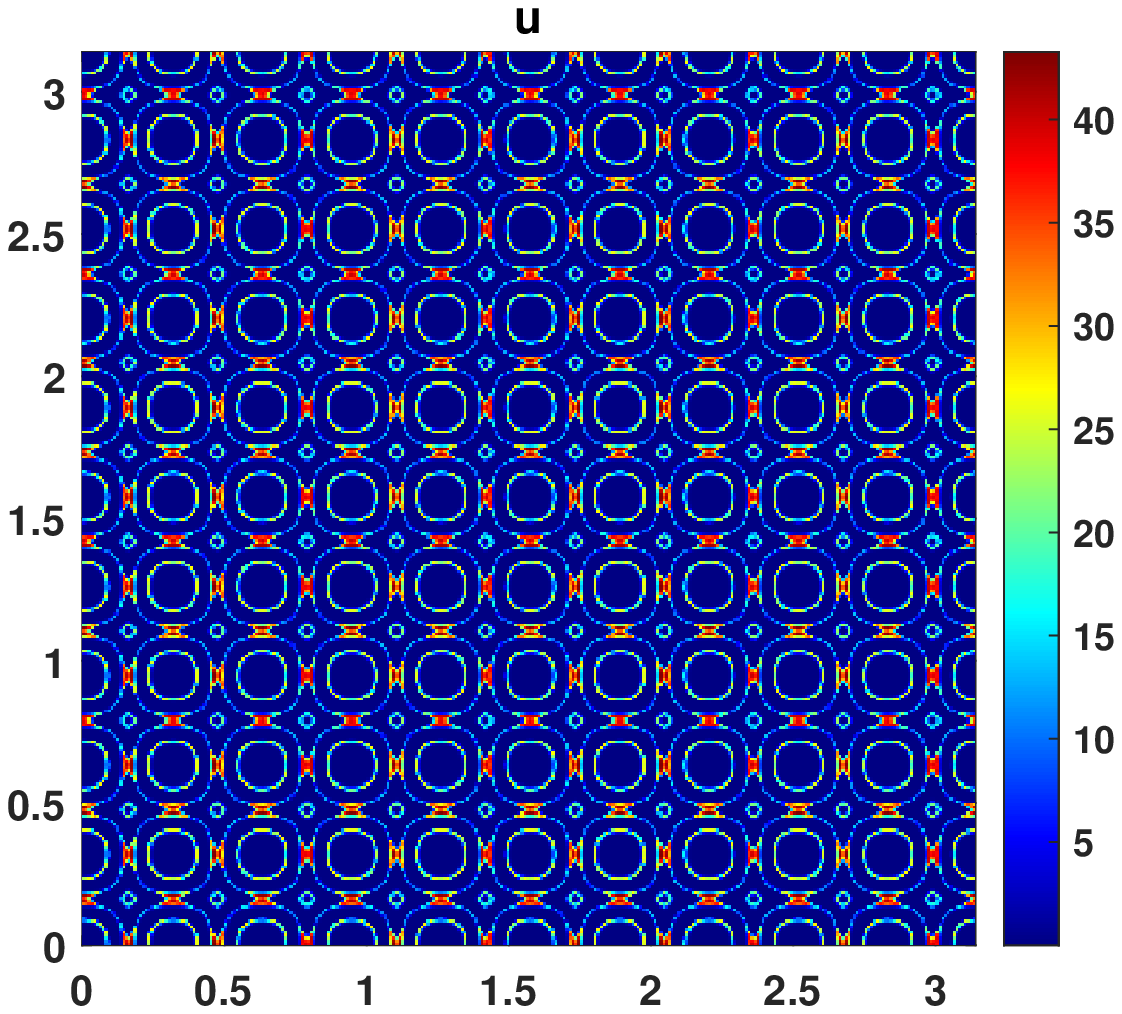} \hspace{-0.5cm} &    \includegraphics[scale=0.26]{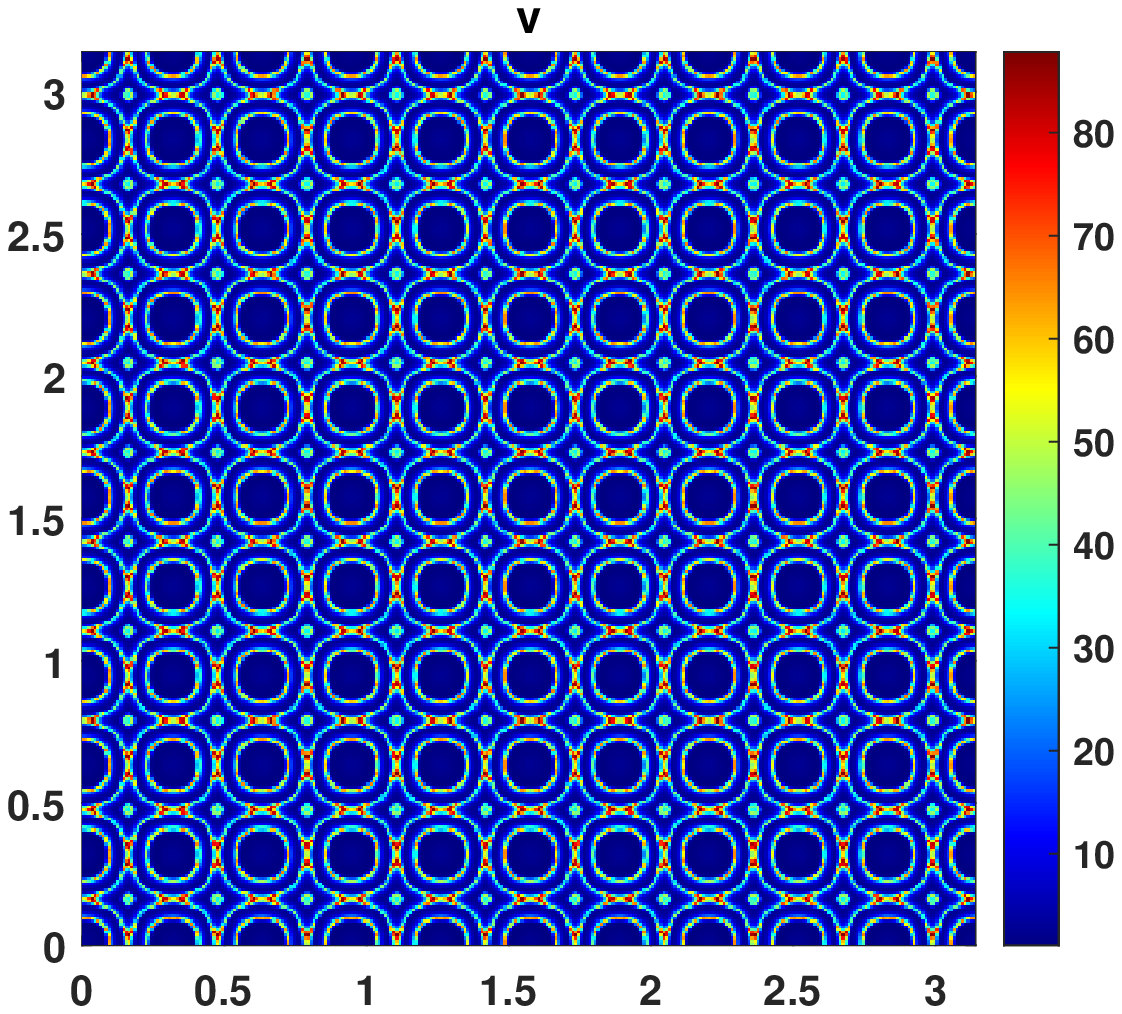} \hspace{-0.5cm} &  \includegraphics[scale=0.26]{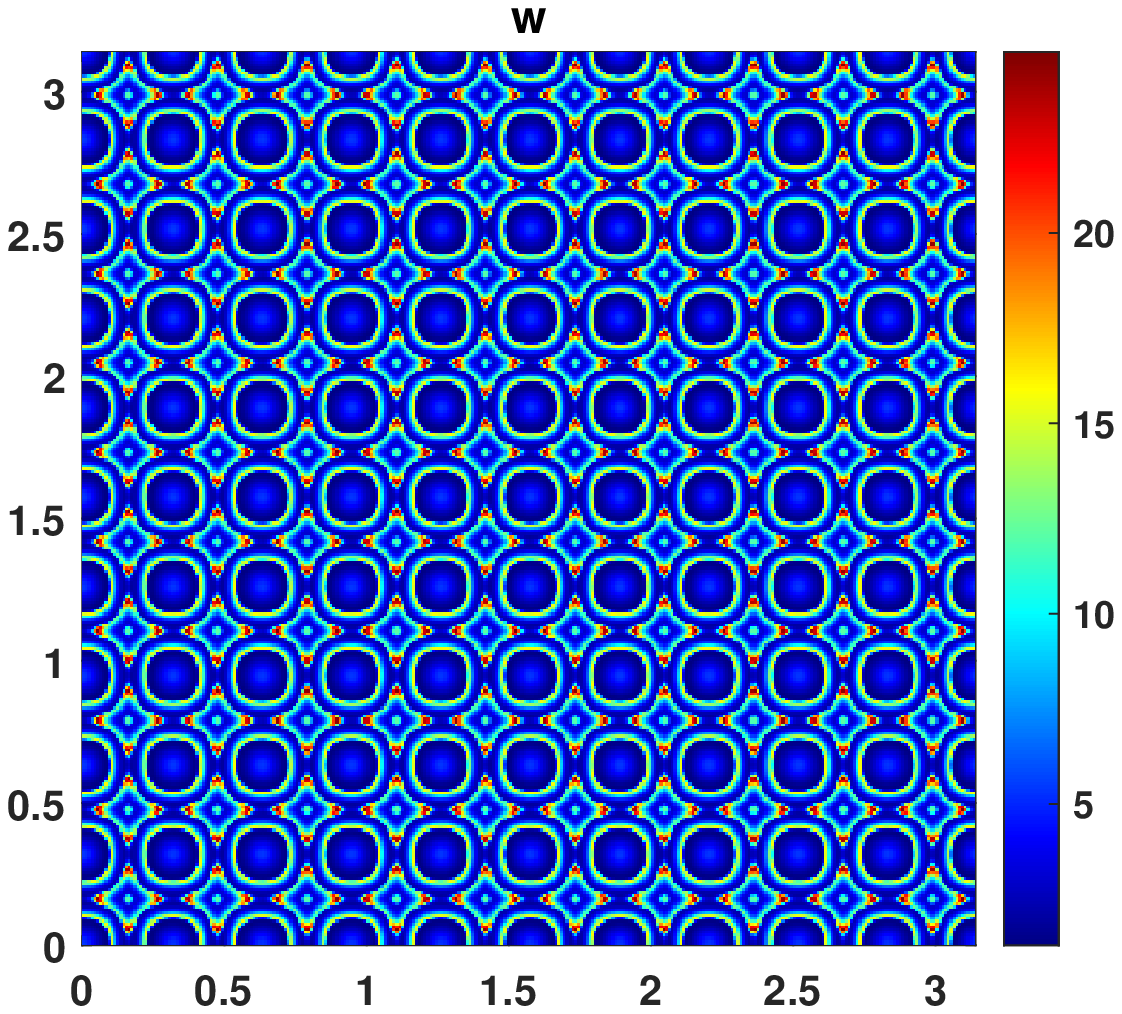}\\\vspace{-1cm}
\begin{tabular}[c]{@{}c@{}c@{}c@{}c@{}}B \\\\\\ \\\\  \\\\\\\\\end{tabular}\hspace{-0.3cm} & \includegraphics[scale=0.26]{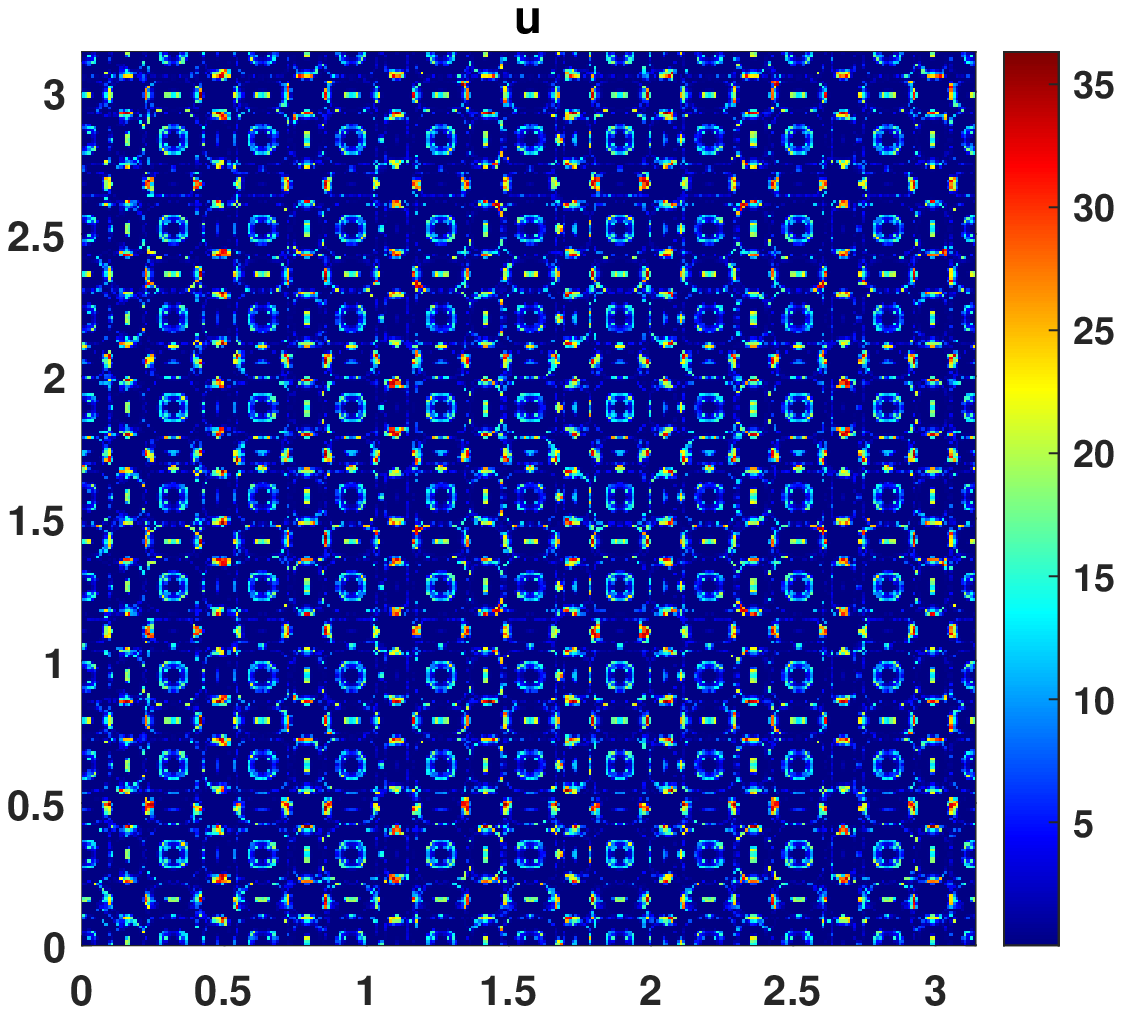}\hspace{-0.5cm}  &   \includegraphics[scale=0.26]{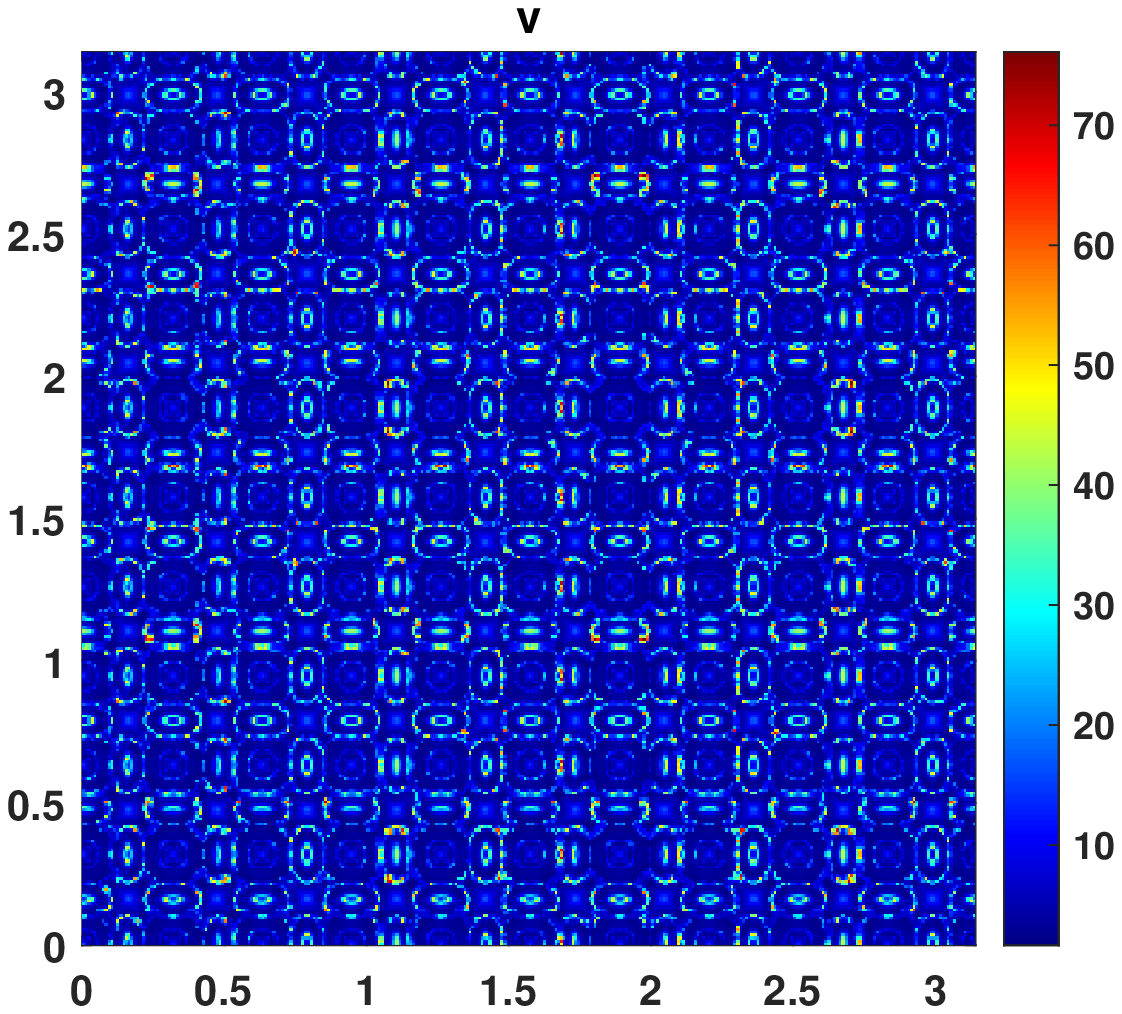} \hspace{-0.5cm} &  \includegraphics[scale=0.26]{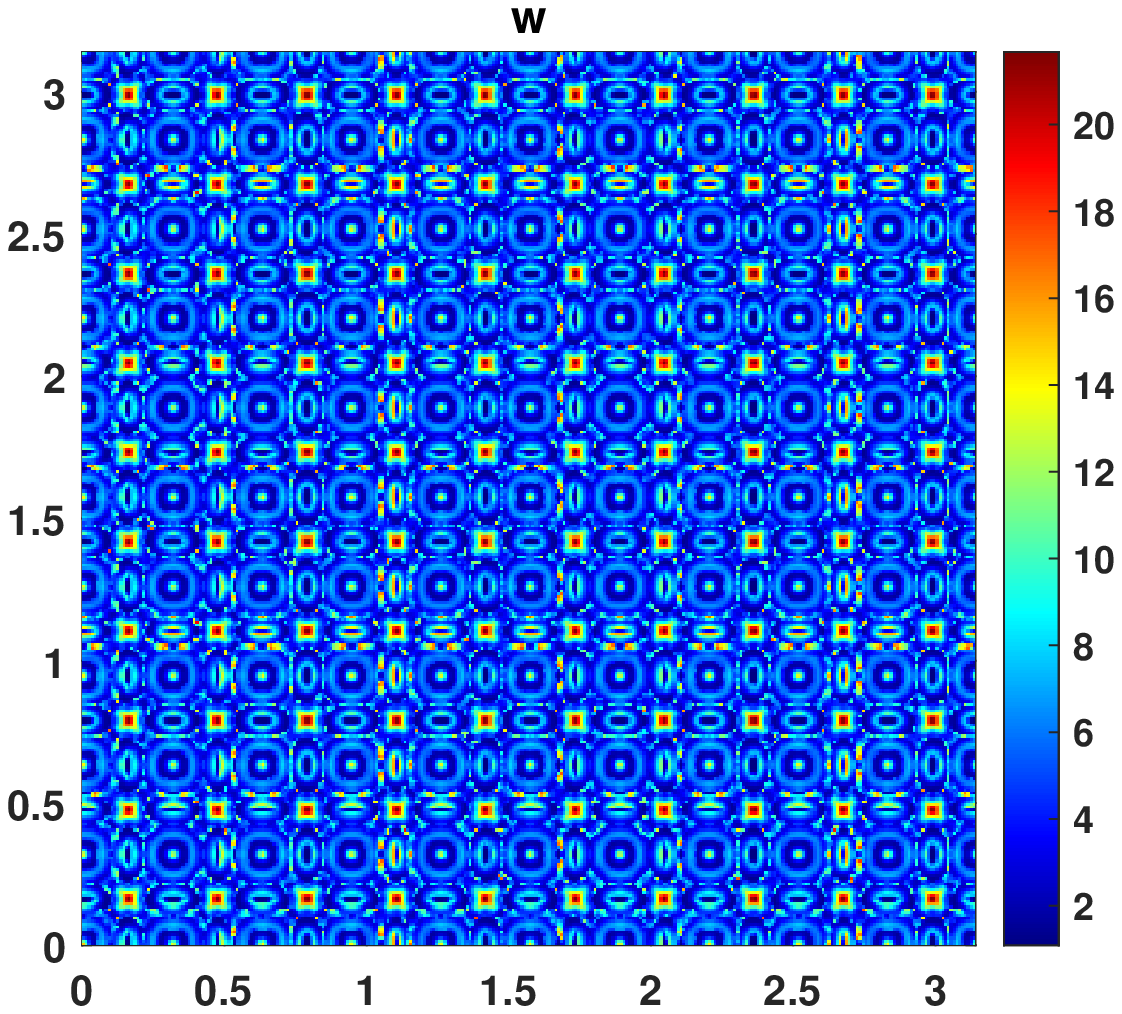}\\\vspace{-1cm} 
\begin{tabular}[c]{@{}c@{}c@{}c@{}c@{}}C \\\\\\ \\\\  \\\\\\\\\end{tabular}\hspace{-0.3cm} & \includegraphics[scale=0.26]{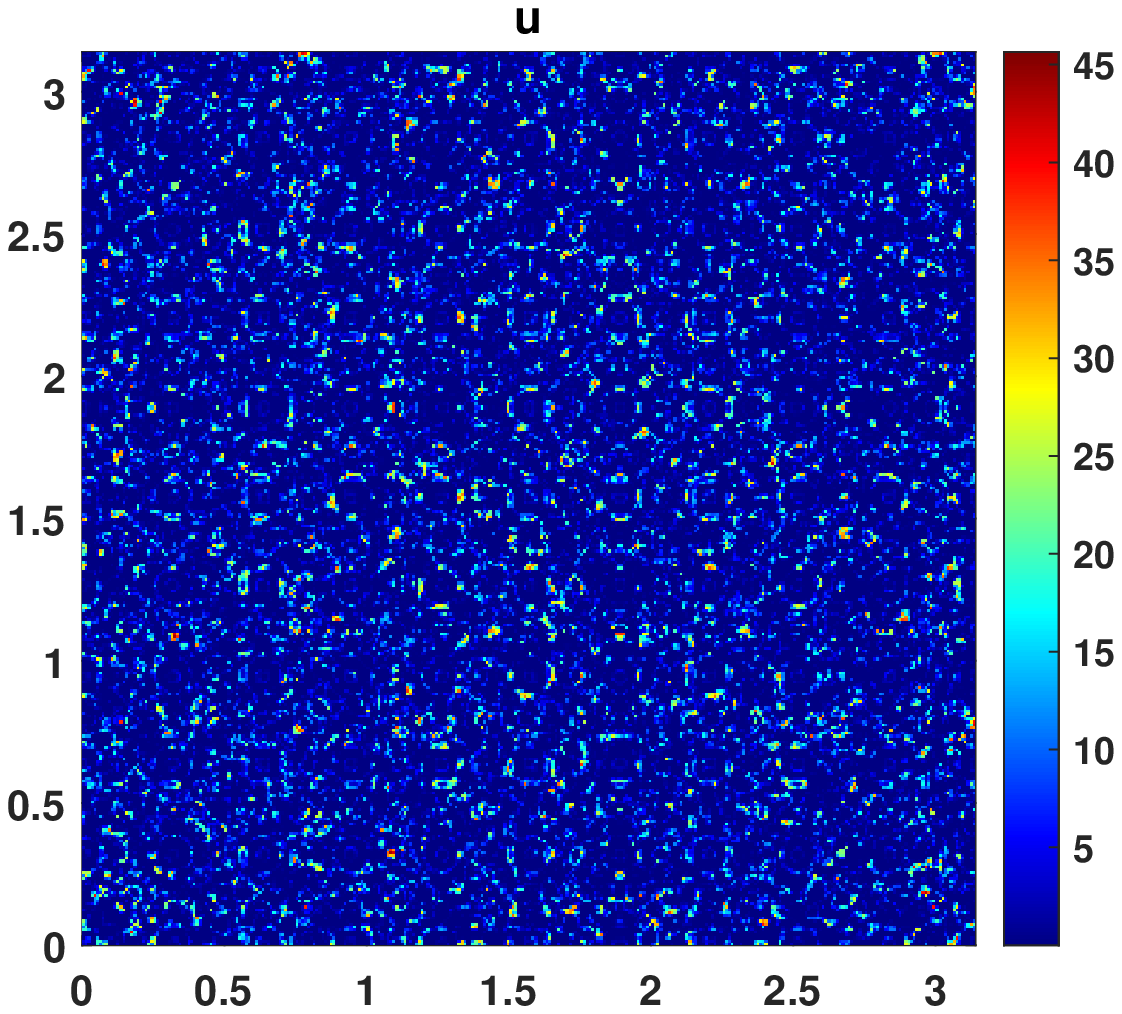}\hspace{-0.5cm}  &   \includegraphics[scale=0.26]{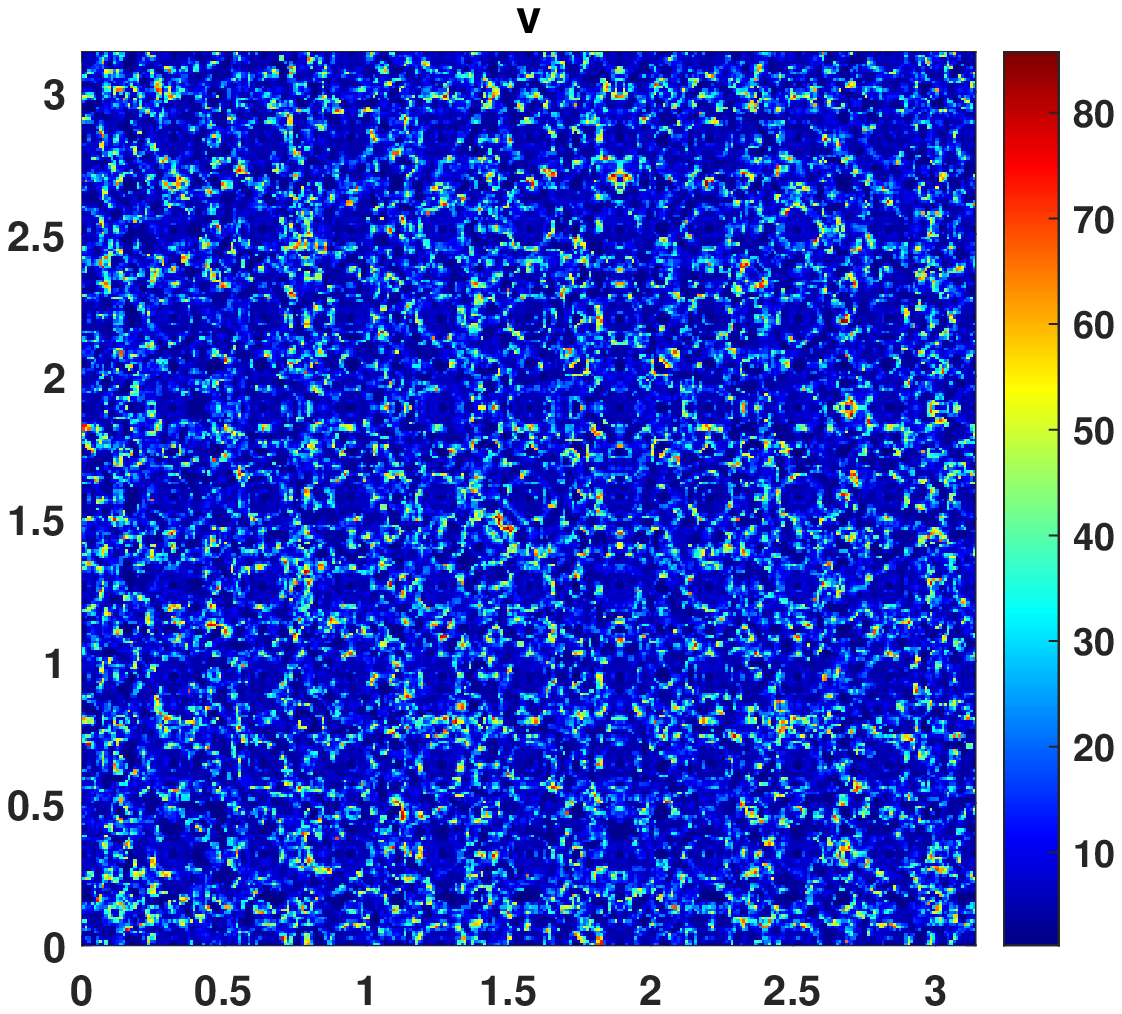} \hspace{-0.5cm} &  \includegraphics[scale=0.26]{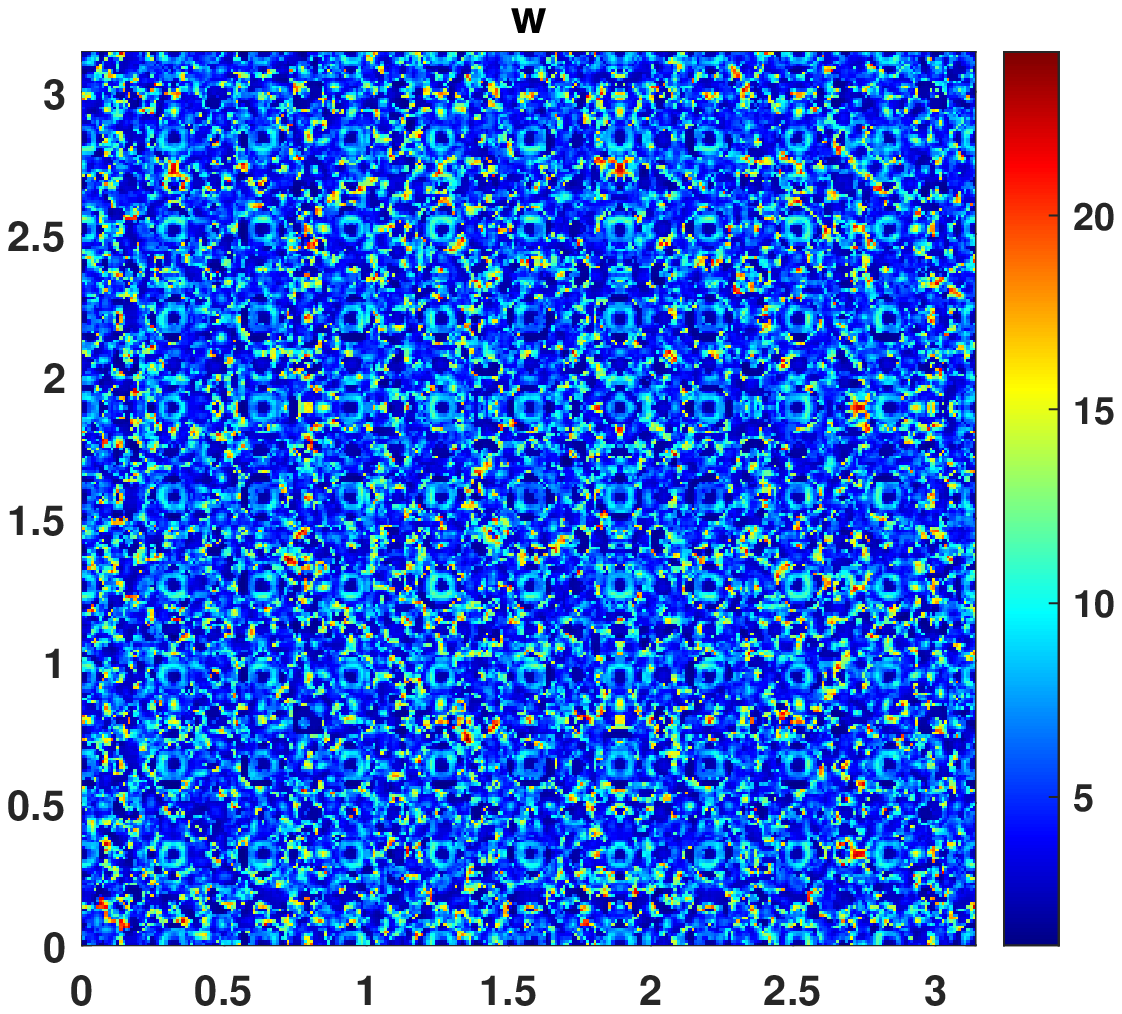} \\\vspace{-1cm}
\begin{tabular}[c]{@{}c@{}c@{}c@{}c@{}}D  \\\\\\\\ \\ \\\\\\\\\end{tabular}\hspace{-0.3cm} & \includegraphics[scale=0.26]{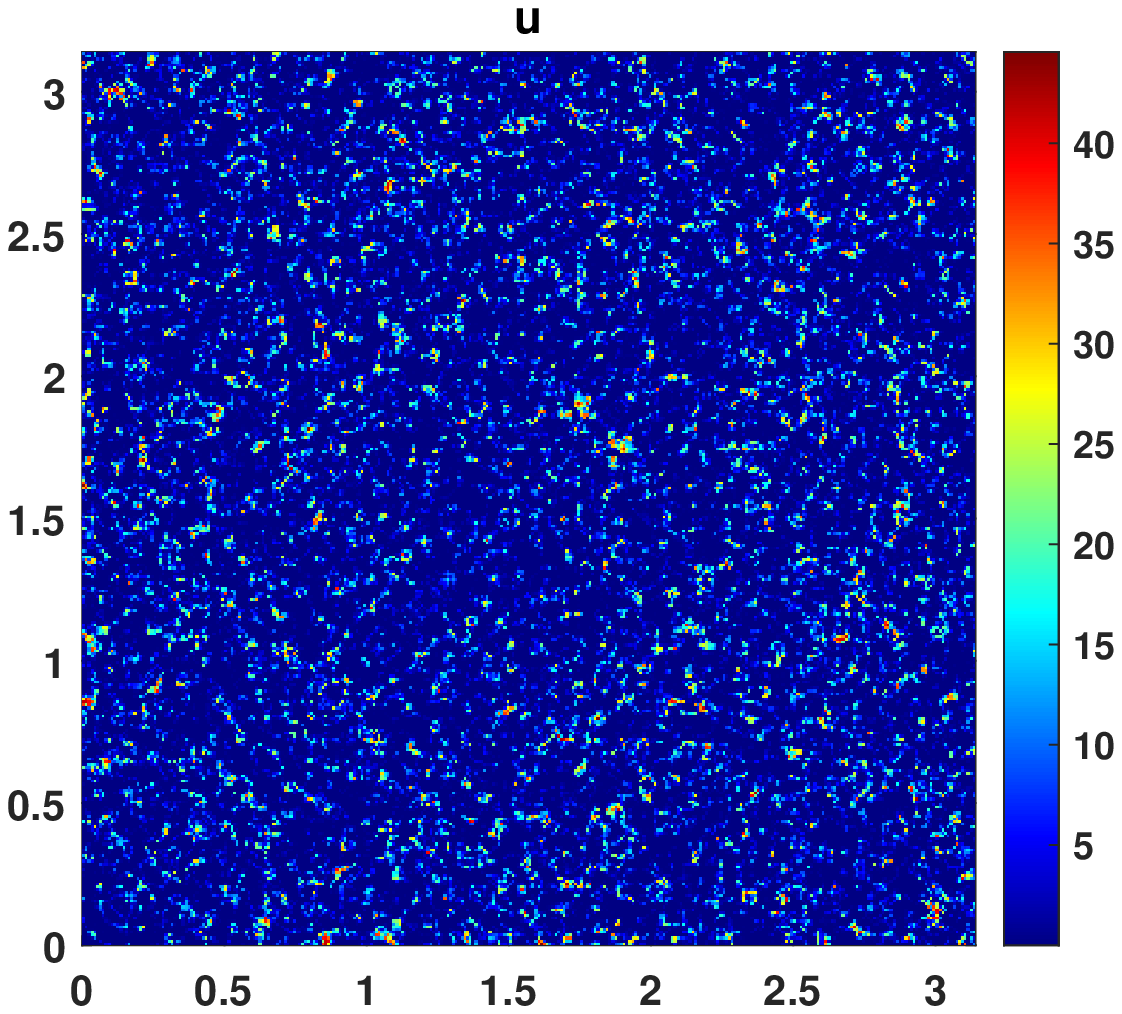}\hspace{-0.5cm}  &   \includegraphics[scale=0.26]{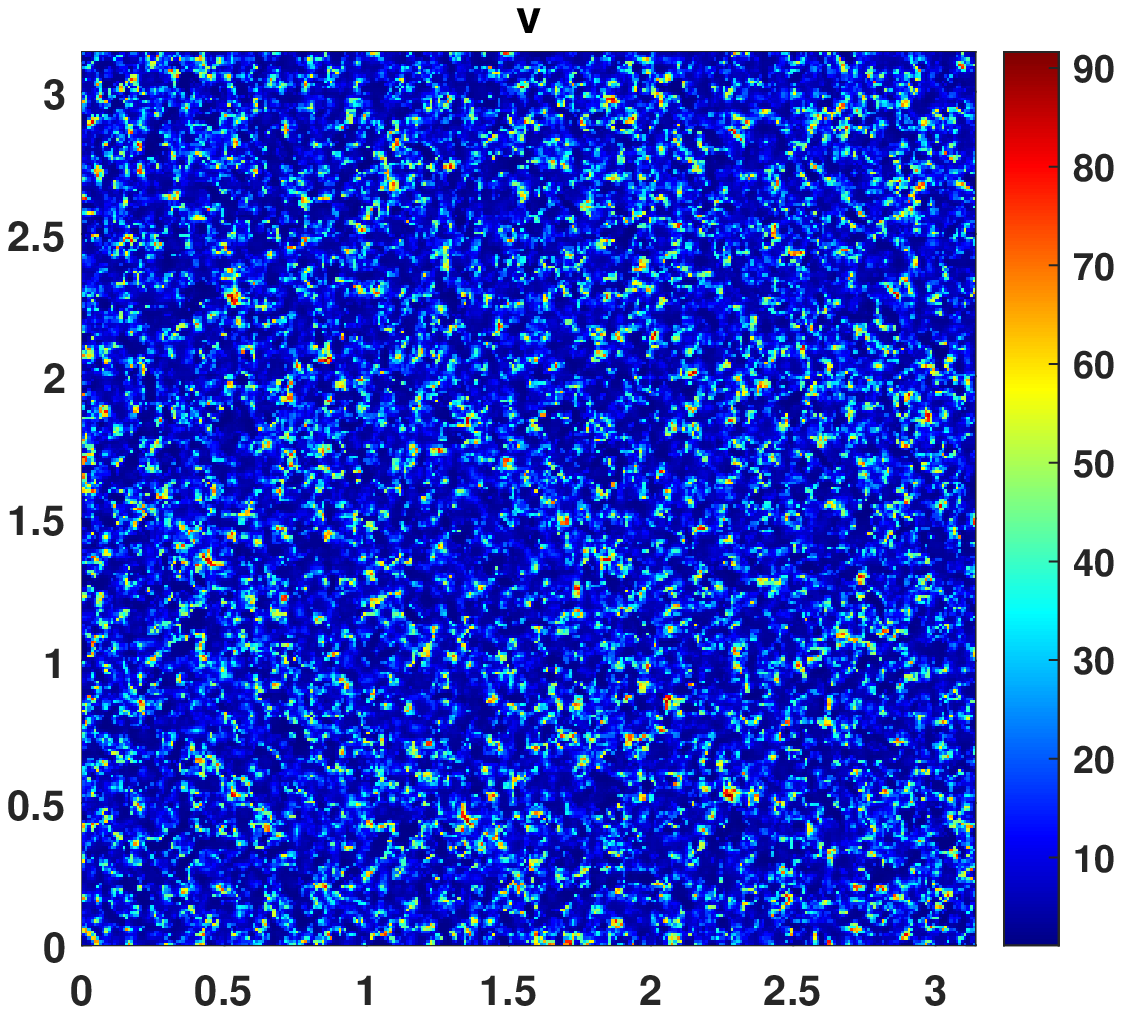} \hspace{-0.5cm} &  \includegraphics[scale=0.26]{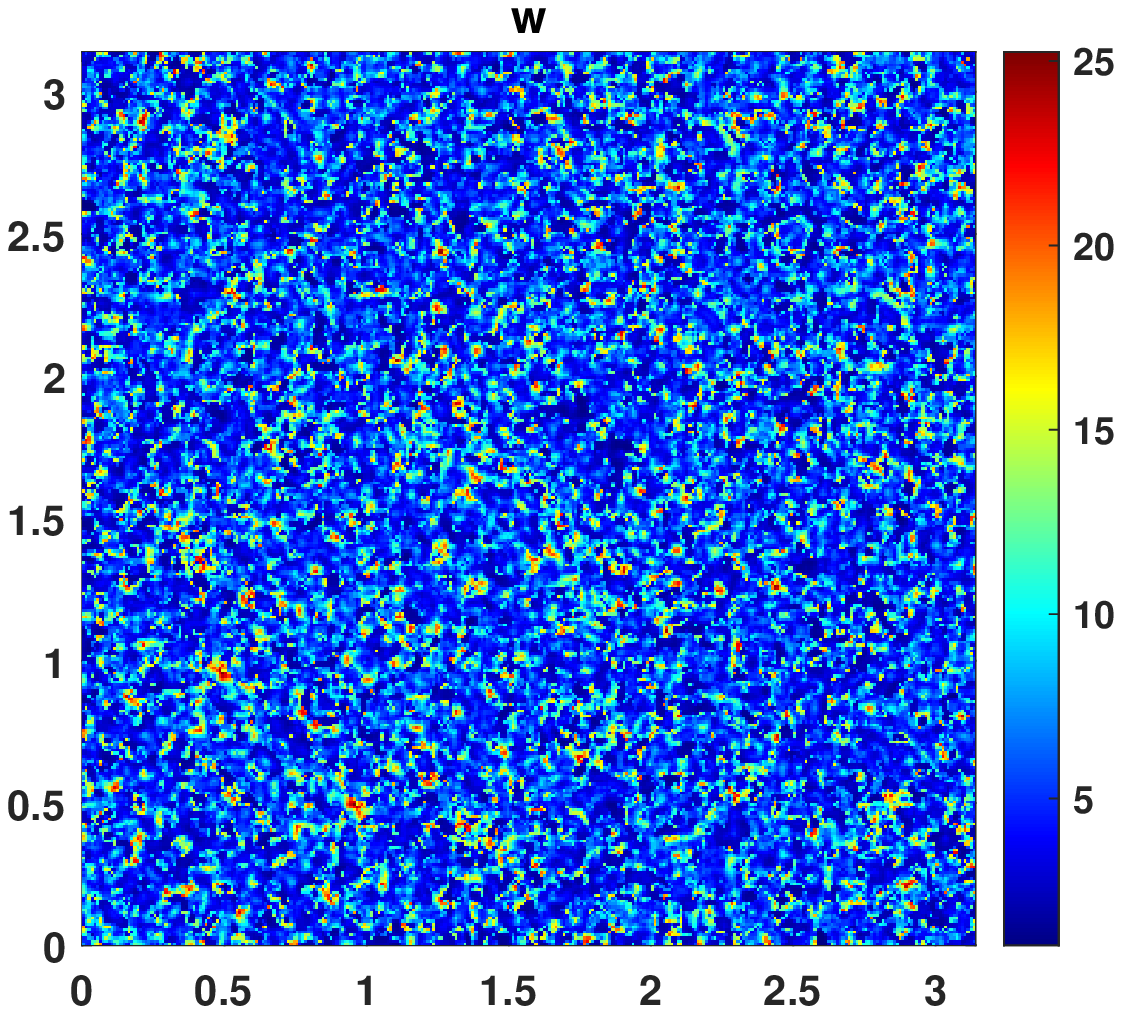}  
\end{tabular}
 \caption{Stationary non-Turing patterns seen in prey, susceptible predator and infected predator densities of the model system (\ref{eq2}), at (A) $t =500$,~(B)~ $t=1000$,~(C)~ $t=1500$,~and (D)~ $t=2000$. The diffusive constants are $d_1=10^{-10}, d_2=10^{-6}$ and $d_3=10^{-10}$. Other parameter values are given in Table \ref{paratable}. }
\label{table11}
\end{figure}

\end{document}